\documentclass[12pt]{amsart}

\usepackage{amssymb}
\usepackage{xspace}
\usepackage{enumerate}

\newtheorem{thm}{Theorem}
\newtheorem{lem}[thm]{Lemma}
\newtheorem{prop}[thm]{Proposition}
\newtheorem{cor}[thm]{Corollary}
\newtheorem{claim}[thm]{Claim}
\newtheorem{fact}[thm]{Fact}

\newtheorem{conj}[thm]{Conjecture}

\theoremstyle{definition}
\newtheorem{defn}[thm]{Definition}

\theoremstyle{remark}
\newtheorem{ex}[thm]{Example}
\newtheorem{rem}[thm]{Remark}

\newcommand{\xxx}[1] {} 

\newcommand{\Fp} {{\ensuremath{\mathbb{F}_p}}\xspace}
\newcommand{\Fq} {{\ensuremath{\mathbb{F}_q}}\xspace}
\newcommand{\Fclosed} {{\ensuremath{\overline{\,\mathbb{F}\,}}}\xspace}
\newcommand{\Fpclosed} {{\ensuremath{\Fclosed\kern-3pt_p}}\xspace}
\newcommand{\Fqclosed} {{\ensuremath{\Fclosed\kern-3pt_q}}\xspace}
\newcommand{\e} {{\ensuremath{\varepsilon}}\xspace}
\newcommand{\Prod} {{\textstyle\prod}}
\newcommand{\Frac}[2] {{\ensuremath{\textstyle\frac{#1}{#2}}}\xspace}
\newcommand{\gen} {{\rm gen}}
\newcommand{\nongen} {{\rm nongen}}

\newcommand{\Span}[1] {{\ensuremath{\langle{#1}\rangle}}\xspace}
\newcommand{\cl}[1] {{\ensuremath{\overline{#1}}}}
\newcommand{\Frob} {{\it Frob}\xspace}

\DeclareMathOperator{\pr}{pr}
\DeclareMathOperator{\diam}{diam}
\DeclareMathOperator{\inv}{inv}
\DeclareMathOperator{\mult}{mult}

\def\BC {\mathbb{C}}
\def\BF {\mathbb{F}}

\def\BP {\mathbb{P}}

\def\BZ {\mathbb{Z}}
\def\CA {\mathcal{A}}
\def\CC {\mathcal{C}}

\def\CH {\mathcal{H}}
\def\CK {\mathcal{K}}
\def\CL {\mathcal{L}}
\def\CM {\mathcal{M}}
\def\CN {\mathcal{N}}
\def\CO {\mathcal{O}}

\def\CR {\mathcal{R}}
\def\CS {\mathcal{S}}

\def\CZ {\mathcal{Z}}
\def\ua {{\underline{a}}}
\def\ug {{\underline{g}}}
\def\ux {{\underline{x}}}
\def\uz {{\underline{z}}}

\begin{document}
\title{Growth in finite simple groups of Lie type}
\author{L\'aszl\'o Pyber and Endre Szab\'o }
\date{\today
}
\thanks{L.P. is supported in part by OTKA 78439 and
  72523} \thanks{E.Sz. is supported in part by OTKA 61116 and 72523}

\begin{abstract}
  We prove that if $L$ is a finite simple group of Lie type and $A$ a
  set of generators of $L$, then $A$ grows i.e $|A^3| > |A|^{1+\e}$
  where $\e$ depends only on the Lie rank of $L$, or $A^3=L$.  This
  implies that for a family of simple groups $L$ of Lie type
  the diameter of any Cayley graph is polylogarithmic in $|L|$.
  We also obtain some new families of expanders.
  \par
  We also prove the following partial extension.  Let $G$ be a
  subgroup of $GL(n,p)$, $p$ a prime, and $S$ a symmetric set of
  generators of $G$ satisfying $|S^3|\le K|S|$ for some $K$.  Then $G$
  has two normal subgroups $H\ge P$ such that $H/P$ is soluble, $P$ is
  contained in $S^6$ and $S$ is covered by $K^c$ cosets of $H$ where
  $c$ depends on $n$.
  We obtain results of similar flavour for sets generating infinite subgroups
  of $GL(n,\mathbb F)$, $\mathbb F$ an arbitrary field.
\end{abstract}

\maketitle

\section{Introduction}

The diameter, $\diam(X)$, of an undirected graph $X=(V,E)$ is the
largest distance between two of its vertices.

Given a subset $A$ of the vertex set $V$ the expansion of $A$, $c(A)$,
is defined to be the ratio $|\sigma(A)|/|A|$ where $\sigma(A)$ is the
set of vertices at distance $1$ from $A$. A graph is a $C$-expander
for some $C>0$ if for all sets $A$ with $|A|< |V|/2$ we have $c(A)\ge
C$. A family of graphs is an expander family if all of its members are
$C$-expanders for some fixed positive constant $C$.

Let $G$ be a finite group and $S$ a symmetric (i.e. inverse-closed)
set of generators of $G$.  The Cayley graph $\Gamma(G,S)$ is the graph
whose vertices are the elements of $G$ and which has an edge from $x$
to $y$ if and only if $x=sy$ for some $s \in S$. Then the diameter of
$\Gamma$ is the smallest number $d$ such that $S^d=G$.

The following classical conjecture is due to Babai \cite{BS}

\begin{conj}[Babai]
  For every non-abelian finite simple group $L$ and every symmetric
  generating set $S$ of $L$ we have $\diam\big( \Gamma(L,S)\big) \le
  C\big(\log|L|\big)^c$ where $c$ and $C$ are absolute constants.
\end{conj}

In a spectacular breakthrough Helfgott \cite{He1} proved that the
conjecture holds for the family of groups $L=PSL(2,p)$, $p$ a prime.
In recent major work \cite{He2} he proved the conjecture for the
groups $L=PSL(3,p)$, $p$ a prime.  Dinai \cite{Di} and Varj\'u
\cite{Va} have extended Helfgott's original result to the groups
$PSL(2,q)$, $q$ a prime power.

We prove the following.
\begin{thm} \label{intro-thm} Let $L$ be a finite simple group of Lie
  type of rank $r$. For every symmetric set $S$ of generators of
  $L$ we have
  $$
  \diam\big(\Gamma(L,S)\big) < \big(\log|L|\big)^{c(r)}
  $$
  where the constant $c(r)$ depends only on $r$.
\end{thm}

This settles Babai's conjecture for any family of simple groups of Lie
type of bounded rank.

A key result of Helfgott \cite{He1} shows that generating sets of
$SL(2,p)$ grow rapidly under multiplication. His bound on diameters is
an immediate consequence.

\begin{thm}[Helfgott]
  Let $L=SL(2,p)$ and $A$ a generating set of $L$. Let $\delta$ be a
  constant, $0<\delta<1$.
  \begin{enumerate}[a)]
  \item \label{item:1} Assume that $|A|< |L|^{1-\delta}$. Then
    $$
    |A^3|\gg |A|^{1+\e}
    $$
    where $\e$ and the implied constant depend only on $\delta$
  \item \label{item:2} Assume that $|A| > |L|^{1-\delta}$. Then
    $A^k=L$ where $k$ depends only on $\delta$.
  \end{enumerate}
\end{thm}

It was observed in \cite{NP} that a result of Gowers \cite{Gow}
implies that \ref{item:2}) holds for an arbitrary simple group of Lie
type $L$ with $k=3$ for some $\delta(r)$ which depends only on the Lie
rank $r$ of $L$ (see \cite{BNP} for a more detailed discussion).
Hence to complete the proof of our theorem on diameters it remains to
prove an analogue of the (rather more difficult) part \ref{item:1}) as
was done by Helfgott for the groups $SL(3,p)$ in \cite{He2}.

We prove the following.

\begin{thm} \label{simple-L} Let $L$ be a finite simple group of Lie
  type of rank $r$ and $A$ a generating set of $L$.  Then either
  $A^3=L$ or
  $$
  |A^3|\gg |A|^{1+\e}
  $$
  where $\e$ and the implied constant depend only on $r$.
\end{thm}

We also give some examples which show that in the above result the
dependence of $\e$ on $r$ is necessary. In particular we construct
generating sets $A$ of $SL(n,3)$ of size $2^{n-1}+4$ with $|A^3|< 100|A|$
for $n\ge3$.

Theorem~\ref{simple-L} was first announced in \cite{PSz}.  The same
day similar results were announced by Breuillard, Green and Tao
\cite{BrGrTao} for finite Chevalley groups.  It is noted in
\cite{BrGrTao} that their methods are likely to extend to all simple
groups of Lie type, but this has not yet been checked.  On the other
hand in \cite{BrGrTao} various interesting results for complex matrix
groups were also announced.

Somewhat earlier Gill and Helfgott \cite{GH} had shown that small
generating sets (of size at most $p^{n+1-\delta}$ for some $\delta>0$)
in $SL(n,p)$ grow.

Helfgott's work \cite{He1} has been the starting point and inspiration
of much recent work by Bourgain, Gamburd, Sarnak and others.  Let
$S=\{g_1,g_2,\dots, g_k\}$ be a symmetric subset of $SL(n,\BZ)$ and
$\Lambda=\Span S$ the subgroup generated by $S$.  Assume that
$\Lambda$ is Zariski dense in $SL(n)$.  According to the theorem of
Matthews-Vaserstein-Weisfeiler \cite{MVW} there is some integer $m_0$
such that $\pi_m(\Lambda)=SL(n,\BZ/m\BZ)$ assuming $(m,m_0)=1$.  Here
$\pi_m$ denotes reduction $\mod m$.

It was conjectured in \cite{Lu}, \cite{BGS} that the Cayley graphs
$\Gamma\big(SL(n,\BZ/m\BZ),\pi_m(S)\big)$ form an expander family,
with expansion constant bounded below by a constant $c=c(S)$.  This
was verified in \cite{BG1}, \cite{BG2}, \cite{BGS} in many cases when
$n=2$ and in \cite{BG3} for $n>2$ and moduli of the form $p^d$ where
$d\to\infty$ and $p$ is a sufficiently large prime.

In \cite{BG3} Bourgain and Gamburd also prove the following

\begin{thm}[Bourgain, Gamburd]
  \label{Bourgain-Gamburd}
  Assume that the analogue of Helfgott's theorem on growth holds for
  $SL(n,p)$, $p$ a prime. Let $S$ be a symmetric finite subset of
  $SL(n,\BZ)$ generating a subgroup $\Lambda$ which is Zariski dense
  in $SL(n)$.  Then the family of Cayley graphs
  $\Gamma(SL(n,p),\pi_p(S))$ forms an expander family as
  $p\to\infty$. The expansion coefficients are bounded below by a
  positive number $c(S)> 0$.
\end{thm}

By Theorem~\ref{simple-L} the condition of this theorem is satisfied
hence the above conjecture is proved for prime moduli.

For $n=2$ Bourgain, Gamburd and Sarnak \cite{BGS} proved that the
conjecture holds for square free moduli. This result was used in
\cite{BGS} as a building block in a combinatorial sieve method for
primes and almost primes on orbits of various subgroups of $GL(2,\BZ)$
as they act on $\BZ^m$ (for $m\ge2$).

Recently, extending Theorem~\ref{Bourgain-Gamburd}
P. Varj\'u \cite{Va} has shown that if the analogue of
Helfgott's theorem holds for $SL(n,p)$, $p$ a prime, then the
above conjecture holds for square free moduli
and Zariski dense subgroups of $SL(n)$.
Hence our results constitute
a major step towards obtaining a generalisation to Zariski dense
subgroups of $SL(n,\BZ)$ and to other arithmetic 
groups.\footnote{Finally the conjecture has very recently been proved by
  Bourgain and Varj\'u \cite{BV}.}

Simple groups of Lie type can be treated as subgroups of simple
algebraic groups.  In fact, instead of concentrating on simple groups,
we work in the framework of arbitrary linear algebraic groups over
algebraically closed fields.
We set up a machinery which can be used to obtain various results on growth of
subsets in linear groups.
In particular, we prove the following extension of
Theorem~\ref{simple-L}, valid for finite groups obtained from
connected linear groups over \Fpclosed, which produces growth within certain
normal subgroups (for the terminology see
Definition~\ref{Frobenius-def}).

\newcommand{\mainThm} {{\rm main}}
\begin{thm}
  \label{pre-main-thm}
  Let $G$ be a connected linear algebraic group over \Fpclosed and
  $\sigma:G\to G$ a Frobenius map.  Let $G^\sigma$ denote the subgroup
  of the fixpoints of $\sigma$ and $1\in S\subseteq G^\sigma$ a
  symmetric generating set.  Then for all $1>\e>0$ there is an integer
  $M=M_\mainThm\big(\dim(G),\e\big)$ and a real $K$ depending on $\e$
  and the numerical invariants of $G$ (notably $\dim(G)$, $\deg(G)$,
  $\mult(G)$ and $\inv(G)$, see Definition~\ref{G-notations}) with the
  following property.  If $\CZ(G)$ is finite and
  $$
  K\le|S|\le |G^\sigma|^{1-\e}
  $$
  then there is a connected closed normal subgroup $H\triangleleft G$
  such that $\deg H\le K$, $\dim(H)>0$ and
  $$
  |S^M\cap H|\ge|S|^{(1+\delta)\dim(H)/\dim(G)}
  $$
  where $\delta=\Frac{\e}{128\dim(G)^3}$.
\end{thm}

Consider the groups $G^\sigma$ for simply connected simple algebraic
groups $G$. Central extensions of all but finitely many simple groups
of Lie type are obtained in this way (see \cite{St}) and the centres
$\CZ(G^\sigma)$ have bounded order. Hence Theorem~\ref{pre-main-thm}
implies Theorem~\ref{simple-L} for both twisted and untwisted simple
groups of Lie type in a unified way.

The proof of Theorem~\ref{pre-main-thm} relies basically on two
properties of the finite groups $G^\sigma$. First, if $G^\sigma$ is
large enough then $\CC_G(G^\sigma)=\CZ(G)$. Second, if a $\sigma$-invariant
connected closed subgroup of $G$ is normalised by $G^\sigma$ then it
is in fact normal in $G$.  In this generality
Theorem~\ref{pre-main-thm} depends on Hrushovski's twisted Lang-Weil
estimates \cite{Hr}.  In the proof of Theorem~\ref{simple-L} this can
be avoided (see Remark~\ref{do-we-need-Hrushovski}).
Hence the constants in this theorem are explicitly computable.

We believe that Theorem~\ref{pre-main-thm} and the general results
concerning algebraic groups involved in its proof will have
many applications to investigating growth in linear
groups. Here we first prove (using Theorem~\ref{pre-main-thm})
the following partial extension of Theorem~\ref{simple-L}:

\begin{thm} \label{pre-partial} Let $S$ be a symmetric subset of
  $GL(n,p)$ satisfying $|S^3|\le K|S|$ for some $K\ge1$.  Then
  $GL(n,p)$ has two subgroups $H\ge P$, both normalised by $S$, such
  that $P$ is perfect, $H/P$ is soluble, $P$ is contained in $S^6$ and
  $S$ is covered by $K^{c(n)}$ cosets of $H$ where $c(n)$ depends on
  $n$.
\end{thm}

Understanding the structure of symmetric subsets $S$ of $GL(n,p)$
(or more generally of $GL(n,q)$, $q$ a prime-power)
satisfying $|S|^3\le K|S|$ is
mentioned by Breuillard, Green and Tao as a difficult open problem in
\cite{BrGrTao}.

Subgroups of $GL(n,p)$ generated by elements of order $p$
were investigated in detail by Nori~\cite{No} and
Hrushovski-Pillay~\cite{HP}.
As a byproduct of the proof of Theorem~\ref{pre-partial} we obtain the
following.

\begin{thm} \label{pre-diameter} Let $P\le GL(n,p)$, $p$ a prime, be a
  perfect subgroup which is generated by its elements of order
  $p$. Let $S$ be a symmetric set of generators of $P$.  Then
  $$
  \diam\big(\Gamma(P,S)\big) \le \big(\log|P|\big)^{M(n)}
  $$
  where the constant $M(n)$ depends only on $n$.
\end{thm}

Theorem~\ref{pre-diameter} is a surprising extension of the fact
(included in Theorem~\ref{intro-thm}) that simple subgroups of
$GL(n,p)$ ($n$ bounded) have polylogarithmic diameter.

Combining Theorem~\ref{pre-diameter} with results of Aldous \cite{Al}
and Babai \cite{Ba2} we immediately obtain the following corollary.

\begin{cor}
  Let $\Gamma=\Gamma(P,S)$ be a Cayley graph as in
  Theorem~\ref{pre-diameter}.  Then $\Gamma$ is a $C$-expander with
  some
  $$
  C\ge\frac1{1+\big(\log|P|\big)^{M(n)}} \;.
  $$
  Equivalently, if $A$ is a subset of $P$ of size at most $|P|/2$,
  then we have
  $$
  |A\cdot S|\ge(1+C)|A| \;.
  $$
\end{cor}

For a very recent unexpected application in arithmetic geometry
of the above corollary see \cite{EHK}.

To indicate the generality of our methods we derive the following
consequence.

\begin{thm}
  \label{pre-tyukszem}
  Let $\BF$ be an arbitrary field
  and $S\subseteq GL(n,\BF)$ a finite symmetric subset
  such that
  $\big|S^3\big|\le K|S|$ for some $K\ge\Frac32$.
  Then there are normal subgroups $H\le\Gamma$ of \Span{S}
  and a bound $m$ depending only on $n$
  such that $\Gamma\subseteq S^6H$,
  the subset $S$
  can be covered by $K^m$ cosets of $\Gamma$,
  $H$ is soluble,
  and the quotient group $\Gamma/H$ is the product of finite simple groups of
  Lie type of the same characteristic as $\BF$.
  (In particular, in characteristic $0$ we have $\Gamma=H$.) 
  Moreover, the Lie rank of the simple factors appearing in $\Gamma/H$
  is bounded by $n$,
  and the number of factors is also at most $n$.
\end{thm}

This theorem may be viewed as a common generalisation of
Theorem~\ref{simple-L} above and a result of Hrushovski \cite{Hr2}
obtained by model-theoretic tools.
It would be most interesting to obtain a result that would also imply
Theorem~\ref{pre-partial}.

The first result of this type was obtained by
Elekes and Kir\'aly \cite{EK}.
In characteristic 0 the above theorem was first proved
by Breuillard, Green and Tao \cite{BrGrTao2}.
Actually in that case they have a stronger conclusion:
one can even require $\Gamma=H$ to be nilpotent.

In earlier versions of our paper,
for subsets of linear groups over infinite fields
we only proved general results on growth.
While writing the final version of this paper,
we realised that Theorem~\ref{pre-tyukszem} is
a relatively easy consequence of these results.

We are particularly indebted to Martin Liebeck who proved
Proposition~\ref{from-Liebeck} for us.
We also thank Nick Gill, Bob Guralnick, Gergely Harcos,
Andrei Jaikin-Zapirain, Attila Mar\'oti, Nikolay Nikolov,
Tam\'as Szamuely for various remarks on
earlier drafts of this paper.

\subsection{Methods}

The proofs of Helfgott combine group theoretic arguments with some
algebraic geometry, Lie theory and tools from additive combinatorics
such as the sum-product theorem of Bourgain, Katz, Tao \cite{BKT}.
Our argument relies on a deeper understanding of the algebraic group
theory behind his proofs and an extra trick, but not on additive
combinatorics.

We prove various results which say that if $L$ is a ``nice'' subgroup
of an algebraic group $G$ generated by a set $A$ then $A$ grows in
some sense.  These were motivated by earlier results of Helfgott
\cite{He1}, \cite{He2} and Hrushovski-Pillay \cite{HP}.

To illustrate our strategy we outline the proof of
Theorem~\ref{simple-L} in the simplest case, when $A$ generates
$L=SL(n,q)$, $q$ a prime-power.  Assume that ``$A$ does not grow''
i.e. $|AAA|$ is not much larger than $|A|$.  Using an ``escape from
subvarieties'' argument it is shown in \cite{He2} that if $T$ is a
maximal torus in $L$ then $|T \cap A|$ is not much larger than
$|A|^{1/(n+1)}$ .  This is natural to expect for dimensional reasons
since $\dim(T)/\dim(L)=(n-1)/(n^2-1)=1/(n+1)$.

We use a rather more powerful escape argument.  The first part of our
paper is devoted to establishing the necessary tools in great
generality (in particular Theorem~\ref{spread-all-over}).

Now $T$ is equal to $L\cap\bar T$ where $\bar T$ is a maximal torus of
the algebraic group $SL(n,\Fqclosed)$.  Let $T_r$ denote the set of
regular semisimple elements in $T$.  Note that $T\setminus T_r$ is
contained in a subvariety $V\subsetneq\bar T$ of dimension $n-2$.  By
the above mentioned escape argument $\big|(T\setminus T_r) \cap
A\big|$ is not much larger than
$$
|A|^{\dim(V)/\dim(L)} = |A|^{1/(n+1) - 1/(n^2 -1)} \;.
$$

By \cite{He2} or by our escape argument $A$ does contain regular
semisimple elements.  If $a$ is such an element then consider the map
$SL(n)\to SL(n)$, $g\to g^{-1}ag$.  The image of this map is contained
in a subvariety of dimension $n^2-1-(n-1)$ since
$\dim\big(\CC_{SL(n)}(a)\big)=n-1$.  By the escape argument we obtain
that for the conjugacy class ${\rm cl}(a)$ of $a$ in $L$, $\big|{\rm
  cl}(a) \cap A^{-1}aA\big|$ is not much larger than
$|A|^{(n^2-n)/(n^2-1)}$. Now $\big|{\rm cl}(a) \cap A^{-1}aA\big|$ is
at least the number of cosets of the centraliser $C_L(a)$ which
contain elements of $A$ . It follows that $\big|AA^{-1} \cap
C_L(a)\big|$ is not much smaller than $|A|^{1/(n+1)}$.  Of course
$C_L(a)$ is just the (unique) maximal torus containing $a$.

Let us say that $A$ \emph{covers} a maximal torus $T$ if
$\big|T \cap A\big|$
contains a regular semisimple element.  We obtain the following
fundamental dichotomy (see Lemma~\ref{dichotomy-lemma}):

\emph{ Assume that a generating set A does not grow
  \begin{enumerate}[i)]
  \item If $A$ does not cover a maximal torus $T$ then $\big|T \cap
    A\big|$ is not much larger than $|A|^{1/(n+1) -1/(n^2-1)}$.
  \item If $A$ covers $T$ then $\big|T \cap AA^{-1}\big|$ is not much
    smaller than $|A|^{1/(n+1)}$. In this latter case in fact
    $\big|T_r\cap AA^{-1}\big|$ is not much smaller than
    $|A|^{1/(n+1)}$.
  \end{enumerate}
}

It is well known that if $A$ doesn't grow then $B=AA^{-1}$ doesn't
grow either hence the above dichotomy applies to $B$.

Let us first assume that $B$ covers a maximal torus $T$ but does not
cover a conjugate $T'=g^{-1}Tg$ of $T$ for some element $g$ of $L$ .
Since $A$ generates $L$ we have such a pair of conjugate tori where
$g$ is in fact an element of $A$. Consider those cosets of $T'$ which
intersect $A$. Each of the, say, $t$ cosets contains at most $|B \cap
T'|$ elements of $A$ i.e. not much more than $|B|^{1/(n+1)
  -1/(n^2-1)}$ which in turn is not much more than
$|A|^{1/(n+1)-1/(n^2-1)}$.  Therefore $|A|$ is not much larger than
$t|A|^{1/(n+1) -1/(n^2 -1)}$.

On the other hand $A\big(A^{-1}(BB^{-1})A\big)$ has at least
$t\big|T\cap BB^{-1}\big|$ elements which is not much smaller than
$t|A|^{1/(n+1)}$.  Therefore $A\big(A^{-1}(AA^{-2}A)A\big)$ is not
much smaller than $|A|^{1+1/(n^2-1)}$ which contradicts the assumption
that $A$ does not grow.

We obtain that $B$ covers all conjugates of some maximal torus
$T$. Now the conjugates of the set $T_r$ are pairwise disjoint
(e.g. since two regular semisimple elements commute exactly if they
are in the same maximal torus).  The number of these tori is
$|L:N_L(T)| > c(n)|L:T|$ for some constant which depends only on $n$.
Each of them contains not much less than $|B|^{1/(n+1)}$ regular
semisimple elements of $BB^{-1}$.
Altogether we see that $|A|$ is not much
smaller than $q^{n^2-n}|A|^{1/(n+1)}$ and finally that $|A|$ is not
much less than $|L|$. In this case by \cite{NP} we have $AAA=L$.

The proof of Theorem~\ref{pre-main-thm} follows a similar strategy.
However there is an essential difference;
maximal tori have to be replaced
by a more general class of subgroups called CCC-subgroups (see
Definition~\ref{CCC-subgroup-def}).
These subgroups were in fact designed to make the argument work in not
necessarily simple (or semisimple) algebraic groups.
In Sections~\ref{sec:centralisers},~\ref{sec:dichotomy-lemmas} and
\ref{sec:finding-using-ccc} we establish the basic properties of these
subgroups and justify that they indeed play the role of maximal tori
in general algebraic groups.  The proof of Theorem~\ref{pre-main-thm}
is completed in Section~\ref{sec:linear-groups}.

In \cite{No} Nori showed that if $p$ is sufficiently large in terms of
$n$, there is a correspondence between subgroups of $GL(n,p)$
generated by elements of order $p$ and a certain class of closed
subgroups of $GL(n,\Fpclosed)$.
Note that the bounds in \cite{No} are ineffective.
Using this correspondence
Theorem~\ref{pre-partial} is proved for perfect $p$-generated groups
by a short induction argument based on a slight extension of
Theorem~\ref{pre-main-thm}.  The general case can be reduced to this
by applying various known results on finite linear groups.

Theorem~\ref{pre-tyukszem} follows by combining some of the ingredients of the
proof of Theorem~\ref{pre-partial} in a rather more direct way.

Examples given in Section~\ref{sec:examples}
show that in Theorem~\ref{simple-L}
we must have $\e(r)=O(1/r)$.  We believe that this is the right
order of magnitude.

\section{notation}

Throughout this paper $\Fclosed$ denotes an arbitrary algebraically
closed field. For a prime number $p$ we denote by $\Fp$ and
$\Fpclosed$ the finite field with $p$ elements and its algebraic
closure.  Similarly, $\Fq$ denotes the finite field with $q$ elements,
where $q$ is a prime power.  The letters $N$ and $\Delta$ will always
be used for an upper bound for dimensions and degrees respectively,
$K$ is used for a lower bound on the size of certain finite sets.
When we study growth, $M$ will denote the length of the products we
allow.  In several lemmas we use a parameter $\e$, it is the
error-margin we allow in the exponents when we count elements in
certain subsets.

\section{dimension and degree}
\label{sec:dimension-degree}

We use affine algebraic geometry i.e. all occurring sets will be
subsets of some affine space $\Fclosed^m$ for some integer $m>0$, and
we define all of them via $m$-variate polynomials whose coefficients
belong to \Fclosed.  Below we make this more precise.

\begin{defn} \label{closed-open} A subset $Z\subseteq\Fclosed^m$ is
  \emph{Zariski closed}, or simply \emph{closed}, if it can be defined
  as the common zero set of some $m$-variate polynomials.  This
  defines a topology on $\Fclosed^m$, each subset of $\Fclosed^m$
  inherits this topology, called the \emph{Zariski topology}.  This is
  the only topology that we use in this paper, so we omit the
  adjective Zariski.  The complements of closed subsets are called
  \emph{open}, The intersection of a closed and an open subset is
  called \emph{locally closed}.  If we do not use explicitly the
  ambient affine space then locally closed subsets are called
  \emph{algebraic sets} and closed subsets are called \emph{affine
    algebraic sets}.  (Note, that our definition of algebraic set is
  rather restrictive.)
  For an arbitrary subset $X\subseteq\Fclosed^m$
  we denote by $\cl{X}$ the \emph{closure} of $X$.
\end{defn}

 \xxx{Change:}
 Note, that algebraic sets are always equipped (by definition)
 with an ambient affine space,
 even if it is not explicitely given.
 This is one reason for choosing the name ``algebraic set''
 instead of ``variety''.

\begin{defn} \label{irreducible} An algebraic set $X$ is called
  \emph{irreducible} if it has the following property.  Whenever $X$
  is contained in the union of finitely many closed subsets, it must
  be contained in one of them.
\end{defn}

\begin{defn} \label{irreducible-decomposition} Let $X$ be an algebraic
  set.  Then there are finitely many closed subsets $X_i\subseteq X$
  which are irreducible, and maximal among the irreducible closed
  subsets of $X$.  Then $X=\bigcup_iX_i$ is the \emph{irreducible
    decomposition} of $X$ and these $X_i$ are called the
  \emph{irreducible components} of $X$.
\end{defn}

\begin{defn} \label{dimension} Let $Z\subseteq\Fclosed^m$ be an
  algebraic set.  We consider chains $Z_0\subsetneq
  Z_1\subsetneq\dots\subsetneq Z_n$ where the $Z_i$ are nonempty,
  irreducible closed subsets of $Z$.  The largest possible length $n$
  of such a chain is called the \emph{dimension} of $Z$, denoted by
  $\dim(Z)$.
\end{defn}

\begin{defn} \label{degree} Let $X\subseteq\Fclosed^m$ be an algebraic
  set.  An \emph{affine} subspace of $\Fclosed^m$ is a translate of a
  linear subspace.  If $X$ is irreducible then we consider all affine
  subspaces $L\subseteq\Fclosed^m$ such that $\dim(X)+\dim(L)=m$ and
  $X\cap L$ is finite.  The \emph{degree} of $X$ is the largest
  possible number of intersection points:
  $$\deg(X) = \max_L|X\cap L| \;.$$
  In general, the degree of $X$ is defined as the sum of the degrees
  of its irreducible components.
\end{defn}

\begin{rem} \label{too-many-points} Let $X$ be an algebraic set.  Then
  $\dim(X)=0$ iff $X$ is finite.  A finite subset
  $X\subset\Fclosed^m$ is always closed, and satisfies $\deg(X)=|X|$.
\end{rem}

\begin{defn} \label{morphisms}
  \xxx{Was totally wrong, now it is corrected.}
  Let $X\subseteq\Fclosed^m$ and
  $Y\subseteq\Fclosed^{n}$ be algebraic sets.  A function $f:X\to Y$
  is called a \emph{morphism} if it is the restriction to $X$ of a map
  $\phi:\Fclosed^m\to\Fclosed^n$
  whose $n$ coordinates are $m$-variate polynomials.
  Then the graph of $f$, denoted by
  $\Gamma_f\subseteq X\times Y\subseteq\Fclosed^{m+n}$,
  is locally closed.  We define the
  \emph{degree} of $f$ to be $\deg(f)=\deg(\Gamma_f)$.
\end{defn}

\begin{rem} \label{category-of-closed-sets} Algebraic sets form a
  category with the above notion of morphism.  Isomorphic algebraic
  sets have equal dimensions and isomorphisms respect the irreducible
  decomposition.  In contrast, the degrees of isomorphic algebraic
  sets may not be be equal.
\end{rem}

In the present paper we work mainly in the category of algebraic sets
and morphisms.  To obtain explicit bounds we need to estimate the
degrees of all appearing objects. If one is satisfied with existence
results only then one can avoid all these calculations by simply
noticing that all of our constructions can be done simultaneously in
families of algebraic sets.
(Such proofs a priori do not give explicit constants, but with careful
examination, in principle they can be made explicit.) 
In fact this technique is really used e.g. in the proof of
Proposition~\ref{Chevalley-embedding}.


The following fact is standard:

\begin{fact} \label{dimension-degree} Let $X,Y\subseteq\Fclosed^m$ be
  locally closed sets.
  \begin{enumerate}[\indent(a)]
  \item \label{item:3}
    The dimension and the degree of $X$ are
    equal to the dimension and the degree of its closure $\cl{X}$.
  \item \label{item:4} Any closed subset of $X$ has dimension at most
    $\dim(X)$.
  \item \label{item:5} The irreducible components $X_i\le X$ satisfy
    $$
    \dim(X_i) \le \dim(X) = \max_j\big(\dim(X_j)\big)\;,
    $$
    $$
    \deg(X_i) \le \deg(X) = \sum_j\deg(X_j) \;.
    $$
    It follows that there are at most $\deg(X)$ components and at
    least one of them has the same dimension $\dim(X_i)=\dim(X)$.
  \item \label{item:6}
    The sets  $X\cap Y$, $\cl{X}\cup\cl{Y}$,
    $X\setminus\cl{Y}$ and $X\times Y$
    are also locally closed with the following bounds:
    \begin{eqnarray*}
      \dim(\cl{X}\cup\cl{Y}) &=&
      \max\big(\dim(X),\dim(Y)\big) \\
      \deg(\cl{X}\cup\cl{Y}) &\le& \deg(X)+\deg(Y) \\
      \dim(X\cap Y) &\le& \min\big(\dim(X),\dim(Y)\big) \\
      \deg(X\cap Y) &\le& \deg(X)\deg(Y) \\
      \dim(X\setminus\cl{Y}) &\le& \dim(X) \\
      \dim(X\times Y)    &=&   \dim(X)+\dim(Y) \\
      \deg(X\times Y)    &=&   \deg(X)\deg(Y)
    \end{eqnarray*}
    Note that we cannot estimate $\deg(X\setminus\cl{Y})$ in this
    generality.
  \item \label{item:7} Suppose that $X$ is irreducible.  Then each
    nonempty open subset $U\subset X$ is dense in $X$ with
    $\dim(X\setminus U)<\dim(X)$ (and we do not bound the degree of
    $X\setminus U$).
  \item \label{item:8} The direct product of irreducible algebraic
    sets is again irreducible.
  \item \label{item:9}
    \xxx{This is corrected!!!!}
    If $X$ is the common zero locus of degree $d$
    polynomials, then it is the common zero locus of at most
    $(d+1)^{m}$ of them, and $\deg(X)\le d^{m}$.
    On the other hand,
    a closed set $X$ is the common zero locus of 
    polynomials of degree at most $\deg(X)$.
  \end{enumerate}
\end{fact}

Most of this Fact is proved in \cite[Chapters~I.1~and~II.3]{Ha}.
The bound on $\deg(X\cap Y)$ is (an appropriate version of) B\'ezout's theorem
(see \cite{Fu})
and \eqref{item:9} follows from
\cite[Section~I.3]{Ko}.

We also need the following:

\begin{fact} \label{closed-set-constructions} Let $X$ and $Y$ be
  affine algebraic sets and $f:X\to Y$ a morphism.  We define several
  (open, closed or locally closed) subsets of $X$ and $Y$.
  Their dimension is at most $\dim(X)$,
  and we bound their degrees from above.
  We define the function $\Phi(d)=(d+2)^{(d+1)^{\dim(X)+\deg(f)}2^d}$
  \xxx{Changed to smaller!\\}
  and the constant
  $D=\Phi\big(\Phi\big(\dots\Phi\big(\deg(f)\big)\big)\dots\big)
  ^{\dim(X)+\deg(f)}$
  where the function $\Phi$ is iterated $\dim(X)+\deg(f)-1$ times.
  \begin{enumerate}[\indent(a)]
  \item \label{item:10}
    There is a partition of $\cl{f(X)}$ into at most $D$
    locally closed subsets $Y_i$ of degree at most $D$
    such that
    the closure of each $Y_i$ is the union of partition classes and
    either $f^{-1}(Y_i)=\emptyset$ or
    $\dim\big(f^{-1}(y)\big)=\dim(X)-\dim(Y_i)$ for all $y\in Y_i$.
  \item \label{item:11} We have $\deg\big(\cl{f(X)}\big)\le\deg(f)$.
    The image $f(X)$ contains a dense open subset of $\cl{f(X)}$.  If
    $X$ is irreducible then so is $\cl{f(X)}$.
  \item \label{item:12} For each $y\in f(X)$ the fibre
    $f^{-1}(y)\subseteq X$ is closed with
    $\deg\big(f^{-1}(y)\big)\le\deg(f)$.  For each closed set
    $T\subseteq Y$ the subset $f^{-1}(T)$ is also closed and its
    degree is at most $\deg(T)\deg(f)$.
  \item \label{item:13} The degree of the closed complement
    $\cl{\cl{f(X)}\setminus f(X)}$ is at most $D^2$.
  \item \label{item:14} Suppose that $X$ is irreducible.  For each
    $t\in X$ we have
    $$
    \dim\Big(f^{-1}\big(f(t)\big)\Big) \ge \dim(X) -
    \dim\big(\cl{f(X)}\big) \;.
    $$
    Those $t\in X$ where equality holds form an open dense subset
    $X_{\min}\subseteq X$ and
    $\deg\big(X\setminus X_{\min}\big)\le D^2\deg(f)$.
  \item \label{item:15} Let $S\subseteq X$ be a closed subset that is
    the intersection of $X$ and a closed set of degree $d$.
    Then the degree of
    the restricted morphism $f\big|_S$ is at most $d\cdot\deg(f)$,
    hence $\deg\big(\cl{f(S)}\big)\le d\cdot\deg(f)$ (see
    \eqref{item:11}).  If $S$ is an irreducible component of $X$ then
    there are better bounds: $\deg\big(f\big|_S\big)\le\deg(f)$ and
    $\deg\big(\cl{f(S)}\big)\le\deg(f)$.
  \end{enumerate}
\end{fact}

Parts \eqref{item:11}, \eqref{item:12} and \eqref{item:15} as well as
the fact that $X_{\min}$ of \eqref{item:14} is open and dense
follows easily using \cite[Chapters~I.1~and~II.3]{Ha}
and Fact~\ref{dimension-degree}.
Moreover, the closed complement considered in
\eqref{item:13}
is the union of a number of the locally closed subsets of \eqref{item:10},
hence its degree bound follows immediately from \eqref{item:10}.
Similarly, the subset discussed in \eqref{item:14}
is the inverse image of the union of a number of the locally closed subsets of
\eqref{item:10}, 
hence its degree is bounded by \eqref{item:10} and \eqref{item:12}.
So the only thing that remains to be proved is \eqref{item:10}.

\begin{proof}[Sketch of the proof of \eqref{item:10}]
  Let $\Fclosed^m\supseteq X$ and $\Fclosed^n\supseteq Y$ be the
  ambient affine spaces,
  $\Gamma_f\subseteq\Fclosed^{m}\times\Fclosed^{n}$ the graph of $f$,
  and $\pi:\Fclosed^{m}\times\Fclosed^n\to\Fclosed^n$
  the linear projection to the second factor.
  Then $\Gamma_f$ is isomorphic to $X$,
  hence it is enough to find an analogous partition of
  $\cl{\pi(\Gamma_f)}=\cl{f(X)}$
  with respect to
  $\pi$ and $\Gamma_f$ ( with the same bound $D$ defined in terms of
  $\deg(f)$ and $\dim(X)$).

  Let $L$ denote the linear span of $\Gamma_f$ and
  set $\tilde\pi=\pi\big|_L$.
  In general, for each variety $V$ of degree at least $2$,
  \cite[Ex.I.7.7]{Ha} constructs a cone containing $V$
  whose dimension is $\dim(V)+1$, and whose degree is strictly smaller
  that $\deg(V)$. By iterating this cone-construction we arrive,
  in at most $\dim(V)-1$ steps, at a variety of degree $1$.
  By \cite[Ex.I.7.6]{Ha} this iterated cone is a linear subspace,
  i.e. the original $V$ is contained in a linear subspace of dimension
  at most $\dim(V)+\deg(V)-1$.
  In particular, we have
  $\dim(L)\le\dim(\Gamma_f)+\deg(\Gamma_f)-1=\dim(X)+\deg(f)-1$. 
  We need to find a partition of
  $\cl{\tilde\pi(\Gamma_f)}=\cl{f(X)}$
  as in \eqref{item:10}
  with respect to
  $\tilde\pi$ and $\Gamma_f$ (with the same bound $D$).
  We factor $\tilde\pi$ into
  $\dim(L)-\dim\big(\tilde\pi(L)\big)\le\dim(X)+\deg(f)-1$
  consecutive linear projections $\tilde\pi_j$,
  each with one-dimensional fibres.
  Our strategy is the following.
  First we partition $\cl{\tilde\pi_1(\Gamma_f)}$
  via the next Claim~\ref{Euclids-algorithm}.
  Then for each partition class
  $C\subseteq\cl{\tilde\pi_1(\Gamma_f)}$
  we apply again Claim~\ref{Euclids-algorithm},
  and partition the closed image $\tilde\pi_2(\cl{C})$
  We obtain various partitions on partially overlapping subsets of
  $\tilde\pi_2\big(\tilde\pi_1(\Gamma_f)\big)$.
  Let us consider the common refinement of them,
  it is a partition of $\tilde\pi_2\big(\tilde\pi_1(\Gamma_f)\big)$
  into locally closed sets.
  We iterate this procedure, and obtain 
  partitions of $\tilde\pi_{j}\circ\dots\circ\tilde\pi_1(\Gamma_f)$
  for each $j$.
  (Note that $k$ in these applications of Claim~\ref{Euclids-algorithm}
  is always at most $\dim(X)+\deg(f)-2$.)
  In the last step we obtain a partition of
  $\cl{\tilde\pi(\Gamma_f)}=\cl{f(X)}$
  as required.
\end{proof}

\begin{claim}
  \label{Euclids-algorithm}
  Let $Z\subseteq\Fclosed^k$ be a locally closed set and
  $\Gamma$ be the common zero locus inside $\Fclosed\times Z$
  of some polynomials of degree at most $d$.
  \begin{enumerate}[\indent(a)]
  \item \label{item:16}
    Then $Z$ has a partition into at most $(d+2)^{(d+1)^{k+2}-1}$
    \xxx{Changed to smaller!}
    locally closed subsets $Z_i$
    and there are corresponding $(k+1)$-variate
    polynomials $P_i$ of degree at most $d^{(d+1)^{k+1}2^d}$
    \xxx{Changed!}
    such that
    $$
    \Gamma\cap\big(\Fclosed\times Z_i\big) =
    \left\{(t,\uz)\in\Fclosed\times Z_i\,\Big|\,
      P_i(t,\uz)=0\right\}
    $$
    for all $i$,
    and the closures $\cl{Z_i}$
    are defined via equations of degree at most $d^{(d+1)^{k+1}2^d}$
    \xxx{Changed!}
    plus the equations of $\cl{Z}$.
  \item \label{item:17}
    Those points $\uz\in Z_i$ for which
    $\Gamma\cap\big(\Fclosed\times\{\uz\}\big)$ has any prescribed
    number of points (it can be $0,1,\dots d$ or $\infty$)
    form a locally closed subset that is defined (inside $Z$)
    via equations of degree at most $d^{(d+1)^{k+1}2^d}$,
    \xxx{Changed!}
    and the total number of these subsets is at most $(d+2)^{(d+1)^{k+2}}$.
    \xxx{Changed to smaller!}
  \item \label{item:18}
    Moreover, one may require both partitions to have the following additional
    property: 
    the closure in $Z$ of each partition class
    is the union of partition classes.
  \end{enumerate}
\end{claim}

\begin{proof}[Sketch of proof]
  The upper bounds and part~\eqref{item:18}
  follow immediately from our construction,
  we leave them to the reader.
  $\Gamma$ can be defined as the common zero locus
  inside $\Fclosed\times Z$
  of at most $(d+1)^{k+1}$
  \xxx{Changed!}
  polynomials of degree at most $d$
  (see Fact~\ref{dimension-degree}.\eqref{item:9}).
  We prove \eqref{item:16} via induction
  on the number of defining polynomials.
  If $\Gamma=\Fclosed\times Z$ then there is nothing to prove.
  Otherwise let $g$ be one of the nonzero defining polynomials of
  $\Gamma$ and $\Gamma'\subseteq\Fclosed\times Z$
  the common zero locus of the other defining polynomials.
  Applying the induction hypothesis to $\Gamma'$ gives us
  a partition $\bigcup_jZ'_j=Z$ and corresponding
  polynomials $P_j'$. Our goal is to refine this partition,
  i.e. find partitions $Z'_j=\bigcup_iZ'_{ji}$ and find appropriate
  polynomials $P'_{ji}$.
  We shall find the $Z'_{ji}$ one by one with the following algorithm.

  The portion of $\Gamma$ that lies inside
  $\Fclosed\times Z'_j$ is defined by the equations
  $P'_j(t,\uz)=g(t,\uz)=0$
  (besides the equations and inequalities defining $Z'_j$).
  We consider $g$ and $P'_j$ as polynomials in the variable $t$
  whose coefficients are polynomial functions of the parameter $\uz$.
  Note that $g$ and $P'_j$ as well as
  all the polynomials $P'_{ji}$ we construct below
  have $t$-degrees at most $d$.
  Our plan is to find the gcd of $g$ and $P'_j$ with respect to the
  variable $t$ for all values of $\uz$ simultaneously.
  In order to do so we try to run Euclid's algorithm simultaneously
  for all $\uz$. There are two obstacles we have to overcome.
  First, for different values of $\uz$ the algorithm needs a different
  number of steps to complete.
  Second, to do a polynomial division uniformly for several values of
  $\uz$ we have to make sure that the degree of the divisor do not
  vary with $\uz$ (i.e. we can talk about the leading coefficient).
  So before each polynomial division
  we construct also a partition of $Z'_j$,
  always refining the partition obtained in the previous step,
  so that the upcoming division can be done uniformly
  for values $\uz$ lying in the same partition class.
  
  To begin with, let $Z'_{j0}$ and $Z'_{j1}$ denote the loci of those
  $\uz\in Z'_j$ where all coefficients of $g$ or $P'_j$ respectively
  vanish. We set $P'_{j0}=P'_j$ and $P'_{j1}=g$.
  Similarly, for each pair of integers $0\le a,b\le d$
  we consider the locus of those $\uz\in Z'_j$ where the $t$-degrees of $g$
  and $P'_j$ are just $a$ and $b$.
  This is a partition of $Z'_j$ into locally closed subsets,
  each defined via the vanishing or non-vanishing of a number of
  coefficients.
  For parameter values $\uz$ lying in $Z'_{j0}$ or $Z'_{j1}$
  the algorithm stops right away
  with gcd equal to $P'_{j0}$ or $P'_{j1}$.
  On the other hand, for any other  partition class
  $\tilde Z\subseteq Z'_j$
  we can do the first polynomial division uniformly for all
  $\uz\in\tilde Z$.

  During the algorithm we do similar subdivisions again and again.
  Suppose that we completed a number of polynomial divisions
  and constructed the partition corresponding to the last completed division.
  Let $\tilde Z$ be a class of that partition
  and suppose that the algorithm is still running for $\uz\in\tilde Z$
  and $\tilde g$ and $\tilde r$ are
  the divisor and the remainder of the last completed polynomial division
  for all values $\uz\in\tilde Z$.
  We consider the locus of those $\uz\in\tilde Z$ where 
  all coefficients of $\tilde r$ vanish (here $\tilde g$ does not vanishes).
  This will be our next $Z'_{ji}$ (whatever $i$ follows now).
  For $\uz\in Z'_{ji}$ Euclid's algorithm stops at this stage,
  and we set $P'_{ji}=\tilde g$, the gcd we obtain.
  As before, we partition $\tilde Z\setminus Z'_{ji}$
  according to the $t$-degree of $\tilde r$ (here the $t$-degree of
  $\tilde g$ is unimportant).
  Then we can do the polynomial division $\tilde g:\tilde r$
  uniformly for values $\uz$ lying in the same partition class.
  This way we obtain our new remainders (one for each partition class),
  and Euclid's algorithm continues.

  It is clear that for each $\uz\in Z'_j$
  the gcd is found in at most $\deg(g)+1\le d+1$ steps,
  hence we obtain the promised partition $Z'_j=\cup_iZ'_{ji}$.
  The induction step is complete.

  Part~\eqref{item:17} follows from part~\eqref{item:16}.
  Indeed, the portion of $\Gamma$ that lies inside
  $\Fclosed\times Z_i$ is defined by the equation
  $P_i(t,\uz)=0$ (besides the equations of $Z_i$).
  For each $\uz\in Z_i$ the number of points in
  $\Gamma\cap\big(\Fclosed\times\{\uz\}\big)$ is either $\infty$ 
  (in case all $t$-coefficients of $P_i$ are zero at $\uz$),
  or equal to the $t$-degree of the polynomial $P_i(t,\uz)$
  (which is at most $d$).
  The locus of those $\uz$ which correspond to a given degree
  can be defined via the vanishing or nonvanishing of a number of
  $t$-coefficients of $P_i(t,\uz)$.
  This proves the claim.
\end{proof}

\section{Concentration in general}
\label{sec:conc-gener}

Let $\alpha\subseteq \Fclosed^m$ be a finite subset.  An essential part
of our general strategy is to find closed sets $X$ which contain a
large number of elements of $\alpha$ compared to their dimension.  To
measure the relative size of $\alpha\cap X$ we introduce the
following:

\begin{defn}
  For each subset $X\subseteq\Fclosed^m$ with $\dim(\cl{X})>0$ we
  define the \emph{concentration} of $\alpha$ in $X$ as follows:
  $$
  \mu(\alpha,X)\ =\ \frac{\log|\alpha\cap X|}{\dim(\cl{X})}
  $$
  For simplicity, here and everywhere in this paper, $\log$ stands for
  the natural logarithm.  When $\alpha\cap X=\emptyset$, we set
  $\mu(\alpha,X)=-\infty$.
\end{defn}

In this section we first show that the concentration in a closed set
$X$ does not decrease too much when we take an appropriate irreducible
closed subset.

\begin{prop} \label{concentration-is-bounded} Let $X\subseteq
  Y\subseteq\Fclosed^m$ be closed sets of positive dimension.  Then
  for all finite sets $\alpha\subseteq\beta\subset\Fclosed^m$ with
  $\alpha\cap X\neq\emptyset$ we have:
  \begin{equation}
    \label{eq:1}
    0\le\mu(\alpha,X)\le
    \mu(\beta,X) \le
    \Frac{\dim(Y)}{\dim(X)}\cdot\mu(\beta,Y)
  \end{equation}
  and for all integers $n>0$ the $n$-fold direct products satisfy
  \begin{equation}
    \label{eq:2}
    \mu\Big(\Prod^n\alpha,\Prod^nX\Big) = 
    \mu(\alpha,X) \;.
  \end{equation}
\end{prop}

\begin{proof}
  Clear from the definition.
\end{proof}

\begin{lem} \label{irreducible-component-with-large-concentration} Let
  $Z\subseteq\Fclosed^m$ be a closed set with $\dim(Z)>0$ and
  $\alpha\subseteq\Fclosed^m$ a finite subset with $|\alpha\cap
  Z|>\deg(Z)$.  Then there is an irreducible component $Z'\subseteq Z$
  such that $\dim(Z')>0$ and
  \begin{equation}
    \label{eq:3}
    \mu(\alpha,Z') \ge \mu(\alpha,Z) - \log\big(\deg(Z)\big) \;.
  \end{equation}
\end{lem}

\begin{proof}
  Since $Z$ has at most $\deg(Z)$ irreducible components (see
  Fact~\ref{dimension-degree}.\eqref{item:5}) there is a component
  $Z'\subseteq Z$ with
  \begin{equation}
    \label{eq:4}
    \big|\alpha\cap Z'\big| \ge
    \frac{|\alpha\cap Z|}{\deg(Z)} >1\;.
  \end{equation}
  In particular we have $\dim(Z')>0$.  We take the logarithm of
  inequality~\eqref{eq:4}, divide the two sides by $\dim(Z')$ and
  rewrite it in terms of concentrations.  Using $\dim(Z')\le\dim(Z)$
  we obtain
  $$
  \mu(\alpha,Z')\ge \frac{\dim(Z)}{\dim(Z')}\mu(\alpha,Z) -
  \frac{\log\big(\deg(Z)\big)}{\dim(Z')} \ge
  $$
  $$
  \ge \mu(\alpha,Z)-\log\big(\deg(Z)\big)
  $$
  as required.
\end{proof}

The proof of Lemma~\ref{irreducible-component-with-large-concentration}
involves a choice.  For proving
Theorem~\ref{pre-main-thm} it will be important to use constructions
that are uniquely determined.  To this end we order the finite set
$\alpha$, and use this order to make the choices unique.  Of course,
$\alpha$-valued sequences and subsets of $\alpha$ can be ordered
lexicographically.

In the rest of the paper we state several existence results.
However, in the proofs we typically use explicit constructions.
When we write that our construction of a subset
(or a tuple of elements, etc.)
is uniquely determined, we understand that
the result of the construction depends uniquely on the input data
(which usually involves an ordered set $\alpha$).

\newcommand{\irredComp} {{\rm irr}}
\begin{lem} \label{something-irreducible-with-large-concentration} For
  all $N>0$ and $\Delta>0$ there are reals
  $B=B_\irredComp(N,\Delta)\ge0$ and $K=K_\irredComp(N,\Delta)\ge0$
  with the following property. \\
  Let $Z\subseteq\Fclosed^m$ be a closed set and
  $\alpha\subseteq\Fclosed^m$ an ordered finite subset.  Suppose that
  $0<\dim(Z)\le N$, $\deg(Z)\le\Delta$ and $|\alpha\cap Z|\ge K$.
  Then there is an irreducible closed subset $Z'\subseteq Z$ such that
  $\dim(Z')>0$, $\deg(Z')\le B$ and
  $$
  \mu(\alpha,Z') \ge \mu(\alpha,Z) - \log(B) \;.
  $$
  Moreover, our construction of $Z'$ is uniquely determined.
\end{lem}

\begin{proof}
  Let $B=\Delta^{(N+1)^N}$ and set $K>\Delta^{2N(N+1)^N}$. Then 
  \begin{equation}
    \label{eq:5}
    \mu(\alpha,Z)\ge\frac{\log(K)}N>\log\big(\Delta^{2(N+1)^N}\big) \;.
  \end{equation}
  We build by induction a sequence $Z=Z_0\supset Z_1\supset
  Z_2\supset\dots\supset Z_I$ of closed subsets such that
  \begin{equation}
    \label{eq:6}
    \begin{array}{c}
      \Big.0<\dim(Z_{i+1})<\dim(Z_{i})  \;, \\
      \deg(Z_{i+1})\le\deg(Z_{i})^{N+1} \le \Delta^{(N+1)^{i+1}}\;, \\
      \Big.\mu(\alpha,Z_i)\ge\mu(\alpha,Z)- 
      \log\big(\Delta^{i(N+1)^{i-1}}\big) \;.
    \end{array}
  \end{equation}
  for all $0\le i< I$.  Since the dimensions are strictly decreasing,
  such a sequence has length $I+1\le N$.  Suppose $Z_i$ is already
  constructed.  If it is irreducible, we stop the induction and set
  $Z'=Z_i$, the lemma holds in this case.  Otherwise, it follows from
  \eqref{eq:5} and \eqref{eq:6} that $|\alpha\cap
  Z_i|>\Delta^{(N+1)^N}>\deg(Z_i)$ and we may apply
  Lemma~\ref{irreducible-component-with-large-concentration}.  So
  there is an irreducible component $Z_i'\subseteq Z_i$ such that
  $\dim(Z_i')>0$ and
  \begin{equation}
    \label{eq:7}
    \mu(\alpha,Z_i')\ge \mu(\alpha,Z_i) - \log\big(\deg(Z_i)\big) \ge
    \mu(\alpha,Z)- \log\big(\Delta^{(i+1)(N+1)^i}\big) \;.
  \end{equation}
  Of course, there are possibly many choices for $Z_i'$, we choose one
  in such a way that the subset $\alpha_i=\alpha\cap Z_i'$ is
  lexicographically minimal among the possible intersections.  Note
  that $\alpha_i$ is uniquely determined, but $Z_i'$ may not be.  Then
  $\mu(\alpha_i,Z_i')=\mu(\alpha,Z_i')$ and using \eqref{eq:5} and
  \eqref{eq:7} we obtain $|\alpha_i|>\deg(Z_i)^{N+1}$.  If $Z_i'$ is
  the only irreducible component containing $\alpha_i$ then it is
  uniquely determined.  We stop the induction and set $Z'=Z_i'$, the
  lemma holds in this case.

  Otherwise let $T_1,T_2,\dots$ denote those irreducible components of
  $Z_i$ which contain $\alpha_i$ and let $Z_{i+1}=\bigcap^jT_j$ be
  their intersection, this is again uniquely determined.  Clearly
  $\dim(Z_{i+1})<\dim(Z_i)$ and we shall prove that
  $$
  \deg(Z_{i+1})\le\deg(Z_i)^{N+1} \;.
  $$
  In fact it is more convenient to prove a slightly stronger
  statement: for each closed subset $W\subseteq Z_i$ we have
  \begin{equation}
    \label{eq:8}
    \deg(W\cap Z_{i+1})\le\deg(W)\cdot\deg(Z_i)^{\dim(W)} \;.
  \end{equation}
  We prove \eqref{eq:8} by induction on $\dim(W)$, it obviously holds
  for $\dim(W)=0$.  Assume for a moment that $W$ is irreducible.  If
  it is contained in all $T_j$ then $W\cap Z_{i+1}=W$ and \eqref{eq:8}
  holds.  On the other hand, if say $W\not\subseteq T_1$ then
  $W'=W\cap T_1$ has smaller dimension, hence satisfies the analogue
  of \eqref{eq:8}.  But
  $\deg(W')\le\deg(W)\deg(T_1)\le\deg(W)\deg(Z_i)$, so we have
  $$
  \deg(W\cap Z_{i+1})=\deg(W'\cap Z_{i+1})\le
  $$
  $$
  \le \deg(W')\deg(Z_i)^{\dim(W)-1}\le \deg(W)\deg(Z_i)^{\dim(W)} \;.
  $$
  as we promised.  In order to complete the induction step for a
  reducible $W$ we simply add up the analogous inequalities for each
  component of $W$.

  Then $\dim(Z_{i+1})>0$ by Remark~\ref{too-many-points}.  Now we have
  $$
  \mu(\alpha,Z_{i+1}) = \mu(\alpha_i,Z_{i+1}) > \mu(\alpha_i,Z_{i}') =
  \mu(\alpha,Z_i') \;,
  $$
  hence $Z_{i+1}$ satisfies \eqref{eq:6}.  As we noted earlier, the
  induction must stop in at most $N$ steps, which proves the lemma.
\end{proof}

Next we show that the concentration in a closed set $X$
does not decrease too much 
when we map $X$ somewhere by a ``nice'' morphism.

\begin{lem} \label{constant-fibre-dimension} Let
  $Z\subseteq\Fclosed^m$ be an irreducible closed set,
  $\alpha\subset\Fclosed^m$ an ordered nonempty finite set and $f:Z\to
  \Fclosed^l$ a morphism such that
  $$
  \dim(Z)>\dim\big(\cl{f(Z)}\big)>0
  $$
  and
  $$
  \dim\big(Z\big) = \dim\big(\cl{f(Z)}\big)+\dim\big(f^{-1}(t)\big)
  $$
  for all $t\in f(\alpha\cap Z)$.
  Then there is a fibre $S=f^{-1}(s)$, $s\in f(\alpha\cap Z)$
  such that
  for each value (negative, positive or $0$) of the parameter $\e$
  one has
  \begin{equation}
    \label{eq:9}
    \left\{
    \begin{array}{lccl}
      \text{either\kern10pt} &
      \mu\big(f(\alpha\cap Z),\cl{f(Z)}\big) &\ge&
       \mu(\alpha,Z) - \e\dim(S) \\
      \text{or} &
      \mu\big(\alpha,S\big) &\ge&
      \mu(\alpha,Z) + \e\dim\big(\cl{f(Z)}\big)
    \end{array}
    \right.
  \end{equation}
  Moreover, our construction of $S$ is uniquely determined.
\end{lem}

Note that if all nonempty fibres of $f$ have the same dimension,
then the condition
$\dim\big(Z\big) = \dim\big(\cl{f(Z)}\big)+\dim\big(f^{-1}(t)\big)$
is satisfied (see Fact~\ref{closed-set-constructions}.\eqref{item:14}).
Note also that $S$ is a closed set with $\deg(S)\le\deg(f)$
by Fact~\ref{closed-set-constructions}.\eqref{item:12}.

\begin{proof}
  Let us consider those fibres $f^{-1}(t)$ where the number of points
  $\big|\alpha\cap f^{-1}(t)\big|$ is maximal, and
  let $S=f^{-1}(s)$ be the one among them for which the subset
  $\alpha\cap S\subseteq\alpha$ is lexicographically minimal.
  Then by assumption we have
  $$
  0<\dim(S)=
  \dim(Z)-\dim\big(\cl{f(Z)}\big)<\dim(Z) \;.
  $$
  We have
  $$
  \big|\alpha\cap Z\big| =
  \sum_{t\in f(\alpha\cap Z)}\big|\alpha\cap f^{-1}(t)\big| \;,
  $$
  hence
  $$
  \big|\alpha\cap Z\big| \le
  \big|f(\alpha\cap Z)\big|\cdot \big|\alpha\cap S\big|
  $$
  We take the logarithm of our inequality 
  and rewrite it in terms of concentrations:
  $$
  \mu(\alpha,Z)\cdot\dim(Z) \le
  \mu\big(f(\alpha\cap Z),\cl{f(Z)}\big)\cdot\dim(\cl{f(Z)}) +
  \mu\big(\alpha,S\big)\cdot\dim(S)
  $$
  We divide both sides by $\dim(Z)$ and we introduce two extra terms
  involving $\e$ on the right hand side
  which cancel each other:
  $$
  \mu(\alpha,Z) \;\le
  \scriptstyle{
    \Big[\mu\big(f(\alpha\cap Z),\cl{f(Z)}\big) + \e\dim(S)\Big]
    \Frac{\dim(\cl{f(Z)})}{\dim(Z)} \ +\ 
    \Big[\mu\big(\alpha,S\big) - \e\dim\big(\cl{f(Z)}\big)\Big]
    \Frac{\dim(S)}{\dim(Z)}
  }
  $$
  On the right hand side we see a weighted arithmetic mean
  of the two expressions in square brackets.
  We can certainly bound it it from above
  with the larger of them,
  which justifies our statement.
\end{proof}

The following extension of Lemma~\ref{constant-fibre-dimension}
is our basic tool for transporting
large concentration from one subset to another.
The idea is that if the transport fails than we get an even larger
concentration somewhere inside the first subset.

\newcommand{\transport} {{\rm transport}}
\begin{lem} [Transport]
  \label{Transport}
  For all $\Delta>0$
  there is a real $B=B_\transport(\Delta)\ge0$ 
  with the following property. 
  Let $X$ be an affine algebraic set, $Z\subseteq X$ a closed subset and 
  $f:X\to\Fclosed^m$ be a morphism with 
  $\deg(Z)\le\Delta$, $\deg(f)\le\Delta$ and $\dim\big(\cl{f(Z)}\big)>0$.
  Suppose that $Z$ is irreducible.
  Then for all ordered finite subsets $\alpha\subseteq X$
  and all $\e\ge0$ either
  \begin{equation}
    \label{eq:10}
    \mu\big(f(\alpha),\cl{f(Z)}\big)\ge
    \mu(\alpha,Z)-\log(B)-\e\cdot\dim(Z)
  \end{equation}
  or there is a closed subset
  $S\subset Z$ such that
  $\deg(S)\le B$,\\ $0<\dim(S)<\dim(Z)$ and
  \begin{equation}
    \label{eq:11}
    \mu(\alpha,S) \ge
    \mu(\alpha,Z) - \log(B) + \e \;.
  \end{equation}
  Moreover, our construction of $S$ is uniquely determined.
\end{lem}

Note, that the condition $\dim\big(\cl{f(Z)}\big)>0$ implies that
$\dim(Z)>0$, hence the concentrations appearing in the lemma are
defined.

\begin{proof}
  To simplify notation we replace $\alpha$ with $\alpha\cap Z$,
  $X$ with $Z$, $\Delta$ with $\Delta^2$ 
  (see Fact~\ref{closed-set-constructions}.\eqref{item:15})
  and $f$ with its restriction to $Z$,
  then $\alpha\subseteq Z$.
  If $\alpha=\emptyset$ then \eqref{eq:10} holds
  automatically since the right hand side is $-\infty$.
  So we assume $\alpha\neq\emptyset$. 
  This implies that $f(\alpha)\neq\emptyset$,
  hence the left hand side of \eqref{eq:10} is non-negative.
  If $\mu(\alpha,Z)\le\log(B)$
  then inequality \eqref{eq:10} obviously holds since the right hand
  side is nonpositive.
  So we assume $\mu(\alpha,Z)>\log(B)$
  which implies $\big|\alpha\big|>B$.

  First we prove a special case:
  \begin{equation}
    \label{eq:12}
    \parbox{300pt}
    {If \ \ $\dim\big(f^{-1}(t)\big)=\dim(Z)-\dim\big(\cl{f(Z)}\big)$
      for all $t\in f(\alpha)$ then the lemma is true
      with any \ \ $B\ge1+\Delta$.}
  \end{equation}
  If $\dim(Z)>\dim\big(\cl{f(Z)}\big)$ then we apply
  Lemma~\ref{constant-fibre-dimension} with parameter $\e$.
  We get a fibre $S=f^{-1}(s)$ satisfying \eqref{eq:9}.
  Since $\e\ge0$, we may replace $\e\dim(\cl{f(Z)})$ with $\e$ and
  $\e\dim(S)$ with $\e\dim(Z)$,
  hence either \eqref{eq:10} or \eqref{eq:11} holds for any $B\ge 1$.
  By Fact~\ref{closed-set-constructions}.\eqref{item:12}
  $S=f^{-1}(s)$ is closed and $\deg(S)\le\Delta$,
  hence \eqref{eq:12} is proved in this case.

  On the other hand, if $\dim(Z)=\dim\big(\cl{f(Z)}\big)$
  (and we are still in the special case of \eqref{eq:12}),
  then all points of $\alpha$ are contained in finite fibres of $f$,
  and the number of points in each finite fibre is at most $\deg(f)\le\Delta$
  (see Fact~\ref{closed-set-constructions}.\eqref{item:12}).
  Hence
  $$
  \mu\Big(f(\alpha),\cl{f(Z)}\Big) =
  \frac{\log\big|f(\alpha)\big|}{\dim\big(\cl{f(Z)}\big)} \ge
  \frac{\log\Big(\big|\alpha\big|\big/\Delta\Big)}{\dim(Z)} \ge
  \mu(\alpha,Z) - \log(\Delta)\;,
  $$
  and therefore \eqref{eq:10} holds for any $B\ge\Delta$.
  The special case \eqref{eq:12} is proved.

  Next we prove the lemma in full generality.
  We define the following subset:
  $$
  \alpha' = \Big\{ z\in\alpha\;\Big|\;
  \dim\big(f^{-1}(f(z))\big) = \dim(Z)-\dim\big(\cl{f(Z)}\big)
  \Big\} \;.
  $$
  First we deal with the case
  $|\alpha'|\ge\big|\alpha\big|/2$.
  We have
  $$
  \mu(\alpha',Z) =
  \frac{\log\big|\alpha'\big|}{\dim(Z)} \ge
  \frac{\log\big|\alpha\big|-\log(2)}{\dim(Z)} \ge
  \mu(\alpha,Z) - \log(2)\;.
  $$
  We apply the special case \eqref{eq:12} of the lemma
  to $\alpha'$ and $Z$.
  We obtain that either
  $$
  \mu\big(f(\alpha),\cl{f(Z)}\big) \ge
  \mu\big(f(\alpha'),\cl{f(Z)}\big) \ge
  $$
  $$
  \ge
  \mu(\alpha',Z)-\log(1+\Delta)-\e\cdot\dim(Z) \ge
  $$
  $$
  \ge
  \mu(\alpha,Z)-\log(2+2\Delta)-\e\cdot\dim(Z) \;,
  $$
  or there is a closed subset $S\subset Z$ such that
  $\deg(S)\le1+\Delta$, $0<\dim(S)<\dim(Z)$ and
  $$
  \mu(\alpha,S)\ge
  \mu(\alpha',S)\ge
  \mu(\alpha',Z)-\log(1+\Delta) +\e \ge
  $$
  $$
  \ge
  \mu(\alpha,Z)-\log(2+2\Delta)+\e\;.
  $$
  The lemma holds in this case with any $B\ge2+2\Delta$.

  In the remaining case we have
  $\big|\alpha'\big| < \big|\alpha\big|/2$.
  Setting
  $$
  S = \Big\{ z\in Z\;\Big|\;
  \dim\big(f^{-1}(f(z))\big) > \dim(Z)-\dim\big(\cl{f(Z)}\big)
  \Big\}
  $$
  we have $\big|\alpha\cap S\big|>\Frac12\big|\alpha\big|$.

  The irreducibility of $Z$ implies 
  (see Fact~\ref{closed-set-constructions}.\eqref{item:14} and
  Fact~\ref{dimension-degree}.\eqref{item:7})
  that $S$ is a closed subset of $Z$ and
  $\dim(S)<\dim(Z)$, $\deg(S)\le\Delta'$
  with a certain bound $\Delta'=\Delta'\big(\dim(Z),\Delta\big)$. 
  We set
  $$
  B=B_\transport(\Delta) =\max\big(2+2\Delta,2\Delta'\big) \;.
  $$ 
  Then the set $S$ has 
  at least $|\alpha\cap S|>|\alpha|/2\ge B/2\ge\Delta'$ points,
  hence $\dim(S)>0$
  (see Remark~\ref{too-many-points}).
  Therefore $\mu(\alpha,S)$ is defined and we can write:
  $$
  \mu(\alpha,S) = \frac{\log|\alpha\cap S|}{\dim(S)} \ge
  \frac{\log|\alpha|-\log(2)}{\dim(S)} \ge
  $$
  $$
  \ge \frac{\dim(Z)}{\dim(S)}\mu(\alpha,Z) -
  \log(2) \ge 
  \mu(\alpha,Z)-\log(B)+\frac{\mu(\alpha,Z)}{\dim(S)} \;.
  $$
  We compare now the last term to $\e$. 
  If $\e\le\frac{\mu(\alpha,Z)}{\dim(S)}$
  then inequality~\eqref{eq:11} holds.
  On the other hand, for
  $\e>\frac{\mu(\alpha,Z)}{\dim(S)}\ge\frac{\mu(\alpha,Z)}{\dim(Z)}$
  the inequality~\eqref{eq:10} holds, since its right hand side 
  becomes negative.
  We proved the lemma in all cases.
\end{proof}

\section{Closed sets in groups}
\label{sec:closed-sets-groups}

\begin{defn} \label{G-notations}
  A \emph{linear algebraic group}
  is a closed subgroup $G\le GL(n,\Fclosed)$.
  We use this matrix realisation of $G$ to calculate degrees
  of closed subsets.
  We shall denote by $\mult(G)$ and $\inv(G)$
  the degrees of the morphisms $(g,h)\to gh$ and $g\to g^{-1}$.
\end{defn}

As usual, $\CZ(G)$, $[G,G]$ and $G^0$
denote the centre, the commutator subgroup and the unit component of $G$,
and for any subset $A\subseteq G$
we denote by
$\Span A$, $\CN_G(A)$ and $\CC_G(A)$
the generated subgroup, the normaliser and the centraliser of $A$.
The subgroup $\CC_G(A)^0$ is usually called
the \emph{connected centraliser} of $A$.
We shall often use products of several elements and subsets in the
usual sense. In order to distinguish from this kind of product,
the $m$-fold direct product of a subset $\alpha\subseteq G$ is
denoted by $\Prod^m\alpha\subseteq\Prod^mG$.

\begin{defn}
  Let $\alpha\subseteq GL(n,\Fclosed)$ be an ordered finite subset.
  This ordering extends to an ordering of the subgroup $\Span\alpha$ 
  (hence to $\alpha^i$ for all $i$) in a natural way.
  We shall use this extension without further reference.
\end{defn}

\begin{rem}
  We measure the complexity of a closed subset $X\subseteq\Fclosed^m$
  with two numerical invariants: $\dim(X)$ and $\deg(X)$.
  In contrast, we measure the complexity of a
  closed subgroup $G\le GL(n,\Fclosed)$ with four numerical invariants:
  $\dim(G)$, $\deg(G)$, $\mult(G)$ and $\inv(G)$.
  In order to reduce the number of variables to two, say $N$ and $\Delta$,
  we shall consider groups $G$ with
  $\dim(G)\le N$, $\deg(G)\le\Delta$, $\mult(G)\le\Delta$ and
  $\inv(G)\le\Delta$.
\end{rem}

It can be tiresome to bound all four numerical invariants of $G$.
By the following proposition
in most cases it is enough to bound only $\dim(G)$ and $\deg(G)$.

\begin{prop} \label{bounding-subgroups}
  Let $G$ be a linear algebraic group and $H\le G$ a closed subgroup.
  Then $\mult(H)\le\deg(H)^2\cdot\mult(G)$ and $\inv(H)\le\deg(H)\cdot\inv(G)$.
  In particular, if $G=GL(n,\Fclosed)$ then we have
  $\mult(H)\le\deg(H)^2\cdot2^{n^2}$ and $\inv(H)\le\deg(H)\cdot(n+1)^{n^2}$.
\end{prop}

\begin{proof}
  Follows immediately from
  Fact~\ref{closed-set-constructions}.\eqref{item:15}
  and Fact~\ref{dimension-degree}.\eqref{item:6}.
\end{proof}

\begin{fact} \label{degree-of-group-theoretic-maps}
  Let $G$ be a linear algebraic group.
  Suppose that $f:\Prod^mG\to\Prod^nG$ is a morphism
  for some integers $m,n>0$
  whose $n$ coordinates are all defined to be
  product expressions (evaluated in the group $G$)
  of length at most $k$ of some fixed group elements,
  the $m$ variables and their inverses.
  Then $\deg\big(\cl{f(G)}\big)\le\deg(f)\le\inv(G)^{l}\mult(G)^{n(k-1)}$
  where $l\le nk$ denotes the total number of times inverted variables
  occur in the $n$ expressions
  (see Fact~\ref{closed-set-constructions}.\eqref{item:11}).
  If the product expressions do not contain the inverse of the variables
  then of course the bound does not depend on $\inv(G)$.
\end{fact}

\begin{defn} \label{tau}
  Let $G$ be a linear algebraic group.
  For all $m>0$ and
  for each sequence $\ug=(g_1,g_2,\dots, g_m)$, $g_i\in G$ 
  we define the morphism
  $$
  \tau_\ug:\Prod^{m}G\to G\;,
  $$
  $$
  \tau_{\ug}(a_1,\dots, a_m) =
  (g_1^{-1}a_1g_1)(g_2^{-1}a_2g_2)\dots(g_m^{-1}a_mg_m) \;,
  $$
\end{defn}

\begin{rem} \label{tau-rem}
  Let $G$ be a linear algebraic group and
  $\ug=(g_1,g_2,\dots, g_m)$ any sequence.
  Suppose that $\dim(G)\le N$, $\deg(G)\le\Delta$ and $\mult(G)\le\Delta$
  for certain values $N$ and $\Delta$.
  According to Fact~\ref{degree-of-group-theoretic-maps}
  there is a common upper bound on the degrees:
  $$
  \deg\big(\tau_\ug\big) \le
  \Delta_\tau\big(m,N,\Delta\big) \;.
  $$
  In fact, it is easy to see that conjugation by $g_i$ is a linear
  transformation hence $\deg(\tau_\ug)\le\mult(G)^{m-1}\le\Delta^{m-1}$.
\end{rem}

\begin{fact} \label{closure-of-products}
  Let $G$ be a connected linear algebraic group
  and $A,B\subseteq G$ arbitrary subsets.
  Then
  $$
  AB\subseteq\cl{A}\;\cl{B}\subseteq\cl{AB} \;.
  $$
\end{fact}

We give a short proof, see also \cite[page~56]{Hu1}.
Let us consider the multiplication map $f:G\times G\to G$.
If $AB = f(A\times B)$ satisfies a polynomial equation $p=0$ then
$p\big(f(A\times B)\big)=0$, i.e. the polynomial
$p\big(f(\text{\_})\big)$
vanishes on $A\times B$. But then it must vanish on its closure
$\cl{A\times B}=\cl{A}\times\cl{B}$, hence $p$  vanishes on
$f\big(\cl{A}\times\cl{B}\big)=\cl{A}\;\cl{B}$.
\qed
\vskip 3pt

Closed subgroups of an algebraic group can be very complicated. 
In contrast, centraliser subgroups
are defined by linear equations, and normalisers of a closed subset
$X$ can be defined in terms of the equations of $X$.
This proves that

\begin{fact} \label{normaliser-centraliser-fact}
  Let $G$ be a linear algebraic group.
  \begin{enumerate}[(a)]
  \item \label{item:19}
    The centraliser $\CC_G(X)$ of any subset $X\subseteq G$ is closed and 
    its numerical invariants are bounded:
    $\deg\big(\CC_G(X)\big)\le\deg(G)$,
    $\mult\big(\CC_G(X)\big)\le\mult(G)$ and
    $\inv\big(\CC_G(X)\big)\le\inv(G)$.
    If $X$ is closed then its normaliser $\CN_G(X)$ is also closed
    and its numerical invariants are also bounded:
    $\deg\big(\CN_G(X)\big)\le\deg(G)\deg(X)^{\dim(G)}$,
    $\mult\big(\CN_G(X)\big)\le\mult(G)\deg(X)^{\dim(G)}$ and
    $\inv\big(\CN_G(X)\big)\le\inv(G)\deg(X)^{\dim(G)}$.
  \item \label{item:20}
    Cosets of a closed subgroup $H\le G$ are also closed, they all have
    the same degree. Therefore 
    $$
    \big|G:G^0\big| = \frac{\deg(G)}{\deg(G^0)}\le\deg(G) \;.
    $$
  \end{enumerate}
\end{fact}

Later we plan to apply the Transport Lemma~\ref{Transport}
to various morphisms of the form $\tau_\ug$.
In the rest of this section we construct the appropriate sequences $\ug$. 

The following proposition gives a morphism which maps a direct power of a given
closed subset $Y$ onto a closed subgroup $H$.
It should be considered folklore,
see e.g. \cite[Proposition~on~page~55]{Hu1} for a similar statement.
Nevertheless, for the sake of completeness, we include a proof.

\newcommand{\largeProduct} {{\text{}}}
\begin{prop}
  \label{normal-subgroup-in-product}
  Let $Y\subseteq GL(n,\Fclosed)$ be an irreducible closed subset of positive
  dimension and $1\in\alpha\subset GL(n,\Fclosed)$ an ordered finite subset.
  Let $H\le GL(n,\Fclosed)$ denote the
  smallest closed subgroup which is normalised by $\alpha$ and contains $Y$.
  Suppose that $\dim(H)\le m$.
  Then there is a sequence
  $\ug=(g_1,g_2,\dots, g_{2m})$
  of elements $g_i\in\alpha^{m-1}$
  such that
  $$
  H = \tau_\ug\;\Big(\Prod^{2m}(Y^{-1}Y)\Big) =
  (g_1^{-1}Y^{-1}Yg_1)(g_2^{-1}Y^{-1}Yg_2)\dots(g_{2m}^{-1}Y^{-1}Yg_{2m})
  \;.
  $$
  Moreover, our construction of $\ug$ is uniquely determined,
  $H$ is connected
  and there is a universal bound
  $\deg(H)\le\delta_\largeProduct\big(m,\deg(\cl{Y^{-1}Y})\big)$.
\end{prop}

\begin{rem}
  In applications the dimension of $H$ may not be known, but
  if $G\le GL(n,\Fclosed)$ is any closed subgroup
  normalised by $\alpha$ which contains $Y$ then
  one may set $m=\dim(G)$ and one may also use the bound
  $$
  \deg(\cl{Y^{-1}Y})\le
  \inv(G)\cdot\mult(G)\cdot\deg(Y)^2
  $$
  (see Fact~\ref{dimension-degree}.\eqref{item:6}
  and Fact~\ref{closed-set-constructions}.\eqref{item:15}).
\end{rem}

\begin{proof}
  We set $g_1=1$.
  We will define $g_i\in \alpha^{i-1}$ by induction and consider
  the product sets
  $$
  Z_i =
  (g_1^{-1}Y^{-1}Yg_1)(g_2^{-1}Y^{-1}Yg_2)\dots(g_{i}^{-1}Y^{-1}Yg_{i})
  \subseteq H \;.
  $$
  Suppose that $g_1,g_2,\dots,g_i$ are already defined.
  We set $g_{i+1}\in\alpha^i$ to be the first element
  such that
  $$
  \dim\big(\cl{Z_i}\big) <
  \dim\Big(
  \cl{Z_i\cdot(g_{i+1}^{-1}Y^{-1}Yg_{i+1})}
  \Big)
  \;,
  $$
  if there is any.
  Since the dimension of $\cl{Z_i}$ is strictly increasing,
  eventually we must arrive to an index $i\le m$ so that
  $g_{i+1}$ does not exist.
  But then for all $g\in\alpha^i$ the closed subsets
  $$
  \cl{Z_i}\subseteq \cl{Z_i\cdot(g^{-1}Y^{-1}Yg)}
  $$
  are irreducible (see Fact~\ref{closed-set-constructions}.\eqref{item:11})
  of the same dimension, hence they are equal.
  This implies that
  $\cl{Z_i}^2\subseteq\cl{Z_i}$ and $g^{-1}\cl{Z_i}g\subseteq\cl{Z_i}$
  for all $g\in\alpha$,
  hence $\cl{Z_i}$ is a closed connected subgroup normalised by $\alpha$
  i.e. $\cl{Z_i}=H$.
  By Fact~\ref{closed-set-constructions}.\eqref{item:11}
  the product $Z_i$ contains a dense open subset of $H$,
  hence $H=Z_i^2$ by \cite[Lemma~on~page~54]{Hu1}.
  Setting $g_{i+j}=g_j$ for $1\le j\le i$ and 
  $g_{2i+1}=\dots g_{2m}=1$ we obtain our statement.
\end{proof}

\begin{lem} \label{finite-tau}
  Let $G\le GL(n,\Fclosed)$ be a closed subgroup,
  $Z\subseteq G\times G$
  an irreducible closed set and $(a,b)\in Z$.
  Suppose that $\cl{\tau_{(1,1)}(Z)}$ has dimension $0$ 
  i.e. it is a finite set.
  Then there is an irreducible closed subset $A\subseteq G$ 
  such that
  \begin{equation}
    \label{eq:13}
    Z = \left\{
      (ah,h^{-1}b)\;\Big|\;h\in A\right\}
  \end{equation}
  and
  $$
  \Big\{c\in GL(n,\Fclosed)\;\Big|\; 
  \dim\Big(\cl{\tau_{(c,1)}\big(Z\big)}\Big)=0\Big\} =
  \CC_{GL(n,\Fclosed)}\big(A\big) \;.
  $$
\end{lem}

Note that in the proof we define $A$ explicitly (hence uniquely),
but we do not use this fact later.

\begin{rem}
  Equation \eqref{eq:13} implies immediately that
  $\dim(A)=\dim(Z)$ and $1\in A$.
\end{rem}

\begin{proof}
  By assumption $\tau_{(1,1)}(Z)$ is finite and its closure is irreducible
  (see Fact~\ref{closed-set-constructions}.\eqref{item:11}), 
  hence it is the single point $ab\in G$.
  Let $\pr_1:G\times G\to G$ denote the projection on the first factor.
  We set
  $$
  A = a^{-1}\pr_1(Z) \;.
  $$
  We shall prove later, that it is in fact closed.
  Anyway, $\cl{A}$ is irreducible
  (see Fact~\ref{closed-set-constructions}.\eqref{item:11})
  and by definition $1=a^{-1}a\in A$.
  Then each point of $Z$ has the form $(ah,\beta)$ 
  with some $h\in A$ and $\beta\in G$,
  and for all $h\in A$ there must exist at least one such point.  
  But then $ab = \tau_{(1,1)}(ah,\beta) = ah\beta$
  hence
  $\beta = h^{-1}b$.
  This proves equation \eqref{eq:13}.
  The set $Z$ is closed, 
  hence $A$ is closed by equation \eqref{eq:13}.
  Now
  $$
  \tau_{(c,1)}(Z) =
  \Big\{c^{-1}(ah)c(h^{-1}b)\;\Big|\;h\in A\Big\} =
  c^{-1}a\Big\{hch^{-1}\;\Big|\;h\in A\Big\}b
  $$
  for all $c\in GL(n,\Fclosed)$. This has dimension $0$ iff the set
  $\big\{hch^{-1}\big|h\in A\big\}$ is finite.
  But $A$ is irreducible, hence its closed image
  $\cl{\big\{hch^{-1}\big|h\in A\big\}}$
  is also irreducible
  (see Fact~\ref{closed-set-constructions}.\eqref{item:11}),
  so it is finite iff it is a single point
  (see Fact~\ref{too-many-points})
  i.e. iff $hch^{-1}$ is independent of $h\in A$.
  But $1\in A$, hence this last condition is equivalent to
  $hch^{-1}=c$ for all $h\in A$, which simply means
  that $c$ commutes with all $h\in A$.
  This proves the lemma.
\end{proof}

The following corollary constructs a morphism $\tau_{\ug}$
which maps a given closed subset $Z$ of some direct power of $G$
onto a subset of $G$ of positive dimension.

\begin{cor} \label{positive-dimensional-product}
  Let $G\le GL(n,\Fclosed)$ be a linear algebraic group
  and let $1\in\alpha\subset G$ be an ordered finite subset
  whose centraliser $\CC_G(\alpha)$ is finite.
  Then for each integer $m\ge 0$ and 
  each irreducible closed subset $Z\subset\Prod^{m}G$
  of dimension $\dim(Z)>0$
  there is a sequence 
  $\ug=(g_1,g_2,\dots, g_m)\in\Prod^m\alpha$
  such that the closed image $\cl{\tau_{\ug}(Z)}$
  has positive dimension.
  Moreover, our construction of $\ug$ is uniquely determined.
\end{cor}

\begin{proof}
  We shall prove the theorem by induction on $m$.
  For $m=1$ the statement is obvious.
  So let $m\ge2$ and we assume that the corollary holds whenever the
  number of factors is smaller than $m$.
  We define several morphisms.
  For all $g\in G$ let
  $$
  \sigma_g:\Prod^{m}G\to\Prod^{m-1}G
  \;,\quad\quad
  \sigma_g(a_1,\dots, a_m)= 
  \big(g^{-1}a_1ga_2,a_3,\dots, a_m\big)
  $$  
  and let
  $$
  \pi:\Prod^{m}G\to\Prod^{m-2}G
  \;,\quad\quad
  \pi(a_1,\dots, a_m)= (a_3,a_4,\dots, a_m)\;,
  $$
  $$
  \rho:\Prod^{m-1}G\to\Prod^{m-2}G
  \;,\quad\quad
  \rho(a_2,\dots, a_m)= (a_3,a_4,\dots, a_m)\;.
  $$
  For $m=2$ we use the convention that $\Prod^0G$ is a single point.
  Note, that these morphisms manipulate only the first two
  coordinates. In particular
  $$
  \rho\big(\sigma_g(x)\big) = \pi(x)
  \quad\quad
  \text{for all }x\in\Prod^{m}G\;.
  $$
  Our goal is to find an element $g\in\alpha$ such that 
  \begin{equation}
    \label{eq:14}
    \dim\big(\cl{\sigma_g(Z)}\big)>0\;.
  \end{equation}
  Then we choose the smallest such $g$ (in the order of $\alpha$)
  and use the induction hypotheses for 
  $\cl{\sigma_g(Z)}\subseteq\Prod^{m-1}G$.
  This proves the corollary for $Z$ as well.

  We distinguish two cases.
  Suppose first that for all $z\in\Prod^{m-2}G$ the subset
  $Z\cap\pi^{-1}(z)$ is finite (i.e. $0$ dimensional).
  Then $\dim(Z)=\dim\big(\cl{\pi(Z)}\big)$ is positive
  (see Fact \ref{closed-set-constructions}.\eqref{item:14}).
  But
  $$
  \dim(Z) \ge
  \dim\big(\cl{\sigma_g(Z)}\big) \ge
  \dim\Big(\cl{\rho\big(\sigma_g(Z)\big)}\Big) =
  \dim\big(\cl{\pi(Z)}\big)
  $$
  hence all these dimensions are equal. 
  Hence \eqref{eq:14} is achieved,
  the corollary holds in this case.

  Suppose next that there is a point $z\in\Prod^{m-2}G$ such that
  $Z\cap\pi^{-1}(z)$ has an irreducible component $Z'$ with positive dimension.
  For simplicity we shall identify the subset
  $\pi^{-1}(z)=\Prod^2G\times\{z\}\subset\Prod^{m}G$ with $\Prod^2G$
  and also $\rho^{-1}(z)=G\times\{z\}\subset\Prod^{m-1}G$ with $G$.
  With these identifications we have
  $$
  \sigma_g(x) = \tau_{(g,1)}(x)
  \quad\quad
  \text{for all }x\in\Prod^{2}G\text{ and all }g\in\alpha \;.  
  $$
  If $\cl{\sigma_{1}(Z')}=\cl{\tau_{(1,1)}(Z')}$ has positive dimension
  then (\ref{eq:14}) holds with $g=1$ since
  $\dim\big(\cl{\sigma_1(Z)}\big)\ge\dim\big(\cl{\sigma_1(Z')}\big)$.
  Otherwise we apply Lemma~\ref{finite-tau} to our $Z'$
  and get an infinite subset $A\le G$.
  By assumption $\alpha$ does not centralise $A$,
  hence there is an element $g\in\alpha$
  which does not commute with $A$, i.e. $g\notin\CC_G(A)\cdot1$.
  Now $\cl{\tau_{(g,1)}(Z')}=\cl{\sigma_{g}(Z')}$ has positive dimension.
  But then the potentially larger set
  $\cl{\sigma_g(Z)}\supseteq\cl{\sigma_g(Z')}$
  has positive dimension as well.
  In all cases we proved \eqref{eq:14}, hence the corollary holds.
\end{proof}

\section{spreading large concentration in a group}
\label{sec:spre-large-conc}

In this section we establish our main technical tool, the Spreading Theorem.
Roughly speaking it says the following.
Let $\alpha$ be a finite subset
in a connected linear algebraic group $G$ such that
$\CC_G(\alpha)$ is finite.
If $G$ has a closed subset $X$ in which $\alpha$ has much larger concentration
than in $G$ then we can find a connected closed subgroup $H\le G$
normalised by $\alpha$
in which a small power of
$\alpha$ has similarly large concentration.
(When $G$ is the simple algebraic group used to define a 
finite group of Lie type $L$ and $\alpha$ generates $L$
then $H$ turns out to be $G$ itself.)

\begin{defn} \label{symmetric-set}
  A finite set $\alpha\subset GL(n,\Fclosed)$ is called \emph{symmetric} if
  $\alpha=\alpha^{-1}$.
\end{defn}

We need the following basic facts.

\begin{prop} \label{coset-vs-subgroup}
  Let $\alpha\subset GL(n,\Fclosed)$ be a symmetric subset
  and $hH$ a coset of a closed subgroup $H\le GL(n,\Fclosed)$.
  If $hH\cap \alpha\neq\emptyset$ then
  $$
  \mu(\alpha^2,hH)\ge\mu(\alpha,H) \;,\quad
  \mu(\alpha^2,H)\ge\mu(\alpha,hH)\;.
  $$  
\end{prop}
\qed

In the rest of this paper we restrict our attention to connected linear
algebraic groups. It is not a serious restriction
in the light of the following:

\begin{cor}
  Let $G\le GL(n,\Fclosed)$ be a closed subgroup and
  $1\in\alpha\subset GL(n,\Fclosed)$ a finite symmetric subset.
  Then
  $$
  \mu(\alpha,G^0)\le\mu(\alpha,G)\le\mu(\alpha^2,G^0)+\log\big(\deg(G)\big)
  \;.
  $$
\end{cor}
\begin{proof}
  It follows from Fact~\ref{normaliser-centraliser-fact}.\eqref{item:20}
  and Proposition~\ref{coset-vs-subgroup}.
\end{proof}

\begin{defn} \label{spreading-system}
  A \emph{spreading system} $\alpha|G$ consists of 
  a connected closed subgroup $G\le GL(n,\Fclosed)$,
  an ordered finite symmetric subset $1\in\alpha\subset GL(n,\Fclosed)$
  normalising $G$ such that $\mu(\alpha,G)\ge 0$
  and $\CC_G(\alpha)$ is finite.
  \par
  \noindent We say that $\alpha|G$ is \emph{$(N,\Delta,K)$-bounded}
  for some integer $N>0$ and reals $\Delta>0$, $K>0$
  if
  $$
  \dim(G)\le N \:,\quad
  \deg(G)\le\Delta \:,\quad
  \mult(G)\le\Delta \:,\quad
  \inv(G)\le\Delta \:,\quad
  \big|\alpha\cap G\big|\ge K \:.
  $$
  We say that $\alpha|G$ is \emph{$(\e,M,\delta)$-spreading}
  for some reals $\e>0$, $\delta>0$ and integer $M>0$,
  if there is a connected closed subgroup $H\le G$
  normalised by $\alpha$ such that $\dim(H)>0$ and
  $$
  \deg(H)\le\delta\;,\quad
  \mu\big(\alpha^M,H\big)\ge(1+\e)\cdot\mu(\alpha,G) \;.
  $$
  Note, that $\mult(H)$ and $\inv(H)$ are also bounded
  in terms of $\delta$ and $\Delta$
  by Proposition~\ref{bounding-subgroups}.
  We call such an $H$ a \emph{subgroup of spreading},
  or sometimes \emph{subgroup of $(\e,M,\delta)$-spreading}.
\end{defn}

\begin{rem}
  Note that the assumption $\mu(\alpha,G)\ge0$ is equivalent to 
  $\dim(G)>0$ and $\alpha\cap G\neq\emptyset$.
\end{rem}

Suppose that for some $m\ge0$ we find a closed subset
$Z\subseteq \Prod^{m}G$ in which $\Prod^m\alpha$ has large concentration.
We use the following lemma to find a closed subset of $G$
in which the concentration
of a small power of $\alpha$ is almost as large.

\newcommand{\backtoG} {{\text{b}}}
\begin{lem} [Back to $G$]
  \label{project-back-to-G}
  For all parameters $N>0$ and $\Delta>0$
  there are reals
  $B=B_\backtoG(N,\Delta)>0$ and
  $K=K_\backtoG(N,\Delta)\ge0$ 
  with the following property.\\
  Let $\alpha|G$ be a spreading system
  with $\dim(G)\le N$, $\deg(G)\le\Delta$ and $\mult(G)\le\Delta$.
  Then for all closed subsets $Z\subset\Prod^{m}G$
  with $0< m\le N$, $\dim(Z)>0$, $\deg(Z)\le \Delta$
  and $\big|\Prod^m\alpha\cap Z\big|\ge K$
  there is a closed subset $Y\subseteq G$
  such that $\dim(Y)>0$, $\deg(Y)\le B$ and
  $$
  \mu\big(\alpha^{3N},Y\big) \ge
  \mu\big(\Prod^{m}\alpha,Z\big) - \log(B) \;.
  $$
  Moreover, our construction of $Y$ is uniquely determined.
\end{lem}

\begin{proof}
  There is nothing to prove for $m=1$ , so we assume $m\ge2$.
  We prove the lemma by induction on $\dim(Z)$.
  This is possible, since $\dim(Z)\le N^2$, so the induction has at most $N^2$
  steps.
  We assume that the lemma holds in dimensions smaller than $\dim(Z)$
  with some bounds $B'\big(N,\Delta,\dim(Z)\big)$ and
  $K'\big(N,\Delta,\dim(Z)\big)$ .
  By Lemma~\ref{something-irreducible-with-large-concentration}
  if $K$ is large enough then
  there is a (uniquely determined) positive dimensional irreducible closed set
  $Z'\subseteq Z$ 
  of degree $\deg(Z')\le B_\irredComp(N^2,\Delta)$ with large concentration:
  $$
  \mu\big(\Prod^{m}\alpha,Z'\big) \ge  
  \mu\big(\Prod^{m}\alpha,Z\big) -
  \log\big(B_\irredComp(N^2,\Delta)\big) \;.
  $$
  This implies immediately that
  $$
  \Big|\Prod^m\alpha\cap Z'\Big|\ge
  \frac{|\Prod^m\alpha\cap Z|^{\dim(Z')/\dim(Z)}}
  {B_\irredComp(N^2,\Delta)^{\dim(Z')}} \ge
  \frac{K^{1/N^2}}{B_\irredComp(N^2,\Delta)^{N^2}} \;.
  $$
  By the above it is enough to complete the induction step for $Z'$,
  so from now on we assume that $Z$ is irreducible.
  Corollary~\ref{positive-dimensional-product} gives us a 
  (uniquely determined) sequence
  $\ug=(g_1,g_2,\dots, g_{m})\in\Prod^{m}\alpha$
  such that $\cl{\tau_\ug(Z)}$ has positive dimension.
  Recall from Remark~\ref{tau-rem} the bound
  $\Delta_\tau(N,N,\Delta)\ge\deg\big(\tau_\ug)$.
  Let
  $$
  \tilde\Delta=\max\big(\Delta,\Delta_\tau(N,N,\Delta)\big) \;.
  $$
  
  We use Lemma~\ref{Transport}
  for the two closed sets $Z\subseteq X=\Prod^{m}G$, 
  the morphism $f=\tau_\ug$,
  the finite set $\Prod^{m}\alpha$ 
  (denoted by $\alpha$ in Lemma~\ref{Transport})
  and $\e=0$.
  We note that $\tau_\ug\big(\Prod^{m}\alpha\big)\subseteq\alpha^{3N}$.
  There are two possible outcomes.
  In case of Lemma~\ref{Transport}.\eqref{eq:10} 
  the closed subset 
  $T=\cl{\tau_\ug(Z)}\subseteq G$
  satisfies $\dim(T)>0$, 
  $$
  \mu\big(\Prod^{m}\alpha,Z\big)-\log\big(B_\transport(\tilde\Delta)\big) \le
  \mu\big(\tau_\ug\big(\Prod^{m}\alpha\big),T\big) \le
  \mu\big(\alpha^{3N},T\big)
  $$
  and by Fact~\ref{closed-set-constructions}.\eqref{item:11}
  there is an upper bound $\deg(T)\le D$ depending only on $N$ and $\Delta$.
  Hence the lemma holds now with $Y=T$ and any 
  $B\ge \max\big(B_\transport(\tilde\Delta), D\big)$.
  In case of Lemma~\ref{Transport}.\eqref{eq:11}
  we have a closed subset $S\subseteq Z\subseteq\Prod^{m}G$
  with  $0<\dim(S)<\dim(Z)$, $\deg(S)\le B_\transport(\tilde\Delta)$ and
  $$
  \mu\big(\Prod^{m}\alpha,S\big) \ge
  \mu\big(\Prod^{m}\alpha,Z\big)-
  \log\big(B_\transport(\tilde\Delta)\big) \;.
  $$
  This implies immediately that
  $$
  \Big|\Prod^m\alpha\cap S\Big|\ge
  \frac{|\Prod^m\alpha\cap Z|^{\dim(S)/\dim(Z)}}
  {B_\transport(\tilde\Delta)^{\dim(S)}} \ge
  \frac{K^{1/N^2}}{B_\transport(\tilde\Delta)^{N^2}}
  $$
  that is, we can make $\Prod^m\alpha\cap S$ sufficiently large by choosing
  $K$ large enough.
  We set $B''=B'\big(N,B_\transport(\tilde\Delta),\dim(Z)\big)$ and
  apply the induction hypothesis to this $S$.
  This gives us a closed set $Y\subseteq G$ such that
  $\dim(Y)>0$, $\deg(Y)\le B''$ and
  $$
  \mu\big(\alpha^{3N},Y\big) \ge
  \mu\big(\Prod^{m}\alpha,S\big)-\log(B'') \ge
  $$
  $$
  \ge
  \mu\big(\Prod^{m}\alpha,Z\big)-
  \log\big(B_\transport(\tilde\Delta)B''\big)\;,
  $$
  the lemma holds again with the bound $B=B_\transport(\tilde\Delta)B''$.
  The induction step is complete now,
  the lemma holds in dimension $\dim(Z)$.
\end{proof}

We are now ready to prove the Spreading Theorem.
Let us first give an outline of the proof which avoids technicalities.
Suppose that $\alpha$ has ``large'' concentration in a subset
$X\subseteq G$. We would like to ``spread'' this large
concentration as much as possible,
i.e. we are looking for a small power $\alpha^M$ having
large concentration in a subgroup $H$
(more precisely, we need a subgroup of spreading $H$).

We start with $T_0=X$ and proceed with a simple induction.
Proposition~\ref{normal-subgroup-in-product} gives us a surjective morphism
$\tau_\ug$ which maps $Z=\Prod^{2\dim(G)}(X^{-1}\times X)$
(the direct product of $2\dim(G)$ copies of
the direct product $(X^{-1}\times X)$)
onto a subgroup $H\le G$.
The concentration of the product set $\Prod^{4\dim(G)}\alpha$
is large in  $Z$,
and we try to transport it via $\tau_\ug$ into $H$.
Note, that our $\tau_\ug$ maps $\Prod^{4\dim(G)}\alpha$ into a small
power $\alpha^m$.
According to the Transport Lemma~\ref{Transport}
we either succeed and therefore $H$ is a subgroup of spreading,
or find a subset $S\subseteq Z$ with significantly
larger concentration. This $S$ lives in the direct product
$\Prod^{4\dim(G)}G$,
but Lemma~\ref{project-back-to-G} brings it back to $G$,
i.e. we find a subset $T_1\subseteq G$ such that
a small power $\alpha^{m_1}$ has significantly larger
concentration in $T_1$ than $\alpha$ had in $T_0$
(see Lemma~\ref{try-to-spread-all-over}).

We repeat this process several times.
Either at some point we quit the induction with a subgroup of
spreading $H$
or we obtain a sequence of subsets $T_0,T_1,\dots$
with a quickly growing sequence of concentrations $\mu(\alpha^{m_i},T_i)$.
If we let the concentration grow sufficiently large
i.e. $\mu(\alpha^m,T_i)\ge\dim(G)\mu(\alpha,X)$
for some $i$
then already in $T_i$ there are enough elements
to force large concentration in $G$.
Therefore we either quit the induction with a subgroup of spreading,
or in a bounded number of steps we conclude that $\mu(\alpha^{m_i},G)$
is large i.e. $G$ itself is a subgroup of spreading.

\newcommand{\trytospread} {{\text{t}}}
\begin{lem} [Try to Spread]
  \label{try-to-spread-all-over}
  For all parameters $N>0$ and $\Delta>0$
  there is an integer
  $M_\trytospread=M_\trytospread(N)$, and
  there are reals
  $B_\trytospread=B_\trytospread(N,\Delta)>0$ and
  $K=K_\trytospread(N,\Delta)\ge0$ 
  with the following property.\\
  Let $\alpha|G$ be a spreading system with
  $\dim(G)\le N$, $\deg(G)\le\Delta$,
  $\mult(G)\le\Delta$  and $\inv(G)\le\Delta$.
  Then for all closed subsets $Y\subset G$
  with $\dim(Y)>0$, $\deg(Y)\le \Delta$ and $|\alpha\cap Y|\ge K$
  and all values
  $$
  \kappa\ge \log(B_\trytospread)
  $$
  at least one of the following holds:\\
  Either there is a connected closed subgroup $H\le G$
  normalised by $\alpha$
  such that $\dim(H)>0$, $\deg(H)\le B_\trytospread$ and
  \begin{equation}
    \label{eq:15}
    \mu\big(\alpha^{M_\trytospread},H\big)\ge
    \mu(\alpha,Y)-\kappa\;,
  \end{equation}
  or there is a closed set $T\subseteq G$
  such that $\deg(T)\le B_\trytospread$, $\dim(T)>0$ and
  \begin{equation}
    \label{eq:16}
     \mu\big(\alpha^{M_\trytospread},T\big) \ge 
     \mu(\alpha,Y)+\frac{\kappa}{8N^2}\;.
  \end{equation}
  Moreover, our constructions of $H$ and $T$ are uniquely determined.
\end{lem}

\begin{proof}
  Using Lemma~\ref{something-irreducible-with-large-concentration}
  as in the proof of Lemma~\ref{project-back-to-G},
  we may assume that $Y$ is irreducible.
  Let us recall from Lemma~\ref{Transport},
  Lemma~\ref{project-back-to-G},
  Remark~\ref{tau-rem} and Proposition~\ref{normal-subgroup-in-product}
  the functions $B_\transport$, $B_\backtoG$, $\Delta_\tau$
  and $\delta_\largeProduct$.
  We define the following parameters:
  $$
  \begin{array}{lcl}
    m   &=& N \\
    \Delta_1 &=& \max\big(\Delta^{6m},\Delta_\tau(4m,N,\Delta)\big) \\
    \big.B_\transport &=& B_\transport(\Delta_1) \\
    \big.\Delta_2   &=& \max(\Delta,B_\transport) \\
    \big.B_\backtoG &=& B_\backtoG(4m,\Delta_2) \\
    \big.\e  &=& \frac{\kappa}{8mN}+\log(B_\transport) +\log(B_\backtoG) \\
    \\
    M_\trytospread &=& \max(4m^2,12N)\\ 
    \big.B_\trytospread &=&
    \max\left(\delta_\largeProduct(N,\Delta),\,
    B_\transport^{8m(N+1)}\cdot B_\backtoG^{8m(N+1)}\right)
  \end{array}
  $$
  We apply Proposition~\ref{normal-subgroup-in-product}.
  to the subset $Y$, 
  this gives us a sequence 
  $\ug=(g_1,g_2,\dots, g_{2m})\in\Prod^{2m}\alpha^{m-1}$
  and a connected closed subgroup $H\le G$ normalised by $\alpha$
  such that $\dim(H)>0$, 
  $\deg(H)\le\delta_\largeProduct(N,\Delta)\le B_\trytospread$ and
  $$
  \tau_\ug\big(\Prod^{2m}Y^{-1}Y\big) = H \;.
  $$
  We apply Lemma~\ref{Transport}
  with parameters $\Delta_1$ and $\e$
  to the subsets $X=\Prod^{4m}G$ and $Z=\Prod^{2m}(Y^{-1}\times Y)$,
  the morphism $f=\tau_{(g_1,g_1,g_2,g_2,\dots,g_{2m},g_{2m})}$,
  the finite set $\Prod^{4m}\alpha^{m-1}$
  (denoted by $\alpha$ in Lemma~\ref{Transport}).
  We need to check that all requirements are satisfied.
  By assumption $\dim(Y)>0$ and hence $\dim(H)=\dim\big(f(Z)\big)>0$.
  Since $Y$ is irreducible, $Z$ is also irreducible
  (see Fact~\ref{dimension-degree}.\eqref{item:8})
  with $\deg(Z)=\deg(Y)^{4m}\inv(G)^{2m}\le\Delta^{6m}$
  (see Fact~\ref{dimension-degree}.\eqref{item:6} and
  Fact~\ref{closed-set-constructions}.\eqref{item:15})
  and
  $\deg(f)\le\Delta_\tau(4m,N,\Delta)$.
  Therefore the prerequisites of
  Lemma~\ref{Transport} 
  are satisfied, hence one of the inequalities
  \ref{Transport}.\eqref{eq:10} or
  \ref{Transport}.\eqref{eq:11} 
  is valid with the logarithmic term equal to $\log(B_\transport)$.
  Moreover,
  $\mu\big(\Prod^{4m}\alpha, Z\big) = \mu(\alpha,Y)$
  and
  $$
  f\big(\Prod^{4m}\alpha\big)\subseteq\alpha^{4m^2} \subseteq
  \alpha^{M_\trytospread} \;.
  $$
  In case of \ref{Transport}.\eqref{eq:10}
  we have
  $$
  \mu\big(\alpha^{M_\trytospread},H\big)\ge
  \mu\big(\alpha^{4m^2},f(Z)\big)\ge
  \mu\big(f\big(\Prod^{4m}\alpha\big),f(Z)\big) \ge
  $$
  $$
  \ge
  \mu\big(\Prod^{4m}\alpha,Z\big) - 
  \log(B_\transport) - \e\cdot\dim(Z) \ge
  $$
  $$
  \ge
  \mu(\alpha,Y) -\log(B_\transport) - 
  $$
  $$
  -
  \left(\frac{\kappa}{8mN}+\log(B_\transport) +\log(B_\backtoG) \right)
  \cdot N\cdot 4m \ge
  $$
  $$
  \ge
  \mu(\alpha,Y) - \frac{\kappa}{2} - 
    4m(N+1)\Big(\log(B_\transport) +\log(B_\backtoG) \Big) \ge
  $$
  $$
  \ge
  \mu(\alpha,Y) - \frac{\kappa}{2}-\frac{\log(B_\trytospread)}{2} \ge
  \mu(\alpha,Y) - \kappa
  $$
  which is exactly inequality~\eqref{eq:15}.

  In case of \ref{Transport}.\eqref{eq:11} 
  we have a closed subset $S\subseteq\Prod^{4m}G$
  with $\dim(S)>0$, $\deg(S)\le B_\transport$ such that
  $$
  \mu\big(\Prod^{4m}\alpha,S\big) \ge
  \mu\big(\Prod^{4m}\alpha,Z\big) -\log(B_\transport) + \e =
  $$
  $$
  =
  \mu\big(\alpha,Y\big)-\log(B_\transport)+
  \left(\frac{\kappa}{8mN}+\log(B_\transport) +\log(B_\backtoG) \right) =
  $$
  $$
  \ge
  \mu\big(\alpha,Y\big)+\
  \left(\frac{\kappa}{8N^2}+\log(B_\backtoG) \right) \;.
  $$
  In particular if $K=K_\trytospread(N,\Delta)$ is large enough then
  $$
  \big|\Prod^{4m}\alpha\cap S\big| \ge
  \big|\alpha\cap Y\big|^{\dim(S)/\dim(Y)}\ge
  K_\backtoG(4m,\Delta_2) \;.
  $$
  We apply Lemma~\ref{project-back-to-G}
  with the parameters $4N$ and
  $\Delta_2$
  (which are denoted there by $N$ and $\Delta$)
  to the set $S\subseteq\Prod^{4m}G$ 
  (which is denoted there by $Z$).
  Then in the inequalities we have to use $B_\backtoG=B_\backtoG(4m,\Delta_2)$.
  Lemma~\ref{project-back-to-G} gives us a subset
  $T\subseteq G$ (denoted there by $Y$)
  with $\dim(T)>0$, $\deg(T)\le B_\backtoG$ and
  $$
  \mu(\alpha^{12N},T) \ge
  \mu\big(\Prod^{4m}\alpha,S\big) - \log(B_\backtoG) \ge
  $$
  $$
  \ge
  \mu\big(\alpha,Y\big)+
  \left(\frac{\kappa}{8N^2}+\log(B_\backtoG)\right) -\log(B_\backtoG) =
  \mu\big(\alpha,Y\big) + \frac{\kappa}{8N^2}
  $$
  which implies inequality~\eqref{eq:16}.
  The lemma is proved in all cases.
\end{proof}

\newcommand{\spreading} {{\rm spreading}}
\begin{thm}[Spreading Theorem]
  \label{spread-all-over}
  For all parameters $N>0$, $\Delta>0$ and $\Frac13\ge\e>0$
  there is an integer
  $M=M_\spreading(N,\e)$
  and a real
  $K=K_\spreading(N,\Delta,\e)$
  with the following property.\\
  Let $\alpha|G$ be an
  $(N,\Delta,K)$-bounded spreading system
  and $X$ a closed subset in $\Prod^{m}G$
  for some $0< m\le N$.
  If $\deg(X)\le\Delta$, $\dim(X)>0$ and
  $$
  \mu\left(\Prod^{m}\alpha,X\right)\ge
  (1+3\e)\cdot\mu(\alpha,G)
  $$
  then $\alpha|G$ is $(\e,M,K)$-spreading.
  Moreover, our construction of the subgroup of spreading
  is uniquely determined.
\end{thm}

\begin{proof}
  Using Lemma~\ref{project-back-to-G} we can easily reduce the
  theorem to the special case of $m=1$,
  so we assume $X\subseteq G$.
  Let us recall from Lemma~\ref{try-to-spread-all-over} the functions
  $M_\trytospread$ and $B_\trytospread$.
  By induction on $i\ge0$ we shall define the following numbers:
  $$
  \Delta_0=\Delta
  \;,\quad
  \Delta_i=\max\big(\Delta_{i-1},B_\trytospread(N,\Delta_{i-1})\big)
  \;,\quad
  M_i=M_\trytospread(N)^{i} \;.
  $$
  Let $I=I(N,\e)$ be the smallest positive integer such that
  \begin{equation}
    \label{eq:17}
    \Big(1 + \frac{\e}{4N^2}\Big)^I\ge N \;.
  \end{equation}
  We set $M = M_I$ and
  $$
  K = \max\Big(\Delta_I^{N/\e},K_\trytospread(N,\Delta_{0})^N,
  K_\trytospread(N,\Delta_{1})^N,\dots,K_\trytospread(N,\Delta_{i-1})^N\Big)
  \;.
  $$
  Let $\alpha|G$ be an $(N,\Delta,K)$-bounded spreading system
  and $X\subseteq G$ a closed subset satisfying the conditions of
  the theorem.
  Then
  $$
  \mu(\alpha,X)>\mu(\alpha,G)\ge\frac{\log(K)}{N} \;.
  $$
  By induction on $i$ we build a series of 
  closed subsets $T_i\subseteq G$ such that
  \begin{equation}
    \label{eq:18}
    \left\{
      { \dim(T_i)>0
        \;,\quad
        \deg(T_i)\le \Delta_i 
        \;,}
      \atop
      {\mu\big(\alpha^{M_i},T_i\big)\ge 
        \Big(1 + \frac{\e}{4N^2}\Big)^i\cdot\mu(\alpha,X)
        \ge \Frac{\log K}{N}\;.}
    \right.
  \end{equation}
  We run the induction until we either prove Theorem~\ref{spread-all-over}
  or build the set $T_I$.
  We start the induction with $T_0=X$,
  this certainly satisfies \eqref{eq:18} with $i=0$.
  In the $i$-th step of the induction we assume that $T_{i-1}$ is already
  constructed and $i\le I$.

  We apply the Lemma~\ref{try-to-spread-all-over}
  with parameters $N$ and $\Delta_{i-1}$
  to the closed subset $Y=T_{i-1}$ and
  to the finite set $\alpha^{M_{i-1}}$ and
  $$
  \kappa = 
  \e\cdot\Big(1 + \frac{\e}{4N^2}\Big)^{i-1}\cdot\mu(\alpha,X) \;.
  $$
  We need to check that
  $
  \kappa \ge
  \e\cdot\mu(\alpha,X) \ge
  \Frac{\e}{N}\cdot\log(K) \ge \log(\Delta_I) \ge
  \log(\Delta_i) \ge \log\big(B_\trytospread(N,\Delta_{i-1})\big)
  $
  and
  $
  \big|\alpha^{M_{i-1}}\cap T_{i-1}\big| \ge
  \exp\big(\mu(\alpha^{M_{i-1}},T_{i-1})\big) \ge
  K^{1/N} \ge
  K_\trytospread(N,\Delta_{i-1})
  $.
  Note that
  $$
  \left(\alpha^{M_{i-1}}\right)^{M_\trytospread(N)} =
  \alpha^{M_{i}}\subseteq\alpha^{M} \;.
  $$
  There are two cases.
  If inequality~\ref{try-to-spread-all-over}.\eqref{eq:16}
  holds with a subset $T$ then 
  $$
  \mu\big(\alpha^{M_{i}},T\big) \ge
  \mu(\alpha^{M_{i-1}},T_{i-1}) + \frac{\kappa}{4N^2} \ge
  $$
  $$
  \Big(1 + \frac{\e}{4N^2}\Big)^{i-1}\cdot\mu(\alpha,X)
  + \frac{\e}{4N^2}\Big(1 + \frac{\e}{4N^2}\Big)^{i-1}\cdot\mu(\alpha,X) =
  $$
  $$
  =
  \Big(1 + \frac{\e}{4N^2}\Big)^{i}\cdot\mu(\alpha,X)
  $$
  and
  $
  \deg(T) \le B_\trytospread(N,\Delta_{i-1}) \le \Delta_i
  $
  hence $T_i=T$ satisfies the condition~\eqref{eq:18}.
  On the other hand,
  if inequality~\ref{try-to-spread-all-over}.\eqref{eq:15}
  holds with an appropriate subgroup $H$ then we find that
  $
  \deg(H) \le  B_\trytospread(N,\Delta_{i-1}) \le \Delta_i \le K
  $
  and
  $$
  \mu\big(\alpha^{M},H\big) \ge
  \mu\big(\alpha^{M_{i}},H\big) \ge
  \mu\big(\alpha^{M_{i-1}},T_{i-1}\big) - \kappa \ge
  $$
  $$
  \ge
  \Big(1 + \frac{\e}{4N^2}\Big)^{i-1}\cdot\mu(\alpha,X) -
  \e\cdot\Big(1 + \frac{\e}{4N^2}\Big)^{i-1}\cdot\mu(\alpha,X) \ge
  $$
  $$
  \ge
  (1-\e)\cdot\mu(\alpha,X) \ge
  (1-\e)(1+3\e)\mu(\alpha,G) \ge
  (1+\e)\mu(\alpha,G) \;.
  $$
  The theorem holds in this case and we stop the induction.

  Finally we consider the case when the induction does not stop 
  during the first $I$ steps and we build $T_I$.
  Using the first inequality from
  Proposition~\ref{concentration-is-bounded} and
  inequalities~\eqref{eq:18} and \eqref{eq:17} we obtain that
  $$
  \mu(\alpha^{M},G) \ge
  {\Frac{\dim(T_I)}{\dim(G)}}\cdot\mu(\alpha^{M},T_I) \ge
  $$
  $$
  \ge
  \Frac{1}{N}\cdot
  \left(1 + \Frac{\e}{4N^2}\right)^{I}\cdot\mu(\alpha,X) \ge
  \mu(\alpha,X) \ge
  (1+3\e)\mu(\alpha,G) \;.
  $$
  That is, $\alpha|G$ is $(\e,{M},K)$-spreading with $H=G$.
  The theorem holds in this case too.
\end{proof}

\section{Variations on spreading}

The following useful lemma shows that growth in a subgroup of $G$ implies
growth in $G$ itself. See \cite{He2} for similar results.

\begin{lem}
  \label{induce-from-growing-subgroup}
  Let  $A\le G\le GL(n,\Fclosed)$ be closed subgroups and\\
  $1\in\alpha\subset GL(n,\Fclosed)$ a finite subset. 
  Then for all integers $k>0$ one has
  $$
  \mu\big(\alpha^{k+1},G\big) \ge 
  \mu\big(\alpha,G\big) + 
  \Frac{\dim(A)}{\dim(G)}
  \Big[\mu\big(\alpha^k,A\big)-\mu\big(\alpha^{-1}\alpha,A\big)\Big]
  $$
  or equivalently
  $$
  \Frac{\big|\alpha^{k+1}\cap G\big|}{\big|\alpha\cap G\big|} \ge
  \Frac{\big|\alpha^k\cap A\big|}{\big|\alpha^{-1}\alpha\cap A\big|} \;.
  $$
\end{lem}

\begin{proof}
  The two inequalities are clearly equivalent, we shall prove the
  latter form.
  We shall look at the multiplication map
  $$
  (\alpha\cap G) \times \big(\alpha^k\cap A\big)
  \stackrel{\phi}{\longrightarrow}
  (\alpha\cap G)\cdot\big(\alpha^k\cap A\big) \subseteq
  \big(\alpha^{k+1}\cap G\big)
  $$
  On the left hand side we have
  $|\alpha\cap G|\cdot\big|\alpha^k\cap A\big|$
  elements, on the right hand side there are
  $\big|\alpha^{k+1}\cap G\big|$
  elements. Therefore it is enough to prove that
  $$
  \big|\phi^{-1}(g)\big|\le \big|\alpha^{-1}\alpha\cap A\big|
  \quad\text{for all }g\in\alpha^{k+1}\cap G
  $$
  and this follows from the calculation below:
  $$
  \phi^{-1}(g) \subseteq
  \big\{(a,a^{-1}g)\;\big|\;
  a\in\alpha,\;a^{-1}g\in A\big\} \subseteq
  \big\{(a,a^{-1}g)\;\big|\; a\in\alpha\cap gA\big\} \;,
  $$
  hence
  $$
  \big|\phi^{-1}(g)\big|\le
  \big|\alpha\cap gA\big| \le
  \big|(\alpha\cap gA)^{-1}(\alpha\cap gA)\big| \le
  \big|\alpha^{-1}\alpha\cap A\big| \;.
  $$
\end{proof}

The following result is closely related to the ``escape from subvarieties''
type results in \cite{He1} and \cite{He2}.

\newcommand{\Eescape} {{\ensuremath{\Frac1{7N^2}}}\xspace}
\newcommand{\escape} {{\rm escape}}
\begin{lem} [Escape Lemma]
  \label{escape}
  For all parameters $N>0$, $\Delta>0$ and $\Eescape\ge\e>0$
  there is an integer
  ${M}=M_\escape(N,\e)$
  and a real
  $K=K_\escape(N,\Delta,\e)$
  with the following property.\\
  Let $\alpha|G$ be an
  $(N,\Delta,K)$-bounded spreading system
  and $X\subsetneq Y$ two closed subsets in $\Prod^{m}G$
  for some $1\le m\le N$.
  Suppose that $\dim(Y)>0$, $Y$ is irreducible,
  $\deg(X)\le\Delta$ and
  $$
  \mu\left(\Prod^{m}\alpha,Y\right) \ge
  (1-\e)\cdot\mu(\alpha,G) \;,
  $$
  $$
  \mu\big(\Prod^{m}\alpha,\, Y\setminus X\big) \le
  (1-2\e)\cdot\mu(\alpha,G) \;.
  $$
  Then $\alpha|G$ is $(\e,{M},K)$-spreading.
  Moreover, our construction of the subgroup of spreading
  is uniquely determined.
\end{lem}

\begin{proof}
  We set
  $
  {M}=M_\escape(N,\e) =
  M_\spreading(N,\e)
  $
  and
  $$
  K=K_\escape(N,\Delta,\e) =
  \max\Big(K_\spreading(N,\Delta,\e),2^{N/\e},(2\Delta+1)^{N/(1-\e)}\Big)\;.
  $$
  Then
  $
  \mu(\alpha,G)\ge\Frac{\log(K)}{N}\ge \Frac{\log(2)}{\e}
  $
  and
  $$
  \log\left(
  \frac{\left|\Prod^{m}\alpha\cap Y\right|}
  {\left|\Prod^{m}\alpha\cap(Y\setminus X)\right|} \right) =
  \dim(Y)\Big(\mu\big(\Prod^{m}\alpha,Y\big)-
  \mu\big(\Prod^{m}\alpha,Y\setminus X\big)\Big) \ge
  $$
  $$
  \ge
  \dim(Y)\cdot\e\cdot\mu(\alpha,G\big) \ge
  \log(2) \;.
  $$
  Therefore
  $
  \left|\Prod^{m}\alpha\cap X\right| \ge
  \Frac12\left|\Prod^{m}\alpha\cap Y\right| \ge
  \Frac12|\alpha\cap G|^{(1-\e)\dim(Y)/\dim(G)} > \Delta
  $,
  hence $\dim(X)>0$ and
  $$
  \mu\left(\Prod^{m}\alpha,X\right) \ge
  \Frac{\dim(Y)}{\dim(X)}
  \mu\left(\Prod^{m}\alpha,Y\right)-\log(2) \ge
  $$
  $$
  \ge
  \left(1+\Frac1{\dim(X)}\right)(1-\e)\cdot\mu(\alpha,G) - \log(2)\ge
  (1+7\e)(1-\e)\cdot\mu(\alpha,G) - \log(2) \ge
  $$
  $$
  \ge
  \big(1+5\e\big)\cdot\mu(\alpha,G) -\e\cdot\mu(\alpha,G) >
  (1+3\e)\cdot\mu(\alpha,G) \;.
  $$
  Then $\alpha|G$ is $(\e,{M},K)$-spreading
  by the Spreading Theorem~\ref{spread-all-over}.
\end{proof}

\section{Centralisers}
\label{sec:centralisers}

If $G$ is a simple algebraic group then a maximal torus $T$ can be obtained as
the connected centraliser of a (regular semisimple) element.
Using this it follows that
if an appropriate subset $\alpha\subset G$ does not grow then the
concentration of a small power of
$\alpha$ in $T$ is at least $\mu(\alpha,G)$.
We first generalise this extremely useful result.
Then we define CCC-subgroups and establish some of their basic
properties.

Recall from Fact~\ref{normaliser-centraliser-fact}
that the degree of any centraliser subgroup is at most $\deg(G)$.

\newcommand{\centraliser} {{\rm c}}
\begin{lem} [Centraliser Lemma]
  \label{centraliser-lemma}
  For all parameters $N>0$, $\Delta>0$ and $1\ge\e>0$
  there is an integer
  ${M}=M_\centraliser(N,\e)$
  and a real
  ${K}=K_\centraliser(N,\Delta,\e)$
  with the following property.\\
  Let $\alpha|G$ be an
  $(N,\Delta,{K})$-bounded spreading system
  and $C=\CC_G(b_1,b_2,\dots, b_m)$ the centraliser of
  $m\le N$ elements $b_i\in\alpha\cap G$.
  If $0<\dim(C)$
  then either
  $$
  \mu\big(\alpha^{M},C^0\big)\ge
  \Big(1-\e\cdot8N\Big)\cdot\mu(\alpha,G)
  $$
  or $\alpha|G$ is $(\e,{M},K)$-spreading.
  Moreover, in the latter case
  our construction of the subgroup of spreading
  is uniquely determined.  
\end{lem}

\begin{proof}
  We set
  $
  {M}=M_\centraliser(N,\e)= \max\big(4,3M_\spreading(N,\e)\big)
  $,
  $
  \tilde\Delta = \max(\Delta,\Delta^{3m})
  $
  and
  $$
  {K}=K_\centraliser(N,\Delta,\e)= 
  \max\Big(\Delta^{1/\e}\,,\,
  \Delta\cdot K_\spreading\big(N,\tilde\Delta,\e\big) \Big) \;.
  $$
  Note that $\dim(C^0)=\dim(C)>0$ and 
  $\big|C:C^0\big|\le\Delta$
  by Fact~\ref{normaliser-centraliser-fact}.\eqref{item:20}.
  Combining this with Proposition~\ref{coset-vs-subgroup} we obtain that
  for some $h\in C$
  $$
  \mu\big(\alpha^{M},C^0\big)\ge
  \mu(\alpha^{M/2},hC^0)\ge
  \mu\big(\alpha^{M/2},{C}\big) - \log\big(\Delta\big)
  \;.
  $$
  Since ${K}>\big(\Delta\big)^{1/\e}$ we have
  $$
  \textstyle
  \mu(\alpha,G)>\frac1{\dim(G)}\log(K)\ge
  \frac1{\e\cdot\dim(G)}\log\big(\Delta\big) \ge
  \frac1{\e\cdot N}\log\big(\Delta\big) \;.
  $$
  By the above inequalities it is enough to prove that
  either $\alpha|G$ is $(\e,M,K)$-spreading or
  \begin{equation}
    \label{eq:19}
    \mu\big(\alpha^{M/2},{C}\big)\ge
    \Big(1-\e\cdot7N\Big)\cdot\mu(\alpha,G) \;.
  \end{equation}
  If $\dim({C})=\dim(G)$ then $G={C}$
  and there is nothing to prove.
  So we assume $\dim({C})<\dim(G)$ and
  apply Lemma~\ref{constant-fibre-dimension} to 
  the subsets $Z=G$ and $\alpha$ and to the function
  $$
  f:G\to\Prod^{m}G
  \;,\quad
  f(g) = 
  \big(\;g^{-1}b_1g,\;g^{-1}b_2g,\;\dots\; g^{-1}b_mg\;\big) \in 
  \Prod^{m}G
  $$
  with the parameter $\e'=-7\e\frac{\mu(\alpha,G)}{\dim({C})}$.
  The fibres of $f$ are just the right cosets of the subgroup ${C}$,
  which have equal dimension, hence we obtain a coset $S={C}a$
  that satisfies inequality \eqref{eq:9}:
  either
  $$
  \mu(\alpha,G) \le
  \mu(\alpha, {C}a) +
  7\e\Frac{\mu(\alpha,G)}{\dim({C})}
  \big(\dim(G)-\dim({C})\big) \le
  $$
  $$
  \le
  \mu(\alpha, {C}a) +
  \e\cdot 7\dim(G)\cdot\mu(\alpha,G) \le
  \mu\big(\alpha^2,{C}\big)+
  \e\cdot 7N\cdot\mu(\alpha,G)
  $$
  (see Proposition~\ref{coset-vs-subgroup})
  and the inequality~\eqref{eq:19} holds in this case,
  or else
  $$
  \mu(\alpha,G)\le 
  \mu\big(f(\alpha\cap G),\cl{f(G)}\big) - 
  \Frac{7\e\cdot\mu(\alpha,G)}{\dim({C})}\dim({C}) =
  $$
  $$
  =
  \mu\big(f(\alpha\cap G),\cl{f(G)}\big) - 7\e\cdot\mu(\alpha,G)\;.
  $$
  We know $f(\alpha\cap G)\subseteq\Prod^{m}\alpha^3$ hence
  in this latter case we have
  $$
  \mu\big(\Prod^{m}\alpha^3,\cl{f(G)}\big)\ge
  (1+7\e)\cdot\mu(\alpha,G) \;.
  $$
  If $\mu(\alpha^3,G)\ge(1+\e)\mu(\alpha,G)$ then we are done.
  Otherwise
  $$
  (1+3\e)\mu(\alpha^3,G)\le
  (1+3\e)(1+\e)\mu(\alpha,G)\le
  $$
  $$
  \le
  (1+7\e)\mu(\alpha,G)\le
   \mu\big(\Prod^{m}\alpha^3,\cl{f(G)}\big) \;. 
  $$
  Now $\deg\big(\cl{f(G)}\big)\le\tilde\Delta$
  (see Fact~\ref{degree-of-group-theoretic-maps}).
  We apply the Spreading Theorem~\ref{spread-all-over}
  with parameters $N$, $\tilde\Delta$ and $\e$
  to the spreading system $\alpha^3|G$ and $X=\cl{f(G)}$.
  We obtain that $\alpha^3|G$ is $(\e,\Frac13{M},K)$-spreading,
  hence $\alpha|G$ is $(\e,{M},K)$-spreading.
\end{proof}

\begin{defn} \label{CC-generator-def}
  Let $G$ be an algebraic group and
  $X\subseteq G$ an irreducible closed subset. 
  A \emph{CC-generator}\footnote{CC refers to ``connected centraliser''}
  for $X$ is a $\dim(G)$-tuple
  $\ug\in\Prod^{\dim(G)}X$
  such that
  $$
  \CC_G(\ug)^0 = \CC_G(X)^0 \;.
  $$
  Let $X^\gen\subseteq\Prod^{\dim(G)}X$ denote
  the set of all CC-generators and let
  $X^\nongen=\big(\Prod^{\dim(G)}X\big)\setminus X^\gen$ denote the
  complement.
\end{defn}

Note that $X^\gen$ depends on the group $G$,
but for simplicity we suppressed it from the notation.
When we work with a spreading system $\alpha|G$
then we always define $X^\gen$ with respect to $G$.

\begin{prop} \label{CC-generators-exist}
  Let $G$ be an algebraic group and
  $X\subseteq G$ an irreducible closed subset. 
  Then $X$ has a CC-generator i.e. $X^\gen\neq\emptyset$.
\end{prop}

\begin{proof}
  We consider sequences $\ua=a_1,a_2,\dots, a_m$, $a_i\in X$
  such that
  $$
  G>\CC_G(a_1)^0>\CC_G(a_1,a_2)^0>\CC_G(a_1,a_2,a_3)^0>\dots
  $$
  is a strictly decreasing chain of subgroups.
  The dimension is strictly decreasing in such a chain,
  hence the length of $\ua$ is $m\le\dim(G)$.
  Therefore one of them, say $\ua_{\max}$, is maximal
  i.e. it cannot be extended.
  But then
  $$
  \CC_G(X)^0=\CC_G(\ua_{\max})^0
  $$
  and we can build a CC-generator from $\ua_{\max}$
  by adding to it $\dim(G)-m$ arbitrary elements of $X$.
\end{proof}

\begin{prop} \label{centraliser-semi-continuity}
  Let $G$ be a connected linear algebraic group, $X$ an irreducible closed set
  and $G\times X\to X$ a morphism which is a group action.
  For points $x\in X$ let $G_x$ denote the stabiliser subgroup of $x$.
  These are closed subgroups and
  for each integer $d$ the subset $\{x\in X\,|\,\dim(G_x)>d\}$
  is closed in $X$.
  In particular, for each $d$ the points $\ug\in\Prod^{\dim(G)}G$
  with $\dim\big(\CC_G(\ug)\big)>d$ form a closed subset in $\Prod^{\dim(G)}G$.
\end{prop}

\begin{proof}
  For the first half of the proposition (about stabiliser subgroups)
  we refer to \cite[Proposition~in~1.4]{Hu2}.
  If we apply this to the conjugation map
  $$
  G\times\Prod^{\dim(G)}G\to
  \Prod^{\dim(G)}G
  \:,\quad
  (h,\ug) \to h^{-1}\ug h
  $$
  then we obtain the second half (about centraliser subgroups).
\end{proof}

\begin{lem} \label{CC-generator-properties}
  Let $G$ be a connected linear algebraic group and 
  $\emptyset\neq X\subseteq G$ an irreducible closed subset.
  Then $X^\gen$ is a dense open subset of $\Prod^{\dim(G)}X$.
  Moreover, the degree of its complement $X^\nongen$
  is bounded in terms of $\dim(G)$, $\deg(G)$, $\mult(G)$, $\inv(G)$
  and $\deg(X)$.
\end{lem}

\begin{proof}
  First of all  $X^\nongen=\big\{\ug\,\big|\,\dim(\CC_G(\ug))>\dim(A)\big\}$
  is closed by Proposition~\ref{centraliser-semi-continuity}.
  Its complement $X^\gen$ is naturally open, it is nonempty by
  Proposition~\ref{CC-generators-exist},
  hence it is dense (see Fact~\ref{dimension-degree}.\eqref{item:7}).
 
  Let us consider the conjugation map
  $$
  f:G\times\Prod^{\dim(G)}X\to
  \Prod^{\dim(G)}G\times\Prod^{\dim(G)}X
  \:,\quad
  f(h,\ug) = \big(h^{-1}\ug h,\ug\big) \;.
  $$
  Let $Y$ denote the diagonal subset
  $$
  Y=
  \left\{(\ug,\ug)\,\big|\,\ug\in\Prod^{\dim(G)}X\right\}
  \subset\Prod^{\dim(G)}G\times\Prod^{\dim(G)}X
  $$
  and let $\tilde f$ denote the restriction of $f$ to 
  $f^{-1}(Y)$ composed with the second projection $Y\to \Prod^{\dim(G)}X$.

  The nonempty fibres of $f$ can be easily identified with cosets of
  appropriate centraliser subgroups.
  Namely, if $f^{-1}(\ug',\ug)\neq\emptyset$ then
  $\ug'=h^{-1}\ug h$ for some element $h\in G$
  and 
  $$
  f^{-1}(\ug',\ug) = \CC_G(\ug)h\times\{\ug\} \;.
  $$
  All of the involved centralisers contain the subgroup
  $$
  A=\CC_G(X)^0
  $$
  and by Proposition~\ref{CC-generators-exist}
  at least one of them has dimension $\dim(A)$.
  For $\ug\in\Prod^{\dim(G)}X$ we have
  $\ug\in X^\nongen$ iff
  $\dim\big(f^{-1}(\ug,\ug)\big)>\dim(A)$. 
  By Fact~\ref{closed-set-constructions}.\eqref{item:14} the subset
  $$
  Z = \left\{t \;\Big|\;
    \dim\Big(f^{-1}\big(f(t)\big)\Big)>\dim(A) \right\}
  \subseteq G\times\Prod^{\dim(G)}X
  $$
  is a closed subset and $\deg(Z)$ is bounded in terms of
  $\dim(G)$, $\deg(G)$, $\mult(G)$, $\inv(G)$ and $\deg(X)$.
  By the above $\tilde f\big(Z\cap f^{-1}(Y)\big)=X^\nongen=\cl{X^\nongen}$.
  By  Fact~\ref{closed-set-constructions}.\eqref{item:15}
  and Fact~\ref{dimension-degree}.\eqref{item:6}
  we see that
  $\deg(X^\nongen)=
  \deg\big(\cl{f(Z)}\cap Y\big)\le
  \deg(f)\cdot\deg(Z)\cdot\deg(Y)$
  which is bounded in terms of
  $\dim(G)$, $\deg(G)$, $\mult(G)$, $\inv(G)$ and $\deg(X)$.
\end{proof}

\begin{defn} \label{CCC-subgroup-def}
  Let $G$ be an algebraic group.
  A closed subgroup
  $A<G$ is a \emph{CCC-subgroup}\footnote
  {CCC refers to ``connected centraliser of a connected subgroup''} 
  if $A=\CC_G(X)^0$ for some irreducible closed subset $X\ni1$
  and $A$ is different from $\{1\}$ and $G^0$.
\end{defn}

\begin{lem} \label{ccc-subgroup-properties}
  Let $G$ be an algebraic group and $A<G$ a CCC-subgroup.
  Then
  $$
  \CC_G\big(\CC_G(A)^0\big)^0 = A
  \;,\quad
  \deg(A)\le\deg(G)
  $$
  and $\deg\left(A^\nongen\right)$
  is bounded in terms of $\dim(G)$, $\deg(G)$, $\mult(G)$ and $\inv(G)$.
  If $B<G$ is another CCC-subgroup with $A\neq B$ then 
  $A^\gen\cap B^\gen=\emptyset$.
\end{lem}

\begin{proof}
  Let $1\in X\subseteq G$ be an irreducible closed subset such that
  $A=\CC_G(X)^0$.
  Then $X\subseteq\CC_G(A)^0$, $A$ is connected and
  commutes with $\CC_G(A)^0$,
  hence
  $$
  A=\CC_G(X)^0\supseteq\CC_G\big(\CC_G(A)^0\big)^0\supseteq A \;.
  $$
  Now $\deg(A)\le\deg(G)$
  by Fact~\ref{normaliser-centraliser-fact}
  and then Lemma~\ref{CC-generator-properties} implies that
  $\deg\left(A^\nongen\right)$
  is bounded in terms of $\dim(G)$, $\deg(G)$, $\mult(G)$ and $\inv(G)$.
  Finally if $\ug\in A^\gen$ then
  $\CC_G\big(\CC_G(\ug)^0\big)^0=A\neq B$
  hence $\ug\notin B^\gen$.
  This proves that $A^\gen\cap B^\gen=\emptyset$.
\end{proof}

\section{Dichotomy lemmas}
\label{sec:dichotomy-lemmas}

A central idea of the proof of Theorem~\ref{simple-L}
for $L=SL(n,q)$ (as outlined in the introduction)
is the following.
If a generating set $\alpha$ of $L$ does not grow then
the intersection of $\alpha$
with any maximal torus of $L$ is either relatively large or relatively
small.
This follows from a similar property of appropriate maximal tori
in $SL(n,\cl{\,\Fq})$.
Here we show that CCC-subgroups also satisfy a similar dichotomy.
In fact they were designed to do so.

We first prove that if a set $\alpha$ does not grow (or spread),
then for any closed set $Z$ either the intersection of $\alpha$ with $Z$ is
relatively small or a small power of $\alpha$ has relatively large
intersection with the centraliser of $Z$.

\newcommand{\asymdichotomy} {{\rm a}}
\begin{lem} [Asymmetric Dichotomy Lemma]
  \label{asymmetric-dichotomy}
  For all parameters $N>0$, $\Delta>0$ and $\Frac1{56N^3}>\e>0$
  there is an integer
  $M=M_\asymdichotomy(N,\e)$
  and a real
  $K=K_\asymdichotomy(N,\Delta,\e)$
  with the following property.\\
  Let $\alpha|G$ be an
  $(N,\Delta,K)$-bounded spreading system.
  Then either $\alpha|G$ is  $(\e,M,K)$-spreading
  or for all irreducible closed subsets $Z\subseteq G$ 
  such that $\dim(Z)>0$, $\deg(Z)<\Delta$ and $\dim\big(\CC_G(Z)\big)>0$
  one of the following holds:
  $$
  \mu\left(\alpha,Z\right) <
  \left(1-\Eescape\right)\cdot\mu(\alpha,G)
  $$
  or
  $$
  \mu\big(\alpha^M,\CC_G(Z)^0\big) \ge
  \mu\left(\Prod^{\dim(G)}\alpha^M,\;\big(\CC_G(Z)^0\big)^\gen\right) \ge
  \Big(1-\e\cdot16N\Big)\cdot\mu(\alpha,G) \;.
  $$
  Moreover, our construction of the subgroup of spreading
  is uniquely determined.
\end{lem}

\begin{proof}
  We define the parameters
  $$
  \e'=\Eescape
  \;,\quad
  \e'''= \e\cdot8N \le \Eescape
  $$
  and the closed subsets
  $$
  Y'=\Prod^{\dim(G)}Z\quad\supseteq\quad
  X'=Z^\nongen \;\,\quad
  $$
  $$
  Y'''=\Prod^{\dim(G)}\CC_G(Z)^0\quad\supseteq\quad
  X'''=\big(\CC_G(Z)^0\big)^\nongen\;.
  $$
  We know from Fact~\ref{normaliser-centraliser-fact}
  that $\deg\big(\CC_G(Z)^0\big)\le\Delta$.
  By Lemma~\ref{ccc-subgroup-properties}
  there is an upper bound $\tilde\Delta\ge\Delta$
  for $\deg(X')$ and $\deg(X''')$
  which depends only on $N$ and $\Delta$.
  We set
  $
  M'' = M_\centraliser(N,\e) \;,
  $
  $$
  M = M_\asymdichotomy(N,\e) =
  \max \Big(
  M_\escape(N,\e'),\, M'',\, M''\cdot\, M_\escape(N,\e''')
  \Big)
  $$
  and
  $$
  K = K_\asymdichotomy(N,\Delta,\e) =
  $$
  $$
  =
  \max \Big(
  K_\escape(N,\tilde\Delta,\e'),\, K_\centraliser(N,\Delta,\e),\,
  K_\escape(N,\tilde\Delta,\e''')
  \Big) \;.
  $$
  We apply the Escape Lemma~\ref{escape}
  with the parameters $N$, $\tilde\Delta$ and $\e'$
  to the subsets $X'$ and $Y'$.
  If the Escape Lemma~\ref{escape} gives us a
  subgroup of
  $\big(\e',M_\escape(N,\e'),K_\escape(N,\tilde\Delta,\e')\big)$-spreading
  then the lemma holds since $\e\le\e'$.
  Otherwise there are two possibilities.
  Either
  $$
  \mu(\alpha,Z) =
  \mu\left(\Prod^{\dim(G)}\alpha,Y'\right) <
  (1-\e')\cdot\mu(\alpha,G) =
  \left(1-\Eescape\right)\cdot\mu(\alpha,G)
  $$
  in which case the lemma holds,
  or else there is at least one $\dim(G)$-tuple
  $\ug\in\Prod^{\dim(G)}\alpha\cap Z^\gen$
  (in fact the Escape Lemma gives us many such tuples).
  We select the lexicographically minimal $\ug$ among them.
  Note that $\CC_G(\ug)^0=\CC_G(Z)^0\neq\{1\}$,
  in particular $\dim\big(\CC_G(\ug)\big)>0$.
  In this latter case we apply the
  Centraliser Lemma~\ref{centraliser-lemma}
  with parameters $N$, $\Delta$ and $\e$
  to the spreading system $\alpha|G$
  and the subgroup $C=\CC_G(\ug)$.
  In case we obtain a subgroup of spreading, the lemma holds.
  Otherwise we have
  $$
  \mu\big(\alpha^{M''},\CC_G(Z)^0\big) \ge
  \big(1-\e\cdot8N\big)\cdot\mu(\alpha,G) =
  \left(1-\e'''\right)\cdot\mu(\alpha,G) \;.
  $$
  Finally we apply the Escape Lemma~\ref{escape}
  with parameters $N$, $\tilde\Delta$ and $\e'''$
  to the spreading system $\alpha^{M''}|G$ and the subsets $X'''$ and $Y'''$.
  Again, the lemma holds if we obtain a subgroup of spreading.
  Otherwise we have
  $$
  \mu\left(\Prod^{\dim(G)}\alpha^{M''},\,\big(\CC_G(Z)^0\big)^\gen\right) >
  (1-2\e''')\cdot\mu(\alpha,G) =
  \Big(1-\e\cdot16N\Big)\mu(\alpha,G) \;.
  $$
  Then the lemma follows from Proposition~\ref{concentration-is-bounded}
  via the following calculation:
  $$
  \mu\big(\alpha^{M},\CC_G(Z)^0\big) =
  \mu\big(\Prod^{\dim(G)}\alpha^{M},Y'''\big) \ge
  $$
  $$
  \ge
  \mu\big(\Prod^{\dim(G)}\alpha^{M},(Y'''\setminus X''')\big) =
  \mu\left(\Prod^{\dim(G)}\alpha^{M},\,\big(\CC_G(Z)^0\big)^\gen\right) \;.
  $$
\end{proof}

The connected centraliser of the connected centraliser of a
CCC-subgroup $A$ is $A$ itself,
hence applying the previous lemma twice we obtain the following.

\newcommand{\dichotomy} {{\rm dichotomy}}
\begin{lem} [Dichotomy Lemma]
  \label{dichotomy-lemma}
  For all parameters $N>0$, $\Delta>0$ and $\Frac1{112N^3}>\e>0$
  there is an integer
  $M=M_\dichotomy(N,\e)$
  and a real
  $K=K_\dichotomy(N,\Delta,\e)$
  with the following property.\\
  Let $\alpha|G$ be an
  $(N,\Delta,K)$-bounded spreading system.
  Then either $\alpha|G$ is  $(\e,M,K)$-spreading
  or for all CCC-subgroups $A<G$
  one of the following holds:
  $$
  \mu\left(\alpha,A\right) <
  \left(1-\Eescape\right)\cdot\mu(\alpha,G)
  $$
  or else
  $$
  \mu\big(\alpha^M,A\big) \ge
  \mu\left(\Prod^{\dim(G)}\alpha^M,\;A^\gen\right) \ge
  \Big(1-\e\cdot16N\Big)\cdot\mu(\alpha,G) \;.
  $$
  Moreover, our construction of the subgroup of spreading
  is uniquely determined.
\end{lem}

\begin{proof}
  We set
  $
  M'=M_\asymdichotomy(N,\e)
  \;,\quad
  M = M_\dichotomy(N,\e) =
  \left(M'\right)^2
  $
  and
  $$
  K = K_\dichotomy(N,\Delta,\e) =
  K_\asymdichotomy(N,\Delta,\e)
  \;.
  $$
  We apply the Asymmetric Dichotomy Lemma~\ref{asymmetric-dichotomy}
  with parameters $N$, $\Delta$ and $\e$ to $\alpha|G$
  and the irreducible subset $Z'=A$.
  Note that $\dim(A)>0$ and $\dim\big(\CC_G(A)\big)>0$ follows from
  Definition~\ref{CCC-subgroup-def}.
  If we obtain a subgroup of $(\e,M',K)$-spreading or if
  $$
  \mu\left(\alpha,A\right) <
  \left(1-\Eescape\right)\cdot\mu(\alpha,G)
  $$
  then the lemma holds.
  Otherwise we have
  $$
  \mu\big(\alpha^{M'},\CC_G(A)^0\big) \ge
  \Big(1-\e\cdot16N\Big)\cdot\mu(\alpha,G) \;.
  $$
  We apply again
  the Asymmetric Dichotomy Lemma~\ref{asymmetric-dichotomy}
  with parameters $N$, $\Delta$ and $\e$ to
  $\alpha^{M'}|G$ and $Z''=\CC_G(A)^0$.
  If we obtain a subgroup of $(\e,M',K)$-spreading
  then it is a subgroup of $(\e,M,K)$-spreading for $\alpha|G$
  and the lemma holds.
  Otherwise $\alpha^{M'}|G$ and $Z''$ must satisfy
  one of the two inequalities of that lemma.
  The first one is
  $$
  \mu\left(\alpha^{M'},\CC_G(A)^0\right) <
  \left(1-\Eescape\right)\cdot\mu(\alpha,G) \le
  \Big(1-\e\cdot16N\Big)\cdot\mu(\alpha,G) \;,
  $$
  but this has already been ruled out.
  Therefore the
  other inequality holds:
  $$
  \mu\left(\big(\alpha^{M'}\big)^{M'},\,\CC_G\big(\CC_G(A)^0\big)^0\right) \ge
  $$
  $$
  \ge
  \mu\left(\Prod^{\dim(G)}\alpha^{M'\cdot M'},\,
    \left(\CC_G\big(\CC_G(A)^0\big)^0\right)^\gen\right) \ge
  \Big(1-\e\cdot16N\Big)\mu(\alpha,G) \;.
  $$
  But $\CC_G\big(\CC_G(A)^0\big)^0=A$ and
  the Dichotomy Lemma~\ref{dichotomy-lemma} follows.
\end{proof}

\section{Finding and using CCC-subgroups}
\label{sec:finding-using-ccc}

Let $G$ be a simple algebraic group and $T$ a maximal torus of $G$.
Combining the previously developed techniques we can show that if an
appropriate finite subset $\alpha\subset G$ does not grow
then either $\mu(\alpha,T)$ is relatively small or $\alpha$ itself must be
very large compared to $\Span\alpha$ (which must be finite in this case).
We actually prove a similar result for non-normal CCC-subgroups of arbitrary
connected linear algebraic groups $G$.
For $G$ non-nilpotent we then construct CCC-subgroups which can be used as an
input for the above result.

It is crucial in the proofs of our main theorems
to find sufficiently many
$\Span\alpha$-conjugates of a CCC-subgroup $A\le G$.
We define a quantity $\hat\mu$ measuring their number
in a sense analogous to the concentration $\mu$.
To simplify the notation we restrict this definition to the case
$\alpha\subset G$,
in the more general situation we use a much cruder estimate.

\begin{defn}
  Let $G$ be a connected linear algebraic group, $A\le G$ a closed
  subgroup and $\alpha\subset G$ a finite subset.
  Suppose that $G$ does not normalise $A$.
  We define
  $$
  \hat\mu\big(\Span\alpha,G,A\big) =
  \frac{\log\Big|\big\{t^{-1}At \;\big|\; t\in\Span\alpha\big\}\Big|}
  {\dim(G)-\dim\big(\CN_G(A)\big)} =
  \frac{\log\Big|\Span{\alpha}:\CN_{\Span{\alpha}}(A)\Big|}
  {\dim(G)-\dim\big(\CN_G(A)\big)} \;.
  $$
\end{defn}

\begin{rem}
  The $G$-conjugates of $A$ are parametrised by the quotient variety
  $X=G/\CN_G(A)$. Let $\hat\alpha\subset X$ denote the image of
  $\Span\alpha$, these are the parameter values that correspond to
  the $\Span\alpha$-conjugates of $A$.
  Then $\hat\mu\big(\Span\alpha,G,A\big)=\mu(\hat\alpha,X)$.
\end{rem}

\newcommand{\spreadViaCCC} {{\rm s}}
\begin{lem} [spreading via CCC-subgroups]
  \label{spread-via-CCC}
  For all parameters $N>0$, $\Delta>0$ and $\Frac1{119N^3}>\e>0$
  there is an integer
  $M=M_\spreadViaCCC(N,\e)$
  and a real
  $K=K_\spreadViaCCC(N,\Delta,\e)$
  with the following property.\\
  Let $\alpha|G$ be an
  $(N,\Delta,K)$-bounded spreading system
  and $A<G$ a CCC-subgroup
  such that
  $$
  \mu(\alpha,A)> \left(1-\Eescape\right)\cdot\mu(\alpha,G) \;.
  $$
  Suppose that at least one of the following holds:
  \begin{enumerate} [\indent(a)]
  \item \label{item:21}
    $$
    \Big|\Span{\alpha}:\CN_{\Span{\alpha}}(A)\Big| \ge
    \big|\alpha\big|^{2N} \;,
    $$
  \item \label{item:22}
    $\alpha\subset G$, $A$ is not normal in $G$ and
    $$
    \mu(\alpha,G)\le
    \Big(1-\e\cdot64N^3\Big)\cdot\hat\mu\big(\Span\alpha,G,A\big) \;.
    $$
  \end{enumerate}
  Then $\alpha|G$ is $(\e,M,K)$-spreading.
  Moreover, our construction of the subgroup of spreading
  is uniquely determined.
\end{lem}

\begin{proof}
  By Lemma~\ref{ccc-subgroup-properties} the conjugate subsets
  $h^{-1}A^\gen h$ for various $h$ normalising $G$
  are pairwise disjoint or coincide.
  They are all contained in
  $\Prod^{\dim(G)}G$ which has dimension $\dim(G)^2\le N^2$.

  In case~\eqref{item:22} we consider the following set:
  $$
  X =
  \bigcup\Big\{h^{-1}A^\gen h\;\big|\; h\in G\Big\} \subseteq
  \Prod^{\dim(G)}G \;.
  $$
  Then
  $
  \dim\big(\cl{X}\big)\le N^2
  $.
  The virtue of this estimate is that it depends only on $N$,
  but we also need a precise calculation in terms of $A$ and $G$.
  We consider the conjugation map
  $\phi:G\times \cl{A^\gen}\to\Prod^{\dim(G)}G$
  defined as $\phi(h,\ua)=h^{-1}\ua h$
  (note that $\cl{A^\gen}=\Prod^{\dim(G)}A$).
  By definition $X=\phi\big(G\times A^\gen\big)$
  hence $\cl{X}=\cl{{\rm im}(\phi)}$
  and
  $\deg(\cl{X})$ is bounded in terms of $N$ and $\Delta$
  (see Fact~\ref{closed-set-constructions}.\eqref{item:15}).
  Consider any pair $(h_0,\underline{a_0})\in G\times A^\gen$
  and its image $x=h_0^{-1}\underline{a_0}h_0\in X$.
  The corresponding fibre is
  $$
  \phi^{-1}(x) =
  \left\{ (nh_0, n\ua_0 n^{-1})\,\Big|\, n\in\CN_G(A) \right\} \;,
  $$
  which is isomorphic 
  (as an algebraic set, see Remark~\ref{category-of-closed-sets})
  to $\CN_G(A)$.
  In particular, $G\times A^\gen$
  (which is open and dense in the domain of $\phi$)
  is the union of fibres of dimension
  $\dim\big(\CN_G(A)\big)$.
  Therefore
  $$
  \dim(\cl{X}) = \dim\big(A^\gen\big) +
  \Big[\dim(G)-\dim\big(\CN_G(A)\big)\Big] >
  \dim\big(A^\gen\big)
  $$
  (apply Fact~\ref{closed-set-constructions}.\eqref{item:14} to the
  irreducible set $G\times\cl{A^\gen}$).

  In case~\eqref{item:21} we define the parameters
  $\e'' = \e\cdot16N >\e$ and $\Delta''=\Delta^N$,
  in case~\eqref{item:22} we use the same $\e''$ and
  we set $\Delta''=\max\big(\Delta,\deg(\cl{X})\big)$.
  We define
  $$
  M'= M_\dichotomy(N,\e)
  \;,\quad
  M'' =  M_\spreading(N,\e'')
  \;,
  $$
  $$
  M= \max\big(4M'+1\,,\;  2M'\cdot M''\big)
  \;,
  $$
  $$
  K=\max\Big(K_\dichotomy(N,\Delta,\e)\,,\;K_\spreading(N,\Delta'',\e'')\Big) \;.
  $$
  We consider all the conjugate subgroups
  $$
  \CA=\left\{t^{-1}At \;\Big|\; t\in\Span\alpha\right\} \;,
  $$
  they are all CCC-subgroups of $G$ since $\alpha$ normalises $G$.

  In case~\eqref{item:21} we have
  $\log\big|\CA\big|\ge2N\log|\alpha|$ by assumption.
  In case~\eqref{item:22} we obtain instead
  the following estimate
  $$
  \log\big|\CA\big| =
  \hat\mu\big(\Span\alpha,G,A\big)\cdot
  \Big[\dim(G)-\dim\big(\CN_G(A)\big)\Big] =
  $$
  $$
  =
  \big[\dim(\cl{X})-\dim(A^\gen)\big]\cdot
  \hat\mu\big(\Span\alpha,G,A\big) \ge
  $$
  $$
  \ge
  \big[\dim(\cl{X})-\dim(A^\gen)\big]\cdot
  \Frac1{1-\e\cdot64N^3}\cdot\mu(\alpha,G) >
  $$
  $$
  >
  \big[\dim(\cl{X})-\dim(A^\gen)\big]\cdot
  \big(1+\e''\cdot4\dim(\cl{X})\big)\cdot\mu(\alpha,G) \;.
  $$
  Suppose first that
  \begin{equation}
    \label{eq:20}
    \mu\big(\alpha^2,B\big) \ge \left(1-\Eescape\right)\mu(\alpha,G)
  \end{equation}
  for all $B\in\CA$.
  We apply the Dichotomy Lemma~\ref{dichotomy-lemma}
  with parameters $N$, $\Delta$ and $\e$
  to $\alpha^2|G$ and each $B\in\CA$.
  We get that either $\alpha^2|G$ is $(\e,M',K)$-spreading
  i.e. $\alpha|G$ is $(\e,2M',K)$-spreading, and the lemma
  holds,
  or
  $$
  \mu\Big(\Prod^{\dim(G)}\alpha^{2M'},B^\gen\Big) \ge
  \Big(1-\e\cdot16N\Big)\mu(\alpha^2,G) \ge
  (1-\e'')\mu(\alpha,G)
  $$
  for all $B\in\CA$
  (in particular, $\Prod^{\dim(G)}\alpha^{2M'}$ has at least one element in
  each $B^\gen$).
  Let us consider this latter possibility.
  By Lemma~\ref{ccc-subgroup-properties}
  the subsets $B^\gen$ are pairwise disjoint.
  In case~\eqref{item:22} we obtain
  $$
  \mu\left(\Prod^{\dim(G)}\alpha^{2M'},\cl{X}\right) =
  \Frac1{\dim(\cl{X})}\log\left|\Prod^{\dim(G)}\alpha^{2M'}\cap \cl{X}\right| \ge
  $$
  $$
  \ge
  \Frac1{\dim(\cl{X})}\,
  \log\left(\sum_{B\in\CA}
    \left|\Prod^{\dim(G)}\alpha^{2M'}\cap B^\gen\right| \right) \ge
  $$
  $$
  \ge
  \Frac1{\dim(\cl{X})}\,\left[
    \log\big|\CA\big| +
    \log\left(\min_{B\in\CA}
      \left|\Prod^{\dim(G)}\alpha^{2M'}\cap B^\gen\right| \right)
    \right] =
  $$
  $$
  \ge
  \Frac1{\dim(\cl{X})}\,\left[
    \log\big|\CA\big| +
    \dim(A^\gen)\cdot\min_{B\in\CA} \left(
      \mu\left(\Prod^{\dim(G)}\alpha^{2M'}, B^\gen\right) \right)
    \right] \ge
  $$
  $$
  \ge
  \Frac1{\dim(\cl{X})}\,\Big[
    \log\big|\CA\big| +
    \dim(A^\gen)\cdot\left(1-\e''\right)\mu(\alpha,G)
    \Big] \ge
  $$
  $$
  \ge
  \Frac1{\dim(\cl{X})}\,\Big[
    \big[\dim(\cl{X})-\dim(A^\gen)\big]\cdot
    \big(1+\e''\cdot4\dim(\cl{X})\big)\mu(\alpha,G) +
    $$
    $$
    +
    \dim(A^\gen)\cdot\left(1-\e''\right)\mu(\alpha,G)
    \Big] =
  $$
  $$
  =
  \left[1 + 4\e''\big(\dim(\cl{X})-\dim(A^\gen)\big) -
    \e''\Frac{\dim(A^\gen)}{\dim(\cl{X})}
  \right]\mu(\alpha,G) >
  $$
  $$
  > \big(1 + 3\e''\big)\cdot\mu(\alpha,G) \;.
  $$
  In case~\eqref{item:21} a similar, but much shorter calculation shows that
  $$
  \mu\left(\Prod^{\dim(G)}\alpha^{2M'},\Prod^{\dim(G)}G\right) \ge
  \Frac{\log|\CA|}{\dim(G)^2} \ge
  \Frac{2\log|\alpha|}{\dim(G)}\ge(1+3\e'')\mu(\alpha,G) \;.
  $$
  In both cases we apply the Spreading Theorem~\ref{spread-all-over}
  with parameters $N$, $\Delta''$ and $\e''$ to $\alpha^{2M'}|G$,
  and in case~\eqref{item:21}
  to the set $\Prod^{\dim(G)}G$, in case~\eqref{item:22}
  to the set $\cl{X}$.
  We obtain that $\alpha^{2M'}|G$ is $(\e'',M'',K)$-spreading,
  hence $\alpha|G$ is $(\e,2M'M'',K)$-spreading,
  the lemma holds.

  Finally we assume that condition \eqref{eq:20}
  does not hold for all members of $\CA$.
  As the subgroup $A$ itself satisfies it,
  there must be at least one subgroup $B_0\in\CA$
  and an element $b\in\alpha$ such that
  $B_0$ satisfies \eqref{eq:20}
  but $b^{-1}B_0b$ doesn't:
  \begin{equation}
    \label{eq:21}
    \mu\big(\alpha^2,b^{-1}B_0b\big) < \left(1-\Eescape\right)\mu(\alpha,G) \;.
  \end{equation}
  Conjugating by $b$ we transform \eqref{eq:20} into
  $$
  \mu\big(\alpha^4,b^{-1}B_0b\big) \ge
  \mu\big(b^{-1}\alpha^2b,b^{-1}B_0b\big) =
  $$
  $$
  =
  \mu\big(\alpha^2,B_0) >
  \left(1-\Eescape\right)\,\mu(\alpha,G) \;.
  $$
  Again we apply the Dichotomy Lemma~\ref{dichotomy-lemma}
  with parameters $N$, $\Delta$ and $\e$ to $\alpha^4|G$ and
  the CCC-subgroup $b^{-1}B_0b$.
  We obtain that either $\alpha^4|G$ is $(\e,M',K)$-spreading, and the
  lemma holds in this case, or
  $$
  \mu\big(\alpha^{4M'},b^{-1}B_0b\big) \ge
  \big(1-\e\cdot16N\big)\mu(\alpha,G) \;.
  $$
  Now we compare this to inequality~\eqref{eq:21}
  and apply Lemma~\ref{induce-from-growing-subgroup}
  to the subgroup $b^{-1}B_0b$ with $k=4M'$.
  We obtain that
  $$
  \mu\big(\alpha^{4M'+1},G\big) \ge
  \mu(\alpha,G)+\Frac{\dim(b^{-1}B_0b)}{\dim(G)}
  \left[\Eescape-\e\cdot16N\right]\mu(\alpha,G) \ge
  $$
  $$
  \ge
  \mu(\alpha,G)+
  \Frac1{N}\left[\Eescape-\e\cdot16N\right]\mu(\alpha,G) =
  $$
  $$
  =
  \left[1 +\Frac1{7N^3} - 16\e\right]\mu(\alpha,G) \ge
  (1+\e)\,\mu(\alpha,G) \;,
  $$
  hence $G$ itself is a subgroup of $(\e,4M'+1,K)$-spreading
  for $\alpha|G$.
\end{proof}

\newcommand{\bad} {{\rm bad}}
\begin{lem} \label{non-regular-elements}
  Let $G$ be a non-abelian connected linear algebraic group
  and $\CS\subseteq G$ the closure of the set of those elements $g\in G$
  whose centraliser is either the whole of $G$
  or does not contain any maximal torus.
  Then $\dim(\CS)<\dim(G)$ and the degree of
  $\CS$ is bounded: 
  $$
  \deg\big(\CS\big)\le\Delta_\bad\big(\dim(G),\deg(G)\big) \;.
  $$
\end{lem}

\begin{proof}
  Let $A\le G$ be a Cartan subgroup. Then $A=\CC_G(T)$ for some maximal torus
  $T\le G$. Hence for each $g\in A$ we have $T\le\CC_G(g)$.
  All Cartan subgroups are conjugates of $A$,
  hence their union, denoted by $\CR$,
  is the image of the conjugation map
  $f:A\times G\to G$, $f(a,g)=g^{-1}ag$.
  It is well-known that $\CR$ contains an open subset $U$ of $G$
  and by definition
  $\cl{G\setminus\CR}\subseteq G\setminus U$,
  so $\dim\big(\cl{G\setminus\CR}\big)<\dim(G)$
  (see Fact~\ref{dimension-degree}.\eqref{item:7}).
  Moreover, $\deg\big(\cl{G\setminus\CR}\big)$
  is bounded in terms of $\dim(G)$ and $\deg(G)$
  (see Fact~\ref{closed-set-constructions}.\eqref{item:13}).
  We also know that $\deg\big(\CZ(G)\big)\le\deg(G)$
  (see Fact~\ref{normaliser-centraliser-fact}).
  Hence
  $
  \CS=\big(\cl{G\setminus\CR}\big)\cup\CZ(G)
  $
  also has bounded degree.
\end{proof}

\newcommand{\findCCC} {{\rm CCC}}
\begin{lem} [Finding CCC-subgroups]
  \label{find-CCC}
  For all parameters $N>0$, $\Delta>0$ and $\Frac1{56N^3}>\e>0$
  there is an integer
  $M=M_\findCCC(N,\e)$
  and a real
  $K=K_\findCCC(N,\Delta,\e)$
  with the following property.\\
  Let $\alpha|G$ be an
  $(N,\Delta,K)$-bounded spreading system such that $G$ is
  non-nilpotent.
  Then either it is $(\e,M,K)$-spreading, or there is a CCC-subgroup
  $A\le G$ which contains exactly one maximal torus of $G$ and satisfies
  $$
  \mu\big(\alpha^M,A\big)>(1-\e\cdot16N)\mu(\alpha,G) \;.
  $$
  In particular, $A$ is not normal in $G$.
  Moreover, our construction of $A$ and of the subgroup of spreading
  is uniquely determined.
\end{lem}

\begin{proof}
  Recall the functions
  $M_\escape$, $K_\escape$, $M_\centraliser$, $K_\centraliser$,
  $M_\asymdichotomy$, $K_\asymdichotomy$ and $\Delta_\bad$ from
  the lemmas \ref{escape}, \ref{centraliser-lemma},
  \ref{asymmetric-dichotomy} and \ref{non-regular-elements}.
  We define the following constants:
  $$
  M_\centraliser  =  M_\centraliser(N,\e)\;,\quad
  M_\escape  =  M_\escape(N,\e) \;,\quad
  M_\asymdichotomy  =  M_\asymdichotomy(N,\e) \;,
  $$
  $$
  \tilde\Delta  = \max\big(\Delta,\Delta_\bad(N,\Delta)\big) \;,\quad
  M =  M_\centraliser^N\max\big(M_\escape,M_\asymdichotomy\big)\;,
  $$
  $$
  K  =  \max\Big(
  K_\centraliser(N,\Delta,\e) \,,\;
  K_\escape(N,\tilde\Delta,\e) \,,\;
  K_\asymdichotomy(N,\Delta,\e)
  \Big) \;.
  $$
  Set $g_0=1\in G$, $G_0=G$.
  We define by induction on $i$ the elements 
  $g_{i}\in\alpha^{(M_\centraliser)^{i-1}}\cap G$ in such a way that the subgroups
  $$
  G_i = \CC_G(g_0,g_1,g_2,\dots, g_i)^0 =
  \CC_{G_{i-1}}(g_i)^0
  $$
  satisfy
  \begin{equation}
    \label{eq:22}
    \mu\left(\alpha^{(M_\centraliser)^i}, G_{i}\right)\ge
    \Big(1-\e\cdot8N\Big)\mu(\alpha,G) \;,
  \end{equation}
  all $G_i$ contain some maximal torus of $G$ and they form a strictly
  decreasing series of subgroups.
  Then their dimension is strictly decreasing as well,
  hence the sequence has length smaller than $N$.

  Suppose that such a $G_{i}$ is already defined for some $N>i\ge0$.
  If it is abelian then we stop the induction, otherwise continue.
  Let $\CS_i\subsetneq G_i$ be the subset defined in 
  Lemma~\ref{non-regular-elements}.
  Note, that $\deg(G_i)\le\Delta$, $\mult(G_i)\le\Delta$ and
  $\inv(G_i)\le\Delta$ (see Fact~\ref{normaliser-centraliser-fact}),
  hence $\deg(\CS_i)\le\tilde\Delta$.
  We apply the Escape Lemma~\ref{escape}
  with parameters $N$, $\tilde\Delta$ and $\e$
  to $\alpha^{(M_\centraliser)^i}|G$
  and the subsets $X=\CS_i$ and $Y=G_i$ of $G$.
  If we obtain a subgroup of $(\e,M_\escape,K)$-spreading
  then the lemma holds.
  Otherwise, since \eqref{eq:22} holds, we find at least one element
  $$
  g_{i+1}\in\alpha^{(M_\centraliser)^{i}}\cap\big(G_i\setminus\CS_i\big) \;.
  $$
  (In fact the Escape Lemma gives us many such elements).
  We select the $g_{i+1}$ which is minimal in the order of $\Span\alpha$.
  According to the definition of $\CS_i$,
  $$
  G_{i+1}=\big(G_i\cap \CC_G(g_{i+1})\big)^0
  $$
  contains a maximal torus of $G_i$, which is also a maximal torus in $G$,
  and $G_{i+1}$ is strictly smaller than $G_i$.
  We apply the Centraliser Lemma~\ref{centraliser-lemma}
  with parameters $N$, $\Delta$ and $\e$ to $\alpha^{(M_\centraliser)^i}|G$
  and the centraliser subgroup $\CC_G(g_0,\dots,g_{i+1})$.
  In case we obtain a subgroup of $(\e,M_\centraliser,K)$-spreading,
  the lemma holds.
  Otherwise we have
  $$
  \mu\left(\alpha^{(M_\centraliser)^iM_\centraliser}, G_{i+1}\right)\ge
  \Big(1-\e\cdot8N\Big)\cdot\mu(\alpha^{(M_\centraliser)^i},G)
  \ge \Big(1-\e\cdot8N\Big)\mu(\alpha,G)
  $$
  i.e. $G_{i+1}$ satisfies \eqref{eq:22}.

  As we explained before, this process must stop in at most $N$ steps.
  But the only way it can stop is to arrive at a connected abelian subgroup
  $G_I$ which contains a maximal torus $T$
  and satisfies inequality~(\ref{eq:22}).

  We set $A=\CC_G(G_I)^0$. On the one hand, $T$ commutes with $G_I$,
  hence $T\le A$.
  On the other hand, $A=\CC_G(G_I)^0\le\CC_G(T)$,
  and the latter one is a Cartan subgroup,
  which has a unique maximal torus.
  Therefore $T$ is the only maximal torus in $A$.
  But $G$ is non-nilpotent, hence $G$ has several maximal tori.
  This implies that $A$ is a CCC-subgroup which is not normal.
  We apply the Asymmetric Dichotomy Lemma~\ref{asymmetric-dichotomy}
  with parameters $N$, $\Delta$ and $\e$ to
  $\alpha^{(M_\centraliser)^N}|G$
  and the subset $Z=G_I$.
  In case we obtain a subgroup of $\big(\e,M_\asymdichotomy,K\big)$-spreading,
  the lemma holds.
  Otherwise, since $G_I$ satisfies \eqref{eq:22}, we obtain that
  $$
  \mu(\alpha^M, A) \ge (1-\e\cdot16N)\,\mu(\alpha,G)
  $$
  as required.
\end{proof}

Suppose we want to prove that a certain spreading system $\alpha|G$ is
$(\e,M,K)$-spreading.
Our strategy is to obtain a CCC-subgroup $A<G$ via
Lemma~\ref{find-CCC},
and use Lemma~\ref{spread-via-CCC} to establish the $(\e,M,K)$-spreading.
In order to do this, we need to estimate the number of
$\Span\alpha$-conjugates of $A$.
In Section~\ref{sec:finite-groups-lie} we develop a powerful method
for finite $\Span\alpha$.
Later in Section~\ref{sec:linear-groups}
we deal with the much simpler case
when $A$ has infinitely many conjugates.

\section{Finite groups of Lie type}
\label{sec:finite-groups-lie}

In this section we use the general results established earlier to prove 
Theorem~\ref{pre-main-thm},
our main technical result concerning fixpoint
groups of Frobenius maps of linear algebraic groups.

\begin{defn} \label{Frobenius-def}
  Let $G$ be a linear algebraic group over the field \Fpclosed.
  \begin{enumerate}[\indent(a)]
  \item \label{item:23}
    For each $p$-power $q$ the usual $q$-th power map
    $\Fpclosed\to\Fpclosed$ is a field automorphism.
    Applying this to the entries of the $n\times n$ matrices we obtain
    the group automorphisms
    $$
    \Frob_q:GL(n,\Fpclosed)\to GL(n,\Fpclosed) \;.
    $$
    (Note, that these are not morphisms of varieties.)
  \item \label{item:24}
    More generally, $\Frob_q$ can be defined the same way 
    on any algebraically closed field of characteristic $p$,
    hence we can talk about $\Frob_q$-invariant algebraic sets
    and $\Frob_q$-equivariant morphisms (i.e. morphisms compatible
    with the $\Frob_q$-actions on the domain and the range).
    (These are precisely the algebraic sets and morphisms
    \emph{defined over $\Fq$}.) 
  \item \label{item:25}
    A \emph{Frobenius map} of $G$ is a group automorphism 
    $\sigma:G\to G$ such that there is a $p$-power $q$, an exponent $k$ and
    a faithful representation $G\hookrightarrow GL(n,\Fpclosed)$ such that
    $G$ is $\Frob_q$-invariant, and $\sigma^k$ is the restriction of
    the automorphism $\Frob_q$ to $G$. The fixpoint subgroup of $\sigma$ is
    denoted by $G^\sigma$. We define $q_\sigma=\sqrt[k]{q}$.
  \end{enumerate}
\end{defn}

\begin{rem} \label{Frob-q-remark}
  The fixpoint set of $\Frob_q$ is clearly
  $GL(n,\Fpclosed)^{\Frob_q}=GL(n,\Fq)$.
  More generally, if the closed subgroup $G\le GL(n,\Fpclosed)$
  is $\Frob_q$-invariant then $G^{\Frob_q}=G(\Fq)$,
  the set of those elements whose matrix belongs to $GL(n,\Fq)$.
\end{rem}

We will combine our previous results with the following powerful extension of
the Lang-Weil estimates \cite{Hr}.

\begin{prop}[Hrushovski]
  \label{twisted-Lang-Weil}
  Let $G$ be a connected linear algebraic group and $\sigma:G\to G$ a
  Frobenius map.
  Then there is a constant $C=C\big(\dim(G),\deg(G)\big)$ such that
  $|G^\sigma|$ is approximately $q_\sigma^{\dim(G)}$ with error
  $$
  \Big| |G^\sigma| - q_\sigma^{\dim(G)} \Big| \le
  C\cdot q_\sigma^{\dim(G)-\frac12} \;.
  $$
\end{prop}

In the following corollary, besides various technical estimates, we establish
that the finite group $G^\sigma$ (if it is large enough) reflects the
group-theoretic properties of $G$.
E.g. there is a correspondence between subgroups of $G$ and $G^\sigma$,
and we have
$\CC_G(G^\sigma)=\CZ(G)$.

\newcommand{\sizeOfLieType} {{\rm L}}
\begin{cor}
  \label{size-of-Lie-type}
  For all parameters $N>0$, $\Delta>0$, $I>0$ and $1>\e>0$
  there is an integer
  $K=K_\sizeOfLieType(N,\Delta,I,\e)$
  with the following property.
  \begin{enumerate}[\indent(a)]
  \item \label{item:26}
    Let $G$ be a connected linear algebraic group,
    $\sigma:G\to G$ a Frobenius map and
    $\alpha\subseteq G^\sigma$ a finite subset.
    Suppose that $\dim(G)\le N$, $\deg(G)\le\Delta$,
    $\big|G^\sigma:\Span\alpha\big|\le I$ and $|\alpha|\ge K$.
    Then
    $$
    \dim(G)>0
    \;,\quad
    \CC_G(\alpha) = \CZ(G)
    \;,\quad
    \log(q_\sigma)\ge1/\e
    \;.
    $$
  \item \label{item:27}
    Let in addition $A\le G$ be a $\sigma$-invariant closed subgroup
    of degree $\deg(A)\le\Delta$.
    Then $A^\sigma = A\cap G^\sigma$,
    $$
    \Big|\Span\alpha:\Span\alpha\cap A\Big| \ge
    \Frac{1-\e}{I\Delta}\,\Big|G^\sigma\Big|^{1-\dim(A)/\dim(G)} \ge
    \Frac{1-\e}{I\Delta}\,
    \Big|\Span\alpha\Big|^{1-\dim(A)/\dim(G)}
    $$
    and if $A\neq G$ then $\Span\alpha\cap A \neq \Span\alpha$.
  \item \label{item:28}
    Suppose furthermore that $A$ is not normal in $G$.
    Then $\alpha$ does not normalise $A$ and
    $$
    (1-\e)\log(q_\sigma)\le
    \hat\mu\big(\Span\alpha,G, A\big)\le
    (1+\e)\log(q_\sigma)
    \;.
    $$
  \end{enumerate}
\end{cor}

\begin{proof}
  Recall from Proposition~\ref{twisted-Lang-Weil} the constant
  $C=C\big(N,\Delta\big)$.
  By Proposition~\ref{twisted-Lang-Weil} we have
  $$
  K\le|\alpha|\le|G^\sigma|\le(1+C)q_\sigma^N \;,
  $$
  hence for large enough $K$
  $$
  \log(q_\sigma)\ge\log\left(\sqrt[N]{\Frac{K}{1+C}}\right)>1/\e
  $$
  and $\dim(G)>0$ (see Remark~\ref{too-many-points}).
  This proves the two inequalities
  of \eqref{item:26}. In the rest of this proof we often use,
  that by choosing $K$ large enough one can force $q_\sigma$
  to be arbitrary large.
  
  It is obvious that $A^\sigma=A\cap G^\sigma$.
  By Proposition~\ref{twisted-Lang-Weil} for large enough $q_\sigma$
  (i.e. for large enough $K$) we have
  $$
  (1-\Frac\e3)q_\sigma^{\dim(G)}\le
  |G^\sigma|\le
  (1+\Frac\e3)q_\sigma^{\dim(G)}
  $$
  and
  $$
  (1-\Frac\e3)q_\sigma^{\dim(A)}\le
  \big|(A^0)^\sigma\big|\le
  |A^\sigma|\le \Delta\big|(A^0)^\sigma\big| \le
  \Delta(1+\Frac\e3)q_\sigma^{\dim(A)} \;.
  $$
  Therefore
  $$
  \Big|G^\sigma:A^\sigma\Big| \ge
  \frac{(1-\Frac\e3)q_\sigma^{\dim(G)}}
  {(1+\Frac\e3)\Delta\,q_\sigma^{\dim(A)}} > 
  \frac{1-2\Frac\e3}{\Delta}q_\sigma^{\dim(G)-\dim(A)} >
  $$
  $$
  >
  \frac{1-2\Frac\e3}{(1+\Frac\e3)\Delta}
  \Big|G^\sigma\Big|^{1-\dim(A)/\dim(G)} >
  \frac{1-\e}{\Delta}\Big|G^\sigma\Big|^{1-\dim(A)/\dim(G)} \;.
  $$
  This implies the inequality in \eqref{item:27}.
  If $A\neq G$ then $\dim(A)<\dim(G)$.
  Since $\big|G^\sigma\big|\ge K$,
  for large enough $K$ we have
  $\big|\Span\alpha:\Span\alpha\cap A\big|>1$,
  so $\Span\alpha\neq\Span\alpha\cap A$.
  This completes the proof of \eqref{item:27}.
  
  Let $g\in\CC_G(\alpha)$ be such that $g\notin\CZ(G)$.
  Clearly all elements of the $\Span\sigma$-orbit $g^{\Span\sigma}$ commute
  with the elements of $\alpha$.
  On the other hand
  we know from \eqref{item:27} (say with parameter $\e'=\Frac12$) that 
  $\Span\alpha\cap\CC_G\big(g^{\Span\sigma}\big)\neq\Span\alpha$,
  which is a contradiction. Therefore $\CC_G(\alpha)=\CZ(G)$
  which completes the proof of \eqref{item:26}.

  Suppose now that $A$ is not normal in $G$.
  We apply \eqref{item:27} (say with parameter $\e''=\Frac12$)
  to the proper subgroup $\CN_G(A)< G$.
  We obtain that 
  $$
  \Span\alpha\neq
  \Span\alpha\cap\CN_G(A) =\CN_{\Span\alpha}(A)
  $$
  i.e. $\alpha$ does not normalise $A$.

  By Fact~\ref{normaliser-centraliser-fact} there is an upper bound
  $\Delta'=\Delta'(N,\Delta)\ge\deg\big(\CN_G(A)\big)$.
  We set $\e'''=\Frac\e{2N}$.
  We apply \eqref{item:26} with a sufficiently small parameter $\e'$
  to obtain that
  $
  \log(q_\sigma) > 
  \Frac1{\e'''}\Big( 1+\log\big(\max(\Delta,\Delta',I)\big)\Big)
  $.
  Let $B\le G$ be any $\sigma$-invariant closed subgroup with $\dim(B)>0$ and
  $\deg(B)\le\max(\Delta,\Delta')$.
  We apply Proposition~\ref{twisted-Lang-Weil} to $B^0$
  and obtain
  $$
  \Big|\log|G^\sigma| - \dim(G)\cdot\log(q_\sigma)\Big| < 1
  \;.
  $$
  This gives us upper and lower bounds on $\log\big|\Span\alpha\cap B\big|$:
  $$
  (1-\e''')\dim(B)\log(q_\sigma) \le
  $$
  $$
  \le
  \dim(B)\cdot\log(q_\sigma)  -1-\log(I) \le
  \log\big|(B^0)^\sigma\big| -\log(I) \le
  $$
  $$
  \le
  \log\big|\Span\alpha\cap B^0\big| \le
  \log\big|\Span\alpha\cap B\big| \le
  \log\big|B^\sigma\big| \le
  $$
  $$
  \le
  \log\big|(B^0)^\sigma\big| + \log\big(\max(\Delta,\Delta')\big) \le
  $$
  $$
  \le
  \dim(B)\cdot\log(q_\sigma) +1 + \log\big(\max(\Delta,\Delta')\big) \le
  $$
  $$
  \le
  (1+\e''')\dim(B)\log(q_\sigma)
  $$
  We apply these inequalities to $B=G$ and to $B=\CN_G(A)$:
  $$
  (1-\e''')\dim\big(G\big)\log(q_\sigma) \le
  \log\big|\Span\alpha\big| \le
  (1+\e''')\dim\big(G\big)\log(q_\sigma)
  $$
  and
  $$
  (1-\e''')\dim\big(\CN_G(A)\big)\log(q_\sigma) \le
  \log\big|\CN_{\Span\alpha}(A)\big| \le
  $$
  $$
  \le
  (1+\e''')\dim\big(\CN_G(A)\big)\log(q_\sigma) \;.
  $$
  Subtracting the two estimates and dividing the result with
  $\dim(G)-\dim\big(\CN_G(A)\big)>0$ we obtain
  $$
  (1-\e)\log(q_\sigma) \le
  \Frac{\log|\Span\alpha|-\log|\CN_{\Span\alpha}(A)|}{\dim(G)-\dim(\CN_G(A))} \le
  (1+\e)\log(q_\sigma)
  $$
  and this completes the proof of \eqref{item:28}.
\end{proof}

We arrived at a slightly more general version of
Theorem~\ref{pre-main-thm} of the introduction:

\begin{thm}
  \label{main-thm}
  For all parameters $N>0$, $\Delta>0$, $I>0$ and $1>\e>0$
  there is an integer
  $M=M_\mainThm(N,\e)$
  and a real
  $K=K_\mainThm(N,\Delta,I,\e)$
  with the following property.\\
  Let $G$ be a connected linear algebraic group over \Fpclosed.
  Let $\sigma:G\to G$ a Frobenius map and
  $1\in\alpha\subseteq G^\sigma$ an ordered finite symmetric subset.
  Suppose that $\CZ(G)$ is finite,
  $\dim(G)\le N$, $\deg(G)\le\Delta$, $\mult(G)\le\Delta$, $\inv(G)\le\Delta$,
  $\big|G^\sigma:\Span\alpha\big|\le I$ and
  $$
  K\le|\alpha|\le q_\sigma^{(1-\e)\dim(G)} \;.
  $$
  Then there is a $\sigma$-invariant
  connected closed normal subgroup $H\triangleleft G$
  such that $\deg H\le K$, $\dim(H)>0$ and
  $$
  |\alpha^M\cap H|\ge|\alpha|^{(1+\delta)\dim(H)/\dim(G)}
  $$
  where $\delta=\Frac{\e}{128N^3}$.
  Moreover, our construction of the subgroup $H$
  is uniquely determined.
\end{thm}

\begin{proof}
  We set
  $$
  M_\findCCC=M_\findCCC\left(N,\Delta,\Frac{\e}{119N^3}\right)
  \;,\quad
  M_\spreadViaCCC=M_\spreadViaCCC\left(N,\Delta,\Frac{\e}{128N^3}\right)
  \;,
  $$  
  $$
  M= M_\findCCC\cdot M_\spreadViaCCC  \;,
  $$  
  $$
  K = \max\Big(
  \Delta+1,
  K_\findCCC\left(N,\Delta,\Frac{\e}{119N^3}\right),
  K_\sizeOfLieType\left(N,\Delta,I,\Frac\e3\right),
  K_\spreadViaCCC\left(N,\Delta,\Frac{\e}{128N^3}\right)
  \Big) \;.
  $$
  By Corollary~\ref{size-of-Lie-type}.\eqref{item:26}
  $\dim(G)>0$ and $\CC_G(\alpha)=\CZ(G)$, which is finite,
  hence $\alpha|G$ is an $(N,\Delta,K)$-bounded spreading system.
  By assumption
  $$
  \mu(\alpha,G)\le (1-\e)\log(q_\sigma) \;.
  $$
  Our construction of $H$ will be uniquely determined,
  therefore it will be $\sigma$-invariant.
  By Corollary~\ref{size-of-Lie-type}.\eqref{item:28}
  the rest of the conclusion of the theorem can be rewritten as follows.
  $H$ is normalised by $\alpha$, $\deg(H)\le K$, $\dim(H)>0$ and
  $$
  \mu(\alpha^M,H)\ge(1+\delta)\mu(\alpha,G) 
  $$
  i.e. we need to prove that $\alpha|G$ is $(\delta,M,K)$-spreading
  and construct a subgroup of spreading that is uniquely determined. 

  If $G$ were nilpotent then $\CZ(G)$ would have positive dimension.
  By assumption $\CZ(G)$ is finite, hence $G$ is not nilpotent.
  We apply Lemma~\ref{find-CCC}
  with parameters $N$, $\Delta$ and $\Frac{\e}{119N^3}$ to $\alpha|G$.
  In case we obtain a subgroup of spreading, the theorem holds.
  Otherwise we find a CCC-subgroup $A\le G$ which is not normal in $G$ and
  satisfies
  $$
  \mu\big(\alpha^{M_\findCCC},A\big) >
  \left(1-\Frac{\e}{119N^3}\cdot16N\right)\mu(\alpha,G) >
  \left(1-\Eescape\right)\left(1+\Frac{\e}{119N^2}\right)\mu(\alpha,G) \;.
  $$
  If $\alpha|G$ is $(\Frac{\e}{119N^2},M_\findCCC,K)$-spreading, then it is
  $(\delta,M,K)$-spreading, the theorem holds in this case.
  So from now on we assume that
  $$
  \mu(\alpha^{M_\findCCC},G)<\big(1+\Frac{\e}{119N^2}\big)\mu(\alpha,G)
  $$
  hence
  $$
  \mu\big(\alpha^{M_\findCCC},A\big) >
  \left(1-\Eescape\right)\mu(\alpha^{M_\findCCC},G) \;.
  $$
  We know from Lemma~\ref{ccc-subgroup-properties} that
  $\deg(A)\le\deg(G)$
  and Corollary~\ref{size-of-Lie-type}.\eqref{item:28}
  with parameters $N$, $\Delta$, $I$
  and $\Frac{\e}{3}$ implies that
  $$
  \hat\mu\Big(\Span\alpha,G,A\Big)\ge
  \left(1-\Frac{\e}{3}\right)\log(q_\sigma)>
  \left(1-\Frac\e2\right)\left(1+\Frac\e6\right)\frac{\mu(\alpha,G)}{1-\e}\ge
  $$
  $$
  \ge
  \frac{\left(1-\Frac\e2\right)}{1-\e}
  \left(1+\Frac\e{119N^2}\right) \mu(\alpha,G) >
  \frac{\mu(\alpha^{M_\findCCC},G)}{1-\Frac\e2}
  \;.
  $$
  We apply Lemma~\ref{spread-via-CCC} with parameters $N$, $\Delta$
  and $\Frac{\e}{128N^3}=\delta$
  to the spreading system $\alpha^{M_\findCCC}|G$ and
  the subgroups $W=G$ and $A\le G$.
  If we obtain a subgroup of $(\delta,M_\spreadViaCCC,K)$-spreading
  then it is a subgroup of $(\delta,M,K)$-spreading for $\alpha|G$, 
  the theorem holds.
  Otherwise
  $$
  \mu(\alpha^{M_\findCCC},G) >
  \Big(1-\delta\cdot64N^3\Big)\,
  \hat\mu\big(\Span{\alpha^{M_\findCCC}},G,A\big) =
  \left(1-\Frac\e2\right)\hat\mu\big(\Span\alpha,G,A\big)
  \;,
  $$
  a contradiction.
\end{proof}

\begin{rem} \label{do-we-need-Hrushovski}
  In the proof of Theorem~\ref{simple-L}
  one can avoid using
  Proposition~\ref{twisted-Lang-Weil}.
  We know explicitly the number of elements in all finite simple groups of
  Lie type and also in their maximal tori (see e.g. \cite{Ca}).
  When $G$ is a connected adjoint simple algebraic group,
  one can show directly that $(G^\sigma)'$ does not normalise any closed
  subgroup of positive dimension and small degree.
  This also implies that $\CC_G\big((G^\sigma)'\big)$
  is finite which is all we need for the proofs of
  Theorem~\ref{intro-thm} and Theorem~\ref{simple-L}.
\end{rem}

The following result, communicated to us by Martin Liebeck,
can be used to complete the above sketch.
Let $G$ be a connected adjoint simple algebraic group over an algebraically
closed field $\BF$ of characteristic $p$, and $\sigma$ a Frobenius morphism of
$G$. Let $G(q) = (G^\sigma)'$ and assume $G(q)$ is simple.

\begin{prop} \label{from-Liebeck}
  There is no proper connected subgroup of $G$
  which contains $G(q)$.
\end{prop}

\begin{proof}
Suppose for a contradiction that $G(q) < H < G$, where $H$ is
connected. 

First we consider the action of $G(q)$ on the adjoint module $L(G)$. The
$G$-composition factors of $L(G)$ are well-known, and can be found in 
\cite[1.10]{LiS}.
With the exception of $G = B_n,C_n,D_n,F_4$ with $p=2$
and $G_2$ with $p=3$, $G$ is either irreducible on $L(G)$, or has two 
composition factors, one of which is trivial. In any case, each composition
factor is either a restricted $\BF G$-module, or a field twist of one. It
follows
that $G(q)$ is irreducible on every $G$-composition factor of $L(G)$.
Therefore $H$ is also irreducible on every $G$-composition factor of 
$L(G)$, and hence $H$ must be a semisimple group.

For the moment exclude the exceptions $B_n,\ldots G_2$ in the above 
paragraph. Clearly $G(q)$ fixes $L(H)\subset L(G)$, so it follows that 
$L(H)$ must be a composition factor of co-dimension 1 in $L(G)$. If $U_H$
is a maximal connected unipotent subgroup of $H$, then a standard result
tells us that $\dim H = 2\dim U_H + {\rm rank}(H)$. Since $\dim H = \dim
G-1$, it follows that $U_H$ is also a maximal unipotent subgroup of $G$,
and ${\rm rank}(H) = {\rm rank}(G)-1$. So the root system of $H$ 
has the same number of roots as that of $G$, and $H$ has rank 1 less than
$G$. An easy check of root systems shows that this is impossible.

It remains to handle the exceptional cases $G = B_n,C_n,D_n,F_4$ ($p=2$)
and $G_2$ ($p=3$). Consider $G_2$ and $F_4$, and let $H_0$ be a simple factor
of $H$ which contains an isomorphic copy of $G(q)$. Then $H_0$ is of rank
at most 2 (resp. 4), and the smallest projective representation of $H_0$
has dimension at least that of $G(q)$, which is 7 (resp. 26). This is clearly
impossible.

Next let $G = D_n$. Here the $G$-composition factors of $L(G)$ are 
of high weights $\lambda_2,0$ ($n$ odd) or $\lambda_2,0^2$ ($n$ even). 
We have already dealt with the case where $\dim H = \dim G-1$, so 
we may assume $n$ is
even and $\dim H = \dim G-2$. Then either $\dim U_H = \dim U_G$, 
${\rm rank}(H) = {\rm rank}(G)-2$, or $\dim U_H = \dim U_G-1$, 
${\rm rank}(H) = {\rm rank}(G)$. An inspection of root systems shows that
neither of these is possible.

Now let $G = C_n$, and let $V$ be the natural $2n$-dimensional $G$-module.
As $G(q)$ cannot act nontrivially on a module of dimension less than $2n$,
it must act tensor indecomposably on $V$, and hence so does $H$. Therefore
$H$ is simple. The possibilities for $G(q)$ are $C_n(q)$ and $Sz(q)$ (the
latter just for $n=2$). In the former case $G(q)$ has an elementary abelian
subgroup $R=r^n$, where $r$ is a prime dividing $q+1$. Note that 
$r$ is odd as $p=2$. Also ${\rm rank}(H) \le {\rm rank}(G) = n$.
An elementary argument (see \cite[Section~2]{CG}) 
shows that the abelian $r$-rank of $H$ is equal to 
${\rm rank}(H)$, and hence ${\rm rank}(H)=n$.
The only possibility is that $H = D_n$. But $G(q) = C_n(q)$ does not lie
in $D_n$ as it does not fix a quadratic form on $V$. If $G(q) = Sz(q)$ then
$H$ cannot have rank 2 (as $C_2$ has no connected 
simple proper subgroup of rank 2), so $H=A_1$; but $Sz(q) \not \le A_1$,
a contradiction.

Finally, if $G = B_n$ then there is a morphism from $G$ to $C_n$ which is
an isomorphism of abstract groups, and applying this morphism to $G(q)$ and
$H$, we reduce to the $C_n$ case. This completes the proof. 
\end{proof}

\section{Linear groups over finite fields}
\label{sec:linear-groups-over}

In this section we first prove our main theorem concerning simple groups of
Lie type and various results for $p$-generated subgroups of
$GL(n,\Fp)$
i.e. subgroups generated by elements of order $p$.
These finite groups can be obtained roughly as fixpoint groups of
Frobenius maps of linear algebraic groups.
Theorem~\ref{simple-L} is essentially a special case of
Theorem~\ref{main-thm}.
For perfect $p$-generated groups Theorem~\ref{pre-partial}
follows by an inductive argument based on Theorem~\ref{main-thm}.
To prove Theorem~\ref{pre-partial} in the general case we need a number of
finite group-theoretic results.

For the following useful results see \cite{Ol}
and \cite[proof~of~Lemma~2.2]{He1}.

\begin{prop} [Olson] \label{Olson}
  Let $1\in\alpha$ be a generating set of a finite group $G$
  and $\beta$ a nonempty subset of $G$. Then
  $|\alpha\beta|\ge\min\big(|\beta|+|\alpha|/2,|G|\big)$.
  In particular, if $\alpha^3\neq G$ then $|\alpha^3|\ge 2|\alpha|$.
\end{prop}
\qed

As noted in \cite{He2}  the following proposition is essentially due to
Ruzsa (see \cite{RT} and \cite{Ru}).

\begin{prop} \label{Helfgott-3-is-enough}
  Let $\alpha$ be a finite subset of a group.
  Then
  \begin{enumerate}[\indent a)]
  \item \label{item:29}
    $$
    \frac{\big|\big(\alpha\cup\alpha^{-1}\cup\{1\}\big)^3\big|}
    {|\alpha|} \le
    \left(3\frac{\big|\alpha^3\big|}
      {|\alpha|}\right)^3
    $$
  \item \label{item:30}
    If $\alpha=\alpha^{-1}$ is a symmetric set with $1\in\alpha$
    and $m\ge2$ an integer then
    $$
    \frac{\big|\alpha^m\big|}{|\alpha|}\le
    \left(\frac{\big|\alpha^3\big|}{|\alpha|}\right)^{m-2}
    $$
  \end{enumerate}
\end{prop}
\qed

As mentioned in the introduction, a result of Gowers~\cite{Gow} implies the
following.

\begin{prop} [Nikolov, Pyber \cite{NP}]
  \label{govers-trick}
  Let $G$ be a finite group
  and let $k$ denote the minimal degree of a complex representation.
  Suppose that $\alpha$, $\beta$ and $\gamma$ are subsets of $G$
  such that
  $$
  |\alpha||\beta| |\gamma| > \frac{|G|^3}{k} \;.
  $$
  Then $\alpha\beta\gamma=G$.
  In particular, if $|\alpha|>|G|/\sqrt[3]{k}$ then $\alpha^3=G$.
\end{prop}

\begin{prop} \label{minimal-complex-representation-simple}
  Let $G$ be a simple algebraic group and
  $\sigma:G\to G$ a Frobenius map.
  If $L$ is the simple group of Lie type obtained as a composition factor of
  $G^\sigma$ then the minimal degree of a complex representation of $L$ is
  at least 
  $\Frac{q_\sigma-1}2$.
  If $q_\sigma\ge20$ and $\alpha\subseteq L$ is a subset of size at least
  $q_\sigma^{\dim(G)-\Frac14}$
  then $\alpha^3=L$.
\end{prop}

\begin{proof}
  The first statement is an obvious consequence of the Landazuri-Seitz lower
  bounds (\cite{LS} cf. \cite[Table~5.3A]{KL}).
  If $q_\sigma\ge4$ then $|L|\le q_\sigma^{\dim(G)}$ (see \cite{Ca2}).
  Now the second statement follows from Proposition~\ref{govers-trick}.
\end{proof}

We are now ready to prove our main result, Theorem~\ref{simple-L}.

\begin{thm} \label{simple-final}
  For all parameters $r>0$
  there is a real $\e=\e(r)>0$
  with the following property.\\
  Let $L$ be a finite simple group of Lie type of Lie rank at most $r$
  and $\alpha\subset L$ a generating set.
  Then either $\alpha^3=L$ or 
  $$
  |\alpha^3|\ge|\alpha|^{1+\e} \;.
  $$
\end{thm}

\begin{proof}
  There is a simple adjoint algebraic group $G$ and a Frobenius
  map $\sigma:G\to G$ such that $L\le G^\sigma$,
  and there are universal bounds $I(r)$, $N(r)$ and $\Delta(r)$ such that
  $$
  \big|G^\sigma:L\big|\le I(r)
  \;,\quad
  \dim(G)\le N(r)
  \;,
  $$
  $$
  \deg(G)\le\Delta(r)
  \;,\quad
  \mult(G)\le\Delta(r)
  \;,\quad
  \inv(G)\le\Delta(r)
  \;.
  $$
  If $|\alpha|\ge q_\sigma^{\dim(G)-\Frac14}$ and $q_\sigma\ge20$
  then
  $\alpha^3=L$ by Proposition~\ref{minimal-complex-representation-simple}.
  Assume otherwise.

  Suppose first that $\alpha=\alpha^{-1}$ is symmetric with $1\in\alpha$.
  We apply Theorem~\ref{main-thm} with parameters
  $N(r)$, $\Delta(r)$, $I(r)$ and $\e'=\Frac1{4\dim(G)}$
  and obtain an integer $M=M(r)$ and a real $K=K(r)$.
  We may assume that $M\ge3$,
  and by Corollary~\ref{size-of-Lie-type}.\eqref{item:26}
  we may increase $K$ so that $|\alpha|\ge K$ implies $q_\sigma\ge20$.
  Since $G$ is simple, we have $G=H$ now.
  If
  $
  K\le|\alpha|\le q_\sigma^{\dim(G)-\Frac14}
  $
  then by Theorem~\ref{main-thm} we have
  $$
  |\alpha^M|\ge|\alpha|^{1+\Frac{1}{512N^4}} \;.
  $$
  Finally we assume $|\alpha|\le K$ and $\alpha^3\neq L$.
  By Proposition~\ref{Olson} we have
  $$
  |\alpha^3|\ge2|\alpha|\ge|\alpha|^{1+\e''}
  $$
  where $\e''=\min\left(\Frac{\log(2)}{\log(K)},\Frac{1}{512N^4}\right)$
  (which depends only on $r$).
  We obtain that in any case
  $$
  |\alpha^M|\ge|\alpha|^{1+\e''} \;.
  $$
  The theorem follows in the symmetric case from
  Proposition~\ref{Helfgott-3-is-enough}.\eqref{item:30}.

  The general case then
  follows using Proposition~\ref{Helfgott-3-is-enough}.\eqref{item:29}.
\end{proof}

In Theorem~\ref{main-thm} it is essential to assume that the centre of the
algebraic group $G$ is finite. Without this assumption the statement fails.
However, we can complement it for finite groups with possibly large centre
using the following special case of a deep result of
Nikolov and Segal (\cite[Theorem~1.7]{NS}).

\begin{prop} \label{product-of-few-commutators}
  Let $P$ be a finite perfect group generated by $d$ elements.
  Then  every element of $G$ is the product of $g(d)$  commutators where
  $g(d)=12d^3+\CO(d^2)$ depends only on $d$. 
\end{prop}

Next we will describe more precisely the Nori correspondence between
$p$-generated subgroups of $GL(n,\Fp)$ and certain closed subgroups of
$GL(n,\Fpclosed)$ and some other useful facts about perfect $p$-generated
subgroups. 

\newcommand{\Nori} {{\rm exp}}
\begin{prop} \label{Nori}
  Let $P\le GL(n,\Fp)$ be a $p$-generated subgroup.
  Then there are bounds $I=I_\Nori(n)$,
  $\Delta=\Delta_\Nori(n)$ and $K=K_\Nori(n)$
  with the following properties.
  \begin{enumerate} [\indent(a)]
  \item \label{item:31}
    There is a $\Frob_p$-invariant connected closed subgroup
    $G\le GL(n,\Fpclosed)$ such that
    $\dim(G)\le n^2$, $\deg(G)\le\Delta$, $\mult(G)\le\Delta$,
    $\inv(G)\le\Delta$
    and $P$ is a subgroup of $G(\Fp)$ of index at most $I$.
  \item \label{item:32}
    If $P$ is perfect then the degree of any complex representation is
    at least $(p-1)/2$.
  \item \label{item:33}
    If moreover $|P|\ge K$ and $\alpha\subseteq P$ is a subset of size
    $|\alpha|\ge p^{\dim(G)-\Frac14}$
    then $\alpha^3=P$.
  \end{enumerate}
\end{prop}

\begin{proof}
  We first prove~\eqref{item:31}.
  By a result of Nori~\cite{No} there is a constant $I=I_\Nori(n)$ such that
  there is a $\Frob_p$-invariant connected closed subgroup $G\le
  GL(n,\Fpclosed)$ with $P\le G(\Fp)$ of index $\big|G(\Fp):P\big|\le I$.
  Clearly $\dim(G)\le n^2$.
  By \cite[Proposition~3]{La} there is an upper bound
  $\Delta_\Nori(n)\ge\deg(G)$
  (which can also be proved easily from \cite{No} using the degree of the
  exponential map)
  and by Proposition~\ref{bounding-subgroups} 
  we can also assume that it is also an upper bound
  on the other numerical invariants $\mult(G)$ and $\inv(G)$.
  Let $\sigma:G\to G$
  denote the restriction to $G$ of the automorphism $\Frob_p:G\to G$ of
  Definition~\ref{Frobenius-def}, then $G(\Fp)=G^\sigma$ by
  Remark~\ref{Frob-q-remark}.
  
  Assume now that $P$ is perfect.
  Let $\phi:P\to GL(k,\BC)$ be a nontrivial complex representation.
  If $k<\Frac{p-1}2$ then
  by well-known results of Brauer and Feit-Thompson
  (see e.g. \cite[Theorem~14.11]{Is} and the remark after its
  proof)
  $\phi(P)$ has a normal Sylow-$p$ subgroup.
  This is impossible
  since $\phi(P)$ is also a perfect $p$-generated group.
  This proves~\eqref{item:32}.

  If $K$ is large enough then $p\ge K^{1/n^2}$ is large as well,
  hence by Proposition~\ref{twisted-Lang-Weil} we have
  $|P|\le 2p^{\dim(G)}$
  and $\alpha^3=P$ by Proposition~\ref{govers-trick}.
\end{proof}

\begin{prop} \label{Chevalley-embedding}
  Let $H\le GL(n,\Fclosed)$ be a closed subgroup.
  Then for some $n'=n'\big(n,\deg(H)\big)$ there is a homomorphism
  $\phi_H:\CN_{GL(n,\Fclosed)}(H)\to GL(n',\Fclosed)$
  of degree bounded by $n$ and $\deg(H)$
  whose kernel is $H$.
  Moreover, if $\Fclosed$ has characteristic $p$
  and $H$ is $\Frob_q$-invariant for some $p$-power $q$
  then the homomorphism $\phi_H$
  we construct is $\Frob_q$-equivariant
  (see Definition~\ref{Frobenius-def}.\eqref{item:24}).
\end{prop}

This proposition is a mild strengthening of \cite[Theorem~11.5]{Hu1},
and it is rather clear that the proof can easily be
modified to yield this version.
Since we did not find a good reference,
we reproduce here the argument.
The modified proof is based on the notion of
\emph{families of subgroups},
we recall the definition and prove some of their basic properties.

Throughout the proof the adjectives ($\Frob_q$-invariant) and
$(\Frob_q$-equivariant) appearing in parenthesis
apply only in the case when $\CH$ is $\Frob_q$-invariant.

\begin{defn}
  To simplify the notation let $G=GL(n,\Fclosed)$.  Suppose that $T$
  is an affine algebraic set and $\CH\subseteq T\times G$ is a closed
  subset.  As in \cite{Hu1}, let $K[G]$ and $K[T\times G]$ denote
  the coordinate rings of $G$ and $T\times G$ respectively.
  For each point $t\in T$ we consider the closed
  subset $\CH_t\subseteq G$ defined via the equation
  $\{t\}\times\CH_t=\CH\cap\big(\{t\}\times G\big)$.  We call $\CH$ a
  \emph{family of subgroups} if $\CH_t$ is a subgroup of
  $G$ for each $t\in T$.
  In this case we call $T$ the \emph{parameter space} and
  $\CH_t$ are the \emph{members} of the family.
  Similarly, for vectorspaces $V$ and $W$,
  a closed subset
  $\CM\subseteq T\times W$ is a \emph{family of subspaces} if each
  $\CM_t\subseteq W$ is a subspace of $W$,
  and a closed subset $\CL\subseteq T\times V$ is called
  a \emph{family of lines} if each $\CL_t\subseteq V$ is
  a line through the origin.
  A morphism from a family of subgroups $\CH$  of $GL(n,\Fclosed)$
  to another group $GL(m,\Fclosed)$
  is a \emph{family of homomorphisms}
  if the induced morphisms $\CH_t\to GL(m,\Fclosed)$ are all homomorphisms.
\end{defn}

\begin{claim} \label{finite-dimensional-invariant-subspace}
  Let $T$ be an affine algebraic set and
  $F<K[T\times G]$ a finite dimensional subspace.
  Then the smallest $G$-invariant subspace $W<K[T\times G]$
  containing $F$ is finite dimensional.
  Moreover, if $T$ and $F$ are $\Frob_q$-invariant then $W$ is also
  $\Frob_q$-invariant.
\end{claim}
\begin{proof}
  $G$ acts on $T\times G$ via the right multiplication in the second
  factor. Then $W$ is finite dimensional by
  \cite[Proposition~8.6]{Hu1},
  and the $\Frob_q$-invariance is obvious.
\end{proof}

\begin{claim} \label{subgroups-are-line-stabilisers}
  Let $\CH\subseteq T\times G$ be a family of subgroups.
  Then there is a rational representation
  $\psi:G\to GL(V)$,
  a dense open subset $U\subseteq T$ and
  a family of lines $\CL\subset U\times V$
  such that
  $$
  \CH_{t} =
  \left\{ g\in G \,\big|\, \psi(g)\CL_t=\CL_t \right\}
  $$
  for all $t\in U$.
  Moreover, if $\CH$ is $\Frob_q$-invariant 
  then our construction yields 
  $\Frob_q$-invariant $\psi$, $U$ and $\CL$.
\end{claim}

\begin{proof}
  We shall imitate \cite[proof~of~11.2]{Hu1}.
  Let $I\triangleleft K[T\times G]$ denote the ideal of $\CH$ 
  (i.e. the set of those functions vanishing on $\CH$) and
  $I_t\triangleleft K[G]$ for $t\in T$ the ideal of $\CH_t$.  Then $I$
  is generated by a ($\Frob_q$-invariant) finite dimensional subspace
  $F\le K[T\times G]$.
  By Claim~\ref{finite-dimensional-invariant-subspace} there is a finite
  dimensional $G$-invariant subspace $W<K[T\times G]$ containing $F$
  (which is also $\Frob_q$-invariant).  For each $t\in T$ the
  restriction of functions to $\{t\}\times G$ is a ring homomorphism
  $r_t:K[T\times G]\to K[G]$.

  The closed subset of $G$ corresponding to the ideal $r_t(I)$ is
  precisely $\CH_t$, but the ideal $r_t(I)$ may not be a radical
  ideal, hence it is not necessarily equal to $I_t$.
  It is folklore that there is a ($\Frob_p$-invariant)
  dense open subset $T^*\subseteq T$
  such that $r_t(I)=I_t$ for all $t\in T^*$.
  Here is a quick sketch.
  We consider the projection morphism $\pi:\CH\to T$.
  By \cite[Theorem~I.1.6]{Ko} there is a canonical open dense
  subset $T'$ such that the restriction $\pi^{-1}(T')\to T'$ is flat.
  The fibre of $\pi$ at the generic points of $T'$ are smooth varieties
  (i.e. closed subgroups), hence
  by \cite[Exercise~III/10.2]{Ha} there is a canonical open dense
  subset $T^*\subseteq T'$ such that the restriction
  $\pi^{-1}(T*)\to T*$ is smooth.
  By \cite[Theorem~III/10.2]{Ha}
  the rings $K[G]/r_t(I)$ are regular for all $t\in T^*$.
  In particular, $r_t(I)$ are radical ideals, hence $r_t(I)=I_t$
  for all $t\in T^*$.

  We set $\CM_t=W\cap r_t^{-1}(I_t)$.
  Then
  $\CM=\bigcup_t\{t\}\times\CM_t\subseteq T^*\times W$ is a family of
  subspaces, hence the function $t\to\dim(\CM_t)$ is an upper
  semi-continuous function on $T^*$.  Let $T^*=\bigcup_iT^*_i$ be the
  irreducible decomposition of $T^*$ and
  $d_i=\max_{t\in T^*_i}\dim(\CM_t)$.
  The set of points $t\in T^*_i$ which satisfy
  $\dim(\CM_t)=d_i$ form an open dense subset
  $U_i\subseteq T^*_i$. Then $U=\bigcup_iU_i$ is a ($\Frob_q$-invariant)
  open dense subset of $T$.  We set
  $V=\bigoplus_{j=0}^{\dim(W)}\bigwedge^{j}W$ and the representation
  $\psi:G\to GL(V)$ is just the natural $G$-action on $V$.
  For $t\in U_i$ we set
  $\CL_t=\bigwedge^{d_i}\CM_t\le\bigwedge^{d_i}W\le V$ and let
  $\psi_t:G\to GL\big(r_t(W)\big)$ be the natural $G$-action on $r_t(W)$.

  Then $\CL=\bigcup_{t\in U}\{t\}\times\CL_t\subseteq U\times V$ is a
  ($\Frob_q$-invariant) family of lines and for each $t\in U$ the
  stabiliser of $\CL_t$ in $\psi(G)$ is equal to the stabiliser of
  $\CM_t$ in the image of $G$ in $GL(W)$, which is in turn equal the
  stabiliser of $r_t(M_t)=I_t\cap r_t(W)$ in
  $\psi_t(G)$.  On the other hand this last stabiliser is just $\CH_t$
  by \cite[proof~of~11.2]{Hu1}.
\end{proof}

\begin{claim} \label{family-of-factor-groups}
  Let $\CH\subseteq T\times G$ be a family of subgroups.
  Then there is a family of homomorphisms
  $\phi:\CN_G(\CH_t)\to GL(n',\Fclosed)$
  for a common value of $n'$.
  In particular, there is a common upper bound on $\deg(\phi_t)$.
  Moreover, if $\CH$ is $\Frob_q$-invariant then our construction
  yields a $\Frob_q$-equivariant $\phi$
  (see Definition~\ref{Frobenius-def}.\eqref{item:24}).
\end{claim}

\begin{proof}
  We prove the claim by induction on $\dim(T)$.
  We apply Claim~\ref{subgroups-are-line-stabilisers}
  (and use its notation)
  to this family of subgroups.
  We obtain an open dense subset $U\subseteq T$.
  Then $\dim(T\setminus U)<\dim(T)$ so by  the induction hypothesis
  for each ($\Frob_q$-invariant) $t\in T\setminus U$
  there is a ($\Frob_q$-equivariant) embedding
  $\CN_G(\CH_t)/\CH_t\to GL(n'',\Fclosed)$ with a common $n''$
  and a common bound on their degrees.

  Consider any ($\Frob_q$-invariant) point $t\in U$ and
  apply \cite[proof~of~Theorem~11.5]{Hu1}
  to the subgroup $N=\CH_t$ of $\CN_G(\CH_t)$
  (which is denoted there by $G$).
  For the representation and the line at the beginning of that proof
  we may choose our $G\to GL(V)$ and $\CL_t\le V$.
  The proof then constructs a representation
  $\phi_{\CH_t}:\CN_G(\CH_t)\to GL(W)$ whose kernel is just $\CH_t$.
  Moreover, the homomorphisms $\phi_{\CH_t}$ together
  form a family of homomorphisms $G\times T\to GL(W)$,
  hence there is a common upper bound on their degrees.
  The construction is uniquely determined, so it must be
  $\Frob_q$-equivariant whenever $\CH$ and $t$ are so.
  Moreover, by construction $\dim(W)\le\dim(V)^2$,
  hence the Claim is valid with $n'=\max\big(n'',\dim(V)^2)$.
\end{proof}

\begin{proof}[Proof of Proposition~\ref{Chevalley-embedding}]
  By \cite[Section~I.3]{Ko}
  there is a canonical open subset of the
  Chow variety of the projectivisation of $G$
  which parametrises
  all the closed subgroups of $G$ of degree $\deg(H)$.
  This open subset is not neccessarily affine,
  but it is defined over \Fq, hence
  it is the union of finitely many $\Frob_q$-invariant
  affine subvarieties.
  Hence there is a  $\Frob_q$-invariant
  family of subgroups which contains
  (as members) all the closed subgroups of $G$ of degree $\deg(H)$.
  The proposition follows from Claim~\ref{family-of-factor-groups}
  applied to this family.
\end{proof}

The proofs of all the results obtained in this section concerning not
necessarily simple subgroups of $GL(n,\Fp)$ rest on the following somewhat
technical consequence of Theorem~\ref{main-thm}. This theorem complements the
results about growth of generating sets of simple groups.  It would be most
interesting to establish an appropriate analogue for subgroups of
$GL(n,\Fq)$.

\begin{thm} \label{ugly-litle-duckling}
  For all parameters $n>0$
  there is a real $\e=\e(n)>0$
  with the following property.\\
  Let $P\le GL(n,\Fp)$ be a perfect $p$-generated subgroup.
  Let $1\in\alpha\subseteq P$ be a symmetric generating set
  which projects onto each simple quotient of $P$.
  Then either $\alpha^3=P$ or 
  $$
  |\alpha^3|\ge|\alpha|^{1+\e} \;.
  $$
  Moreover, the diameter of the Cayley graph of $P$ with respect to
  $\alpha$ is at most $d(n)$ where $d(n)$ depends on $n$.
\end{thm}

\begin{proof}
  Let $l$ be the smallest integer such that $|P|\le p^{l/2}$,
  note that $l\le 2n^2$.
  We prove the first statement (concerning $\alpha^3$) by induction on $l$.
  For $l=0$ it is clear.
  We assume that $l>0$ and the statement holds for all groups of order at
  most $p^{(l-1)/2}$ and
  for all matrix sizes $n$
  with an $\e$-value $\e'(n,l)\le1$.

  We apply Proposition~\ref{Nori} to $P$ and obtain the bounds
  $I_\Nori$, $\Delta_\Nori$, $K_\Nori$ (which depend only on $n$)
  and the $\Frob_p$-invariant connected closed subgroup $G\le GL(n,\Fpclosed)$
  for which $\big|G(\Fp):P\big|\le I_\Nori$ and $\dim(G)\le n^2$.
  We shall apply Theorem~\ref{main-thm} with parameter
  $\e''=\Frac1{4\dim(G)}$ and obtain the constants
  $$
  \delta=\Frac{\e''}{128\dim(G)^3}
  \;,\quad
  M_\mainThm=M_\mainThm\big(\dim(G),\e''\big)
  \;,
  $$
  $$
  K_\mainThm=K_\mainThm\big(\dim(G),\Delta_\Nori,I_\Nori,\e''\big)
  \;.
  $$
  We shall choose later a real $K\ge\max\big(K_\mainThm,K_\Nori\big)$.
  If $|\alpha|\le K$ and $\alpha^3\neq P$ then
  $|\alpha^3|\ge2|\alpha|$ by Proposition~\ref{Olson}
  and the induction step is complete in this case with any 
  $\e\ge\log(2)/\log(K)$.
  So we may assume that $|\alpha|>K$. 
  If $|\alpha|>p^{\dim(G)-1/4}$
  then $\alpha^3=P$ by Proposition~\ref{Nori}.\eqref{item:33}.
  So we assume
  $$
  K<|\alpha|\le p^{\dim(G)-\Frac14} \;.
  $$

  Consider all $\Frob_p$-invariant connected closed normal subgroups $1\neq
  H\triangleleft G$ of degree $\deg(H)\le K_\mainThm$.
  Then by Proposition~\ref{size-of-Lie-type}.\eqref{item:27},
  for sufficiently large $K$
  either $H=G$ or $\alpha\not\subseteq H$.
  By Proposition~\ref{Chevalley-embedding}
  there is a $\Frob_p$-equivariant homomorphism
  $G\to GL(n',\Fpclosed)$
  for some common $n'=n'\big(\dim(G) ,K_\mainThm\big)$
  whose kernel is $H$.
  The elements of $\alpha$ are fixpoints of $\Frob_p$,
  so by the equivariance their images are also fixpoints of $\Frob_p$
  (see Definition~\ref{Frobenius-def}.\eqref{item:24}),
  i.e. the image set $\alpha_H$ of $\alpha$ generates a subgroup of
  $GL(n',\Fp)$ isomorphic to $P/(H\cap P)$.
  This subgroup is again perfect, $p$-generated
  and $\alpha_H$ projects onto each of its simple quotients.
  In particular, if $H\neq G$ i.e. $\alpha_H\neq\{1\}$ then
  $\big|\alpha_H\big|\ge p\ge|\alpha|^{1/n^2}$.
  We know from Proposition~\ref{twisted-Lang-Weil}
  that if $K$ is large enough then
  $|H\cap P|\ge \big|H(\Fp)\big|\big/{I_\Nori}>\sqrt{p}$
  so $\big|P/(H\cap P)\big|<|P|/\sqrt{p}\le p^{(l-1)/2}$ and
  the induction hypothesis holds
  for $\alpha_H$ and $P/(H\cap P)$ with the $\e$-value
  $\e'=\e'\big(n',l\big)\le1$.

  Suppose that we find such an $H$ different from $G$ and
  $\big|\alpha_H^3\big|\ge\big|\alpha_H\big|^{1+\e'}$.
  Then using Proposition~\ref{coset-vs-subgroup} we obtain
  $$
  \big|\alpha^{5}\big|\ge
  \big|\alpha_H^3\big|\cdot\big|\alpha^{2}\cap H\big|\ge
  \big|\alpha_H\big|^{1+\e'}\cdot\big|\alpha^2\cap H\big|\ge
  \big|\alpha\big|\cdot\big|\alpha_H\big|^{\e'} \ge
  \big|\alpha\big|^{1+\e'/n^2}
  $$
  and by Proposition~\ref{Helfgott-3-is-enough}.\eqref{item:30}
  the induction step is complete.
  So we may assume that for all such $H$ we have
  $\alpha_H^3=P/(H\cap P)$.
  It follows from Corollary~\ref{size-of-Lie-type}.\eqref{item:27}
  that if $K$ is sufficiently large then
  $$
  \big|\alpha_H^3\big|=\big|P/(H\cap P)\big|\ge
  \big|P\big|^{1-{\dim(H)}/{\dim(G)}-\delta/(2n^2)}
  \ge \big|\alpha\big|^{1-{\dim(H)}/{\dim(G)}-\delta/(2n^2)} \;.
  $$

  Suppose next that $\CZ(G)$ is finite. 
  We apply Theorem~\ref{main-thm}
  with parameters $\dim(G)$, $\Delta_\Nori$, $I_\Nori$ and
  $\e''=\Frac1{4\dim(G)}$ to the subset
  $\alpha\subset G^{\Frob_p}$.
  We obtain a $\Frob_p$-invariant
  connected closed normal subgroup $H\triangleleft G$
  such that
  $\deg(H)\le K_\mainThm$, $\dim(H)>0$ and 
  $$
  \big|\alpha^{M_\mainThm}\cap H\big| \ge
  \big|\alpha\big|^{(1+\delta)\dim(H)/\dim(G)} \;.
  $$
  If $H=G$ then
  $\big|\alpha^{M_\mainThm}\big|\ge \big|\alpha\big|^{(1+\delta)}$,
  otherwise
  $$
  \big|\alpha^{3+M_\mainThm}\big|\ge
  \big|\alpha_H^3\big|\cdot\big|\alpha^{M_\mainThm}\cap H\big|\ge
  \big|\alpha\big|^{1-\Frac{\dim(H)}{\dim(G)}-\Frac\delta{2n^2}} \cdot
  \big|\alpha\big|^{(1+\delta)\Frac{\dim(H)}{\dim(G)}}\ge
  \big|\alpha\big|^{1+\Frac{\delta}{2n^2}} \;.
  $$
  By Proposition~\ref{Helfgott-3-is-enough}.\eqref{item:30}
  the induction step is complete in this case as well.

  Finally we suppose that $\CZ(G)$ is infinite.
  In this case we consider the normal subgroup $H=\CZ(G)^0$.
  By assumption $\alpha_H^3=P/(H\cap P)$ 
  hence $\alpha^{3}$ intersects every $(H\cap P)$-coset
  in $P$.
  Hence every commutator element of $P$ is in fact the commutator of
  two elements in $\alpha^{3}$.
  It is well-known that $P$ is generated by at most
  $n^2$ elements (see \cite{Py})
  hence
  by Proposition~\ref{product-of-few-commutators}
  each element of $P$ 
  is the product of $Cn^6$ commutators for some constant $C$.
  By assumption $|\alpha|\le p^{\dim(G)-1/4}$.
  Since $|P|\ge|\alpha|>K$,
  if we choose $K$ sufficiently large then
  $|P|\ge p^{\dim(G)-1/8}$ by Proposition~\ref{twisted-Lang-Weil}.
  Therefore
  $$
  \big|\alpha^{3\cdot 4\cdot Cn^6}\big| = |P| > |\alpha|^{1+1/8\dim(G)}
  $$
  and by Proposition~\ref{Helfgott-3-is-enough}.\eqref{item:30}
  the induction step is complete in this case too.
  The first statement is proved.

  Let us apply the (now established) first statement successively to
  $\alpha,\alpha^3,\alpha^9,\dots$. We obtain by induction that
  either $\alpha^{3^i}=P$ or
  $\big|\alpha^{3^i}\big|\ge|\alpha|^{(1+\e)^i}$
  for all $i$. By assumption $|\alpha|\ge p$ and $|P|<p^{n^2}$
  hence $\alpha^{d(n)}=P$ where $d(n)$ is the smallest integer above
  $n^{2\log(3)/\log(1+\e)}$.
  That is, the diameter of the Cayley graph with respect to $\alpha$
  is at most $d(n)$.
\end{proof}

Now we prove Theorem~\ref{pre-diameter} of the Introduction.

\begin{thm}
  For all natural numbers $n$ there is an integer $M=M(n)$
  with the following property.\\
  Let $P\le GL(n,\Fp)$ be a perfect $p$-generated subgroup.
  Then the diameter of the Cayley graph of $P$ with respect to any
  symmetric generating set is at most $\big(\log|P|\big)^M$.
\end{thm}

\begin{proof}
  Let $\alpha$ be a symmetric generating set of $P$ containing $1$.
  Let $L$ be any simple quotient of $P$,
  we denote by $\tilde\alpha$ the image of $\alpha$ in $L$. 
  The Lie rank of $L$ is at most $n$
  (see \cite{FT} and \cite[Proposition~5.2.12]{KL}).
  Let $\e=\e(n)$ be as in Theorem~\ref{simple-final}.
  Applying that theorem successively to
  $\tilde\alpha,\tilde\alpha^3,\tilde\alpha^9,\dots$
  we obtain by induction that
  either $\tilde\alpha^{3^i}=L$ or
  $\big|\tilde\alpha^{3^i}\big|\ge
  |\tilde\alpha|^{(1+\e)^i}\ge
  3^{(1+\e)^i}$
  for all $i$.
  With $m=\Frac{\log\log|P|-\log\log(3)}{\log(1+\e)}$
  we obtain that
  $\big|\tilde\alpha^{3^{m}}\big|\ge|P|\ge|L|$
  hence $\alpha^{3^m}$ projects onto $L$.
  This holds for each simple quotient with the same exponent $m$.

  By Theorem~\ref{ugly-litle-duckling}
  the diameter of the Cayley graph corresponding to $\alpha^{3^m}$
  is at most $d(n)$,
  hence the diameter of the Cayley graph corresponding to $\alpha$
  is at most $3^md(n)\le\big(\log|P|\big)^{M(n)}$
  where $M(n)$ is the smallest integer above
  $\Frac{\log(3)}{\log(1+\e)}+\log\big(d(n)\big)$
 \end{proof}

We will reduce the proof of Theorem~\ref{pre-partial}
to the perfect $p$-generated case (more precisely to
Theorem~\ref{ugly-litle-duckling}) using finite group theory.

\begin{defn}
  As usual $Sol(G)$ denotes the soluble radical and $O_p(G)$ the maximal
  normal $p$-subgroup of a finite group $G$. A group is called
  \emph{quasi-simple} if it is perfect and simple modulo its centre.
  We denote by $Lie^*(p)$ the set of direct products of simple
  groups of Lie type of characteristic $p$,
  and by $Lie^{**}(p)$ the set of central products of quasi-simple
  groups of Lie type of characteristic $p$.
  If $G/Sol(G)$ is in $Lie^*(p)$ then we call $G$ a
  \emph{soluble by $Lie^*(p)$ group}.
\end{defn}

The following deep result is essentially due to Weisfeiler \cite{We}.

\begin{prop} \label{Weisfeiler}
  Let $G$ be a finite subgroup of $GL(n,\BF)$
  where $\BF$ is a field of characteristic $p>0$. Then $G$ has a normal
  subgroup $H$ of index at most $f(n)$ such that $H\ge O_p(G)$
  and $H/O_p(G)$ is the central product of an abelian $p'$-group and
  quasi-simple groups of Lie type of characteristic $p$, where the bound
  $f(n)$ depends on $n$.
\end{prop}

It was proved by Collins~\cite{Co} that for $n\ge71$ one can take
$f(n)=(n+2)!$. Remarkably a (non-effective) version of the above result was
obtained by Larsen and Pink \cite{LP} without relying on the classification of
finite simple groups.
It is clear that $H$ is a soluble by $Lie^*(p)$ subgroup.

\begin{rem} \label{commutator}
  Let $P$ be a perfect $p$-generated subgroup of $GL(n,\Fp)$.
  Using Proposition~\ref{Weisfeiler} and \cite[Lemma~3]{Ha2}
  one can easily show that every element of $P$ is the product of $g(n)$
  commutators where $g(n)$ depends on $n$.
  This could be used to replace the (rather more difficult)
  Proposition~\ref{product-of-few-commutators} in the proof of
  Theorem~\ref{ugly-litle-duckling}.
\end{rem}

The rest of this section will be devoted to proving results concerning subsets
$\alpha$ of $GL(n,\Fp)$ that satisfy $\big|\alpha^3\big|\le K|\alpha|$.
We consider the group $G=\Span\alpha$ and we will establish step by step a
close relationship between $\alpha$ (and its powers) and the structure of $G$
described in Proposition~\ref{Weisfeiler}.
Throughout the proof we need to establish several auxiliary results.

\begin{prop} \label{Schreier}
  Let $G$ be a group and $\alpha\subseteq G$ a symmetric
  generating set with $1\in\alpha$.
  If $H$ is a normal subgroup of index $t$ in $G$ then $\alpha^{2t}\cap H$
  generates $H$.
\end{prop}

\begin{proof}
  It is clear that $\alpha^{t-1}$ contains a full system of coset
  representatives $g_1,\dots,g_t$ of $G/H$.
  It is well-known (see \cite[Theorem~2.6.9]{Su}) that
  $H$ is generated by elements of the form $g_iag_j^{-1}$ where $a\in\alpha$.
\end{proof}

\begin{prop} \label{quotient}
  Let $\alpha$ be a finite subset of a group $G$ and $\tilde G = G/N$ a
  quotient of $G$. Set $\tilde\alpha=\alpha N /N$ .
  Then 
  $|\alpha^4|/|\alpha|  \ge  |\tilde\alpha ^3|/|\tilde\alpha|$.
  Moreover, if $\alpha$ is symmetric and $1\in\alpha$ then
  $\big(|\tilde\alpha ^3|/|\tilde\alpha|\big)^2 \ge
  |\tilde\alpha ^3|/|\tilde\alpha|$.
\end{prop}

\begin{proof}
  There is a coset $gN$ of $N$ such that
  $|\alpha \cap gN| \ge |\alpha|/|\tilde\alpha|$.
  We may assume that $g\in\alpha$.
  Let $\{g_i\}$ be a system of representatives of the
  cosets in $\tilde\alpha^3$ with $g_i \in \alpha^3$.
  Then the sets $g_i(\alpha \cap gN)$ are disjoint subsets of $\alpha^4$
  hence
  $\big|\alpha^4\big|\ge
  \big|\tilde\alpha ^3\big||\alpha|/|\tilde\alpha|$ as required.  
  The other inequality follows then from
  Proposition~\ref{Helfgott-3-is-enough}.\eqref{item:30}.
\end{proof}

\begin{prop} \label{onto}
  Let $H$ be a soluble by $Lie^*(p)$ subgroup of $GL(n,\Fp)$ and
  $\gamma \le H$ a symmetric generating
  set with $1 \in \gamma$. Assume that $\gamma$ satisfies
  $|\gamma^3| \le K|\gamma|$ for some $K > 2$.
  Then there is a soluble by $Lie^*(p)$ normal subgroup $S$ of $H$
  such that $\gamma^6 \cap S$ projects onto all Lie type simple quotients of
  $S$ and $\gamma$ is covered
  by $K^c$ cosets of $S$, where $c=c(n)$ depends only on $n$.
\end{prop}

\begin{proof}
  Let $H/N \cong L$ be a Lie type simple quotient of $H$ and set
  $\tilde\gamma =\gamma N /N$. The Lie rank of $L$ is at most $n$
  (see \cite{FT} and \cite[Proposition~5.2.12]{KL}).
  Now $|\tilde\gamma ^3| \le K^2|\tilde\gamma|$
  by Proposition~\ref{quotient}.
  Hence by Theorem~\ref{simple-final} we have two possibilities;
  either
  $|\tilde\gamma|\ge |\tilde\gamma ^3|/K^2=|L|/K^2$
  or
  $|\tilde\gamma| \le K^b$ where $b=b(n)$ depends only on $n$.
  Set $c=6n^2(2+nb)$.
  If $(p-1)/2 \le K^{3(2+nb)}$ then we have
  $|GL(n,\Fp)|< K^c$ (since $K>2$)
  and our statement holds for $S=1$.

  Otherwise let $H/N_j\cong L_j$  $(j=1,..,t)$
  be all the Lie type simple quotients of $H$
  (there are at most $n$ such quotients e.g. by \cite[Corollary~3.3]{LPy}).
  Let $H/N_1,H/N_2,\dots,H/N_i$ be the 
  quotients for which the second possibility holds .
  Consider the subgroup $S=N_1 \cap\dots\cap N_i$.
  It is clear that $S$ is a soluble by $Lie^*(p)$ normal subgroup and
  its Lie type simple quotients are
  $S/(S \cap N_{i+1}),..,S/(S \cap N_t)$.
  Moreover $\gamma$ is covered by at most
  $K^{nb}$ cosets of $S$.

  It remains to prove that $\gamma^6 \cap S$ projects onto, say,
  $S/(S \cap N_{i+1})$.
  Consider the quotient group $\overline H=H /(S \cap N_{i+1})$.
  The image $\overline\gamma$ of $\gamma$
  in $\overline H$ is covered by at most $K^{nb}$ cosets of
  $\overline S=S/(S \cap N_{i+1})\cong L_{i+1}$
  and we have $|\overline\gamma|\ge|\overline S|/K^2$.
  This implies that some coset of
  $\overline S$ in $\overline H$ contains
  at least $|\overline S|/K^{2+nb}$ elements of $\overline\gamma$
  and it follows that
  $\big|\overline\gamma^2 \cap \overline S\big|\ge
  |\overline S|/K^{2+nb}$.
  By Remark~\ref{minimal-complex-representation-simple}
  the minimal degree of a complex
  representation of $\overline S$ is at least $(p-1)/2 > (K^{2+nb})^3$
  hence by Proposition~\ref{govers-trick} we have
  $\big(\overline\gamma^2 \cap \overline S\big)^3 = \overline S$,
  which implies our statement.   
\end{proof}

\begin{prop} \label{coset}
  Assume that a symmetric subset $\alpha$ of a group $G$
  is covered by $x$ right
  cosets of a subgroup $H$ and $\alpha^2\cap H$ is covered by $y$ right cosets
  of a subgroup $S\le H$. Then $\alpha$ is covered by $xy$ right cosets of
  $S$.
\end{prop}

\begin{proof}
  We have $\alpha\subseteq Hg_1\cup\dots\cup Hg_x$ and
  $\alpha^2\cap H\subseteq Sh_1\cup\dots\cup Sh_y$ where the coset
  representatives $g_i$ are chosen from $\alpha$.
  If $a\in\alpha\cap Hg_i$ then by our assumptions
  $ag_i^{-1}\in Sh_j$ for some $j$, hence $a\in Sh_jg_i$.
  Therefore $\alpha\subseteq\bigcup_i\bigcup_jSh_jg_i$.
\end{proof}

\begin{prop} \label{normal}
  Let $G$ and $H$ be as in Proposition~\ref{Weisfeiler}.
  Let $\alpha$ be a symmetric set of generators of $G$ with $1\in\alpha$
  satisfying $\big|\alpha^3\big|\le K|\alpha|$
  for some $K>2$.
  Set $\gamma=\alpha^{2f(n)}\cap H$.
  \begin{enumerate}[\indent a)]
  \item \label{item:34}
    The set $\gamma$ generates $H$ and satisfies
    $\big|\gamma^3\big|\le K_0|\gamma|$
    where $K_0=K^{7f(n)}$.
  \item \label{item:35}
    Let $S$ be the subgroup constructed from $\gamma$ and $H$
    in the proof of Proposition~\ref{onto}.
    If $p\ge K_0^{b_0(n)}$ 
    (where $b_0(n)=b(n)+4$ with the same $b(n)$
    as in the proof of Proposition~\ref{onto})
    then $S$ is normal in $G$.
  \item \label{item:36}
    $\alpha$ is covered by at most $K_0^{c_0(n)}$ cosets of $S$
    (where $c_0(n)=c(n)+\log\big(f(n)\big)/\log(2)$
    with the same $c(n)$ as in Proposition~\ref{onto}).
  \item \label{item:37}
    The commutator subgroup $S'$ is an extension of a $p$-group by a
    $Lie^{**}(p)$-group. 
  \end{enumerate}
\end{prop}

\begin{proof}
  Consider $\beta=\alpha^{f(n)}$. By Proposition~\ref{Schreier}
  $\gamma=\beta^2\cap H$ generates $H$.
  Using Lemma~\ref{induce-from-growing-subgroup} and
  Proposition~\ref{Helfgott-3-is-enough}
  we see that
  $$
  \frac{\big|\gamma^3\big|}{|\gamma|}\le
  \frac{\big|\beta^6\cap H\big|}{\big|\beta^2\cap H\big|}\le
  \frac{\big|\beta^7\big|}{|\beta|}\le
  \frac{\big|\alpha^{7f(n)}\big|}{|\alpha|}\le
  K^{7f(n)}
  $$
  which proves \eqref{item:34}.
  Part \eqref{item:36} follows using Proposition~\ref{coset}.
  Part \eqref{item:37} follows from Proposition~\ref{Weisfeiler}.

  It remains to prove \eqref{item:35}.
  If $H/N_j$ are all the Lie type simple quotients of $H$
  then $N=\bigcap_jN_j$ is the soluble radical of $H$.
  Consider the quotient
  $\overline{G}=G/N$.
  The set $\overline\gamma$ generates the normal subgroup
  $\overline H\triangleleft\overline G$.
  For each $a\in\alpha$ the conjugation by
  $\overline a\in\overline{\alpha}$
  is an automorphism of $\overline H$.
  Now $\overline H$ is the direct product of nonabelian simple groups
  and an automorphism of $\overline H$ permutes these factors
  (because the direct decomposition is unique).

  If $S$ is not normal in $G$ then there is a Lie type simple
  quotient of $H$, say $H/N_1\cong L_1$ and an element $a\in\alpha$
  such that
  $\gamma$ projects onto at most $K_0^{b(n)}$ elements of $H/N_1$
  and $a^{-1}\gamma a$ projects onto at least $|L_1|/K_0^2$ elements of $H/N_1$.
  Note that $a^{-1}\gamma a=a^{-1}(\beta^2\cap H)a\subseteq\beta^4\cap H$.
  By the above we have
  $|\beta^2\cap H| = |\gamma| \le \big|\gamma^2\cap N_1\big|K_0^{b(n)}$.
  On the other hand,
  $$
  \big|\beta^8\cap H\big| \ge
  \big|(\beta^4\cap H)(\beta^2\cap H)^2\big| \ge
  \big|(a^{-1}\gamma a)(\gamma^2\cap N_1)\big| \ge
  \Frac{|L_1|}{K_0^2}\big|\gamma^2\cap N_1\big| \;.
  $$
  Therefore
  $\Frac{|\beta^8\cap H|}{|\beta^2\cap H|}\ge |L_1|/K_0^{2+b(n)}$.
  But we have
  $\Frac{|\beta^8\cap H|}{|\beta^2\cap H|}\le
  \Frac{|\beta^9|}{|\beta|}\le K^{9f(n)}<K_0^2$.
  We obtain that $|L_1|<K_0^{4+b(n)}$, a contradiction.
\end{proof}

As we saw above, a subset $\alpha$ of $GL(n,\Fp)$ with
$\big|\alpha^3\big|\le K|\alpha|$
is essentially contained in a normal subgroup $S$ of $G=\Span\alpha$
such that a small power of $\alpha$ projects onto all Lie type
simple quotients of $S$.
We proceed to show that the latter property also holds for the last term $P$
of the derived series of $S$.
Later we will prove that a small power of $\alpha$ in fact generates $P$
(see Proposition~\ref{generation}).

\begin{prop} \label{perfect}
  Let $S$ be a soluble by $Lie^*(p)$ subgroup of $GL(n,\Fp)$.
  Let $1\in\alpha$ be a symmetric subset of $S$ which projects onto all
  Lie type simple quotients of $S$.
  Let $P$ be the last term of the derived series of $S$.
  Then $P$ is a perfect soluble by $Lie^*(p)$ subgroup and
  $\alpha^c\cap P$ projects onto all Lie type simple quotients of $P$
  where $c=c(n)$ depends only on $n$.
\end{prop}

\begin{proof}
  Let $S/N_i$ be the Lie type simple quotients of $S$.
  The commutator subgroup $S'$ is clearly also a soluble by $Lie^*(p)$
  subgroup and its Lie type simple quotients are the
  $S'/(S'\cap N_i)\cong S/N_i$. We need the following.

  \begin{claim}
    $S'\cap\alpha^b$ projects onto $S'/(S'\cap N_i)$ for all $i$
    where $b=b(n)$ depends only on $n$.
  \end{claim}

  To see this fix $i$ and consider the quotient $\overline{S}=S/(S'\cap N_i)$.
  This quotient is the direct product of
  $S'/(S'\cap N_i)$ and $N_i/(S'\cap N_i)\cong S/S'$
  (since these have no common quotients).
  Take two elements $a,b\in\alpha$ which project onto noncommuting
  elements of $S/N_i$.
  The image of the commutator $[a,b]\in\alpha^4$ in $\overline{S}$ is a
  nontrivial element of $S'/(S'\cap N_i)$.
  Each element of $S'\cap N_i$ appears as the first coordinate of some element
  of the image $\overline{\alpha}$ of $\alpha$ in $\overline{S}$.
  Taking conjugates of $\overline{[a,b]}$ with these elements we obtain
  that the whole conjugacy class of $[a,b]$ in the simple group
  $S'/(S'\cap N_i)$.
  But this group has Lie rank at most $n$
  and therefore each element of $S'/(S'\cap N_i)$ is the product of at most
  $a(n)$ conjugates of an arbitrary nontrivial element where $a(n)$ depends
  only on $n$ (in fact $a(n)$ is a linear function of $n$ by \cite{LL}).
  Therefore $\alpha^{6a(n)}\cap S'$ projects onto $S'/(S'\cap N_i)$ as
  claimed.

  The length of the derived series of any subgroup of $GL(n,\Fp)$ is bounded
  in $n$ (in fact there is a logarithmic bound).
  Hence our statement follows from the Claim by an obvious induction argument.
\end{proof}

\begin{defn} \label{diagonal}
  If $L=L_1\times\dots\times L_k$ is a direct
  product of isomorphic groups, $D$ a subgroup of $L$ isomorphic to $L_1$
  which projects onto each direct factor then we call $D$
  a \emph{diagonal subgroup}.
\end{defn}

\begin{prop} \label{chain}
  Let $L=L_1\times\dots\times L_k$ be a direct product of
  $k$ nonabelian simple groups
  and $T$ a subgroup which projects onto all simple quotients of $L$.
  Then any chain of subgroups between $T$ and $L$ has length at most $k$.
\end{prop}

\begin{proof}
  Let $H$ be a subgroup of $L$
  which projects onto all simple quotients of $L$ (i.e. a subdirect product).
  Then there is a partition of the set of simple groups $L_i$
  such that the groups in any partition-class are isomorphic
  and $H$ is the direct
  product of diagonal subgroups corresponding to these partition-classes
  (see \cite[Proposition~3.3]{BS}).
  Our statement follows.
\end{proof}

\begin{prop} \label{center}
  Let $L$ be a $Lie^{**}(p)$-group and
  $T$ a subgroup which projects onto $L/\CZ(L)$.
  Then $T=L$.
\end{prop}

\begin{proof}
  We have $T\CZ(L)=L$ which implies that $T$ is a normal subgroup of $L$.
  Moreover, $L/T$ is abelian and since $L$ is perfect, we have $T=L$.
\end{proof}

\begin{prop} \label{generation}
  Let $H$ be a subgroup of $GL(n,\Fp)$,
  $S$ a soluble by $Lie^*(p)$ normal subgroup of $H$
  and $P$ the last term in the derived series of $S$.
  Assume that $P$ is an extension of a $p$-group by a $Lie^{**}(p)$-group.
  Let $1\in\gamma$ be a symmetric generating set of $H$.
  Assume that $\gamma^t\cap P$ projects onto all Lie type simple quotients of
  $P$ for some integer $t$. Then $\gamma^{t+2n+2n^2}\cap P$ generates $P$.
\end{prop}

\begin{proof}
  Set $Q_i=\Span{\gamma^i\cap P}$.
  We first show that $Q_{t+2n}$ projects onto $P/O_p(P)$.
  Since $P/O_p(P)$ is a $Lie^{**}(p)$-group,
  by Proposition~\ref{center} it is sufficient to prove that
  $Q_{t+2n}$ projects onto the central quotient of $P/O_p(P)$,
  which is exactly $P/Sol(P)$.
  Denote $P/Sol(P)$ by $\overline{P}$ and let $\overline{Q_i}$ denote the
  image of $Q_i$ in $\overline{P}$.
  We need the following.

  \begin{claim}
    If $i\ge t$ and $\overline{Q_i}\ne\overline{P}$ then
    $|Q_{i+2}|$ is strictly greater than $|Q_i|$.
  \end{claim}
  To see this, observe that $\overline{Q_i}$ projects onto all simple
  quotients of $\overline{P}$ and the only normal subgroup of $\overline{P}$
  with this property is $\overline{P}$ itself.
  By our assumptions there is an $a\in\gamma$ for which
  $\overline{Q_i}$ and its conjugate $\overline{Q_i}^{\,a}$
  are different subgroups of
  $\overline{Q_{i+2}}$. This implies the claim.
  
  As noted earlier,
  $\overline{P}$ is the direct product of at most $n$ simple groups.
  Hence by Proposition~\ref{chain}
  any chain of subgroups containing $\overline{Q_t}$ has length at most $n$.
  By the above claim $Q_{t+2n}$ projects onto $P/Sol(P)$,
  hence onto $P/O_p(P)$ as stated.
  We also need the following.
  
  \begin{claim}
    If $Q_i$ is not a normal subgroup of $H$ and $i\ge t+2n$
    then $|Q_{i+2}|\ge|Q_i|\cdot p$.
  \end{claim}
  To see this, consider as above
  an element $a\in\gamma$ which does not normalise $Q_i$.
  Then $Q_i$ and $Q_i^a$ are different subgroups of $P$ generated by subsets
  of $\gamma^{i+2}$.
  Hence $P\ge Q_{i+2}\gneq Q_i$. By our assumptions $|P:Q_i|$ is a power of
  $p$ which implies the Claim.

  Repeated applications of the Claim yield an ascending chain of subgroups
  $Q_{t+2n}\lneq Q_{t+2n+2}\lneq Q_{t+2n+4}\lneq\dots\lneq Q_{t+2n+2k}=Q\le P$
  which of course has length less than $n^2$.
  The last term $Q$ of this chain is normal in $H$ hence in $S$.
  By our assumptions all nonabelian simple composition factors of $S$ are
  among the composition factors of $Q$ (with multiplicities).
  Therefore $S/Q$ must be soluble i.e. $Q=P$.
\end{proof}

\begin{prop} \label{max-coset}
  Let $G$ be a finite group and $\alpha$ a generating set such that
  $\alpha^k$ contains the subgroup $P$.
  Then
  $$
  \frac {\max_{g\in G}|\alpha\cap gP|}{|P|} \ge
  \frac{|\alpha|}{\big|\alpha^{k+1}\big|} \;.
  $$
\end{prop}

\begin{proof}
  Let $t$ be the number of cosets of $P$ which contain elements of $\alpha$.
  Then we have $\max_g|\alpha\cap gP|\cdot t\ge|\alpha|$.
  On the other hand it is clear that
  $\big|\alpha^{k+1}\big|\ge t|P|$.
  Hence
  $$
  \frac{\big|\alpha^{k+1}\big|}{|P|} \ge
  t \ge \frac{|\alpha|}{\max_{g\in G}|\alpha\cap gP|}
  $$
  as required.
\end{proof}

Now we are ready to prove our main results concerning subsets $\alpha$
of $GL(n,\Fp)$ with $\big|\alpha^3\big|\le K|\alpha|$.

\begin{thm} \label{partial-dense}
  Let $\alpha$ be a symmetric subset of $GL(n,\Fp)$ satisfying
  $|\alpha^3|\le K|\alpha|$ for some $K\ge1$.
  Then $GL(n,\Fp)$ has two subgroups $S\ge P$,
  both normalised by $\alpha$, such that
  $P$ is perfect, $S/P$ is soluble,
  a coset of $P$ contains at least $|P|/K^{c(n)}$ 
  elements of $\alpha$ and
  $\alpha$ is covered by $K^{c(n)}$ cosets of $S$
  where $c(n)$ depends on $n$.
\end{thm}

\begin{proof}
  If $K\le 2$ then let $S$ be the subgroup generated by $\alpha$
  and $P$ the last term of the derived series of $S$.
  By Proposition~\ref{Olson} we have $\alpha^3=S$
  hence $|\alpha|\ge|S|/K$,
  which implies that some coset of $P$ contains at least $|P|/K$ elements. 
  If $K>2$ and $p<K^{7f(n)b_0(n)}$ (with the notation of
  Proposition~\ref{normal})
  then we set $S=P=\{1\}$. Now we have
  $|\alpha|<K^{7f(n)b_0(n)n^2}$ which proves our
  statement in this case.
  From now on we assume that $K>2$ and $p\ge K^{7f(n)b_0(n)}$.

  Let $S$ be as in Proposition~\ref{normal}.
  Then $\alpha$ is covered by $K^{7f(n)c_0(n)}$ cosets of $S$.
  By Proposition~\ref{onto} the set $\alpha^{12f(n)}\cap S$ projects onto all
  Lie type simple quotients of $S$.

  Let $P$ be the last term of the derived series of $S$.
  Proposition~\ref{normal}.\eqref{item:37}
  implies that $P$ is an extension of a $p$-group by a
  $Lie^{**}(p)$-group,
  in particular $P$ is a $p$-generated group.
  Let $c_1(n)$ be the constant of Proposition~\ref{perfect}
  (denoted there by $c(n)$),
  set $c_2(n)=2f(n)\big(6c(n)+2n+2n^2\big)$.
  $\alpha^{c_2(n)}\cap P$ generates $P$ and projects onto all Lie type
  simple quotients of $P$ by Proposition~\ref{perfect} and
  Proposition~\ref{generation}. 
  By Theorem~\ref{ugly-litle-duckling}
  if $c(n)\ge c_2(n)d(n)$ then
  $\alpha^{c(n)}$ contains $P$.

  Using Proposition~\ref{max-coset}
  and Proposition~\ref{Helfgott-3-is-enough}.\eqref{item:30}
  we obtain that some coset of $P$ contains at least
  $$
  \frac{|P||\alpha|}{|\alpha^{c(n)+1}|} \ge
  \frac{|P|}{K^{c(n)}}
  $$
  elements of $\alpha$.
  The proof is complete.
\end{proof}

The following is a slightly stronger version of Theorem~\ref{pre-partial}.

\begin{cor} \label{partial-coset}
  Let $\alpha$ be a symmetric subset of $GL(n,\Fp)$ satisfying
  $|\alpha^3|\le K|\alpha|$ for some $K\ge1$.
  Then $GL(n,\Fp)$ has two subgroups $S\ge P$,
  both normalised by $\alpha$, such that
  $P$ is perfect, $S/P$ is soluble,
  a coset of $P$ is contained in $\alpha^3$ and
  $\alpha$ is covered by $K^{c(n)}$ cosets of $S$
  where $c(n)$ depends on $n$.
\end{cor}

\begin{proof}
  If $K\le 2$ then $\alpha^3=\Span\alpha$ by Proposition~\ref{Olson}
  and our statement follows.
  Let $c'(n)$ the constant in Theorem~\ref{partial-dense}.
  If  $\Frac{p-1}2\le K^{3c'(n)}$ and $K>2$ then it follows that
  $|\alpha|\le K^{6c'(n)n^2}$ hence our statement holds for $S=P=1$
  with $c(n)=6c'(n)n^2$.

  We assume that $K>2$ and $ K^{3c'(n)}<\Frac{p-1}2$.  Let $S$ and $P$ be as
  in Theorem~\ref{partial-dense}.  By that theorem there is a subset $X$ of
  $P$ of size at least $|P|/K^{c'(n)}$ such that $aX\subseteq\alpha$ for some
  $a\in\alpha$.
  Now
  $$
  \alpha^3\supseteq
  aXaXaX =
  a^3(a^{-2}Xa^2)(a^{-1}Xa)X \;.
  $$
  By our assumptions and Proposition~\ref{Nori}.\eqref{item:32}
  if $k$ is the minimal degree of a complex representation of $P$
  then we have
  $\big|a^{-2}Xa^2\big|\big|a^{-1}Xa\big|\big|X\big|\ge\big|P\big|^3/k$.
  hence by Proposition~\ref{govers-trick} we have $\alpha^3\supset a^3P$ as
  required.
\end{proof}

To obtain a characterisation for symmetric subsets $\alpha$ of $GL(n,\Fp)$
satisfying $\big|\alpha^3\big|\le K|\alpha|$
with polynomially bounded constants (as in Theorem~\ref{partial-dense})
seems to be a very difficult task.
As another step towards such a characterisation
we mention the following (folklore) conjecture.

\begin{conj} \label{important}
  Let $1\in\alpha$ be a  symmetric subset of $GL(n,\Fp)$
  satisfying $\big|\alpha^3\big|\le K|\alpha|$ for some $K\ge 1$.
  Then $GL(n,\Fp)$ has two subgroups $S\triangleright P$
  such that $S/P$ is nilpotent, $P$ is contained in $\alpha^{c(n)}$
  and $\alpha$ is covered by $K^{c(n)}$ cosets of $S$
  where $c(n)$ depends on $n$.
\end{conj}

The following is well-known.

\begin{prop} \label{reduction}
  Let $S$ be a finite group and $P$ a normal subgroup with $S/P$ soluble.
  If $C$ is a minimal subgroup such that $PC=S$ then $C$ is soluble.
\end{prop}
\begin{proof}
  Let $M$ be a maximal subgroup of $C$.
  If $M$ does not contain $C\cap P$ then $(C\cap P)M=C$
  which implies $PM=PC=S$, a contradiction.
  Hence all maximal subgroups of $C$, and therefore its Frattini subgroup
  $\Phi(C)$ contain $C\cap P$.
  But $\Phi(C)$ is nilpotent, hence $C$ is soluble.
\end{proof}

Theorem~\ref{partial-dense} and Proposition~\ref{reduction}
can be used to show that if Conjecture~\ref{important} holds in the case when
$\Span\alpha$ is soluble then it holds in general.
We omit the details.
\footnote{Very recently Gill and Helfgott \cite{GH2} have proved
Conjecture~\ref{important} in the soluble case.}

\section{Linear groups over arbitrary fields}
\label{sec:linear-groups}

In this section we develop another method to show that
a certain spreading system $\alpha|G$ is
$(\e,M,K)$-spreading.
As in the proof of Theorem~\ref{main-thm},
we find an appropriate CCC-subgroup $A<G$,
but now we study the case when
$A$ has infinitely many \Span\alpha-conjugates.

We use the resulting new spreading theorem
(Theorem~\ref{non-nilpotent-spreading})
inductively to show that if $\alpha$ is a non-growing subset of
$GL(n,\BF)$, $\BF$ an arbitrary field,
then \Span\alpha is essentially contained in a virtually soluble group
(see Corollary~\ref{virtually-soluble-improved}).

Combining Corollary~\ref{virtually-soluble-improved}
with various results on finite groups
(in particular Theorem~\ref{simple-L})
we obtain Theorem~\ref{pre-tyukszem},
our main result on arbitrary finitely generated linear groups.

\begin{thm} \label{non-nilpotent-spreading}
  For all parameters $N>0$, $\Delta>0$ and $\Frac1{119N^3}>\e>0$
  there is an integer
  $M=M_\infty(N,\e)>0$
  and a real
  $K=K_\infty(N,\Delta,\e)>0$
  with the following property.\\
  Let $\alpha|G$ be an
  $(N,\Delta,K)$-bounded spreading system.
  Then either $\Span\alpha\cap G$ is virtually nilpotent
  or $\alpha|G$ is $(\e,M,K)$-spreading.
  Moreover, our construction of the subgroup of spreading is uniquely
  determined.
\end{thm}

\begin{proof}
  Using the bounds from Lemma~\ref{find-CCC} and
  Lemma~\ref{spread-via-CCC}
  we set 
  $$
  M_\findCCC=M_\findCCC\left(N,\e\right)
  \;,\quad
  M_\spreadViaCCC=M_\spreadViaCCC\left(N,\e\right)
  \;,
  $$  
  $$
  M= M_\findCCC\cdot M_\spreadViaCCC  \;,
  $$  
  $$
  K=\max\Big(\Delta,K_\findCCC(N,\Delta,\e),
  K_\spreadViaCCC(N,\Delta,\e)\Big) \;.
  $$
  Suppose that $\alpha|G$ is not $(\e,M,K)$-spreading.
  In particular, it is not
  $(N\e,M_\findCCC,K)$-spreading either,
  hence
  $$
  \mu(\alpha^{M_\findCCC},G)<\big(1+N\e\big)\mu(\alpha,G) \;.
  $$
  If $G$ is nilpotent then there is nothing to prove,
  so we assume that $G$ is non-nilpotent.
  Using Lemma~\ref{find-CCC} we obtain a CCC-subgroup
  $A\subseteq G$ containing a single maximal torus $T$ such that
  $$
  \mu\big(\alpha^{M_\findCCC},A\big) >
  \left(1-\e\cdot16N\right)\mu(\alpha,G) >
  $$
  $$
  >
  \left(1-\Eescape\right)\left(1+N\e\right)\mu(\alpha,G)
  >
   \left(1-\Eescape\right)\mu(\alpha^{M_\findCCC},G) \;.
  $$
  In particular $A$ is not normal in $G$.
  If $A$ has infinitely many \Span\alpha-conjugates then
  $\alpha^{M_\findCCC}|G$ is 
  $(\e,M_\spreadViaCCC,K)$-spreading
  by Lemma~\ref{spread-via-CCC}, a contradiction.
  So $A$ has finitely many \Span\alpha-conjugates.
  Then $T$ has finitely many \Span\alpha-conjugates,
  hence $\Span\alpha\cap\CN_G(T)$ has finite index in
  $\Span\alpha\cap G$.

  On the other hand $\CN_G(T)=\CN_G\big(\CC_G(T)\big)$,
  and $\CC_G(T)$ is a Cartan subgroup, so it is nilpotent and has
  finite index in its normaliser.
  Therefore $\CN_G(T)$ is virtually nilpotent,
  hence $\Span\alpha\cap G$ is also virtually nilpotent.
\end{proof}

Our plan is to apply Theorem~\ref{non-nilpotent-spreading}, then apply
it to the subgroup of spreading, then apply it again to the new
subgroup of spreading, and so on, until we eventually arrive to a
subgroup whose intersection with $\Span\alpha$ is virtually nilpotent.

We need the following fact:

\begin{prop}[Freiman \cite{Fr}]
  \label{Freiman-2over3}
  Let $\alpha$ be a finite subset of a group $G$.
  If $|\alpha \cdot \alpha| < \frac{3}{2} |\alpha|$,
  then $S := \alpha \cdot \alpha^{-1}$ is a finite group
  of order $|\alpha \cdot \alpha|$,
  and $\alpha \subset S \cdot x = x \cdot S$
  for some $x$ in the normaliser of $S$. 
\end{prop}

\newcommand{\nilp} {{\rm nilp}}
\begin{prop} \label{find-virtually-nilpotent-for-induction}
  For all parameters $n>0$, $d>0$
  there are integers $m=m_\nilp(n,d)>0$ and $D=D_\nilp(n,d)>0$
  with the following property. \\
  Let $G\le GL(n,\Fclosed)$ be a (possibly non-connected) closed subgroup
  and $\alpha\le G$ a finite subset such that
  $\dim(G)\ge 1$, $\deg(G)\le d$ and
  $\big|\alpha^3\big|\le \CK|\alpha|$ for some $\CK$.
  Then either $|\alpha|\le\CK^m$
  or one can find a connected closed subgroup $H\le G$
  normalised by $\alpha$
  such that $\dim(H)\ge1$, $\deg(H)\le D$ and
  $\Span\alpha\cap H$ is virtually nilpotent.
\end{prop}

\begin{proof}
  During the proof we encounter several lower bounds for $m$,
  we assume that our $m$ satisfies them all.
  Similarly, we shall establish several alternative upper bounds on
  $\deg(H)$, we set $D$ to be the maximum of these bounds.
  If $\CK<\Frac32$ then \Span\alpha is virtually cyclic 
  by Proposition~\ref{Freiman-2over3}
  and the lemma holds with $H=G^0$.
  If $\CC_{G}(\alpha)$ is infinite then
  we take $H=\CC_{G}(\alpha)^0$
  (see Fact~\ref{normaliser-centraliser-fact}).
  So we assume that $\CK\ge\Frac32$,
  $\CC_{G}(\alpha)$ is finite and
  $|\alpha|>\CK^m$.
  By Proposition~\ref{Helfgott-3-is-enough}.\eqref{item:29}
  we can assume that $\alpha$ is symmetric and $1\in\alpha$.
  We order the set $\alpha$.

  By assumption $|G:G^0|\le d$,
  hence $|\alpha^2\cap G^0|\ge\Frac{|\alpha|}{d}$.
  We set $\e=\Frac1{120n^6}$, $G_0=G^0$,
  and construct by induction a sequence of length at most $n^2$
  of connected closed subgroups 
  $G_0>G_1>G_2>\dots$ normalised by $\alpha$
  and corresponding constants $e_i$, $K_i$
  such that 
  $$
  \dim(G_i)\ge1\;,\quad
  \deg(G_i)\le K_i
  \;,\quad
  \big|\alpha^{e_i}\cap G_i\big|\ge
  \left(\Frac{|\alpha|}{d}\right)^{\dim(G_i)/n^2} \;.
  $$
  It will be clear from the construction that all of the appearing
  constants (i.e. $e_i$, $K_i$, $\Delta_i$ and $M$, see below)
  depend only on $n$ and $d$.
  We already defined $G_0$, our statement holds with $K_0=d$ and $e_0=2$
  (since closed subgroups of $GL(n,\Fclosed)$ have dimension at most $n^2$).
  Suppose that $G_i$, $K_i$ and $e_i$ are already constructed for some $i\ge0$.
  We assume that $\Span\alpha\cap G_i$ is not virtually nilpotent,
  since otherwise the lemma holds with $H=G_i$ (whose degree is
  bounded in terms of $n$ and $d$).
  According to  Proposition~\ref{bounding-subgroups}
  the numerical invariants $\deg(G_i)$, $\mult(G_i)$ and $\inv(G_i)$
  are bounded from above by a certain constant
  $\Delta_i=\Delta_i(n^2,K_i)$.
  Recall from Theorem~\ref{non-nilpotent-spreading} the constants
  $M=M_\infty(n^2,\e)$ and $K_{i+1}=K_\infty(n^2,\Delta_i,\e)$.
  We assume that $m$ is large enough so that
  $\CK^m>\left(\Frac32\right)^m>d\,(K_{i+1})^{n^2}$.
  Then the $\alpha^{e}|G_i$ are
  $(n^2,\Delta_i,K_{i+1})$-bounded spreading systems
  for all $e\ge e_i$,
  hence according to Theorem~\ref{non-nilpotent-spreading}
  they are $(\e,M,K_{i+1})$-spreading.

  Let us consider the spreading systems $\alpha^{e_iM^{j}}|G_i$
  for $j=0,1,2,\dots J-1$, where $J=2\Frac{n^2}{\e} = 240n^8$.
  Suppose now that for each $i$,
  $G_i$ itself is the subgroup of spreading obtained above
  using Theorem~\ref{non-nilpotent-spreading}.
  Then
  $\mu\big(\alpha^{e_iM^J},G_i\big) \ge (1+\e)^J\mu(\alpha,G_i)$ i.e.
  $$
  \big|\alpha^{e_iM^{J}}\cap G_i\big| \ge
  \big|\alpha^{e_i}\cap G_i\big|^{(1+\e)^J} >
  \left(\Frac{|\alpha|}{d}\right)^{J\e/n^2} =
  \left(\Frac{|\alpha|}{d}\right)^2 \ge
  |\alpha|\Frac{\CK^{m}}{d^2} \;.
  $$
  On the other hand,
  by Proposition~\ref{Helfgott-3-is-enough}.\eqref{item:30}
  we have
  $\big|\alpha^{e_iM^J}\big|\le|\alpha|\CK^{e_iM^J-2}$.
  We rule this case out by choosing 
  $m\ge e_iM^J +\Frac{\log(d^2)}{\log(3/2)}$.
  Then there is a value $j_0<J$ such that
  the corresponding subgroup of spreading
  is a proper subgroup  of $G_i$. This subgroup will be our $G_{i+1}$,
  and we set $e_{i+1}=e_iM^{J}$. We obtain
  $$
  \big|\alpha^{e_{i+1}}\cap G_{i+1}\big| \ge
  \Big|\alpha^{e_iM^{j_0+1}}\cap G_{i+1}\Big| \ge
  \Big|\alpha^{e_iM^{j_0}}\cap G_i\Big|^{\Frac{\dim(G_{i+1})}{\dim(G_i)}} \ge
  $$
  $$
  \ge
  \Big|\alpha^{e_i}\cap G_i\Big|^{\Frac{\dim(G_{i+1})}{\dim(G_i)}} \ge
  \left(\Frac{|\alpha|}{d}\right)^{
    \Frac{\dim(G_i)}{n^2}\;\Frac{\dim(G_{i+1})}{\dim(G_i)}} \ge
  \left(\Frac{|\alpha|}{d}\right)^{\Frac{\dim(G_{i+1})}{n^2}} \;,
  $$
  the induction step is complete.
  The dimensions $\dim(G_i)$ strictly decrease as $i$ grows,
  hence the induction must stop in at most $n^2$ steps.
  But the only way it can stop is to produce the required subgroup $H$.
\end{proof}

Iterating the previous lemma we obtain that a non-growing subset
$\alpha\subset GL(n,\Fclosed)$ is
covered by a few cosets of a virtually soluble group.
In the proof
we need an auxiliary subgroup $G$ in order to do
induction on $\dim(G)$. For applications the only interesting case is
$G=GL(n,\Fclosed)$, $\deg(G)=1$.

\begin{cor} \label{virtually-soluble-improved}
  Let $G\le GL(n,\Fclosed)$ be a (possibly non-connected)
  closed subgroup and $\alpha\subseteq G$ a finite subset.
  Suppose that
  $\big|\alpha^3\big|\le \CK|\alpha|$ for some $\CK$.
  Then there is a virtually soluble normal subgroup
  $\Delta\triangleleft\Span\alpha$
  and a bound $m=m\big(n,\deg(G)\big)$ such that
  the subset $\alpha$
  can be covered by $\CK^m$ cosets of $\Delta$.
\end{cor}

\begin{proof}
  During the proof we encounter several lower bounds for $m$,
  we assume that our $m$ satisfies them all.
  We prove the corollary by induction on $N=\dim(G)$.
  If $\CK<\Frac32$ then \Span\alpha is virtually cyclic 
  by Proposition~\ref{Freiman-2over3}
  and the lemma holds with $\Delta=\Span\alpha$.
  If $|\alpha|\le\CK^m$ then our statement holds with $\Delta=\{1\}$.
  So we assume that $\CK\ge\Frac32$ and $|\alpha|>\CK^m$.
  If $\dim(G)=0$ then $|\alpha|\le\deg(G)$,
  we exclude this case by choosing $m$ large enough.

  Suppose that $m\ge m_\nilp\big(n,\deg(G)\big)$.
  Applying Proposition~\ref{find-virtually-nilpotent-for-induction}
  we obtain a subgroup $H$ normalised by
  $\alpha$ such that $\Span\alpha\cap H$ is virtually nilpotent,
  $\dim(H)\ge1$, and  $\deg(H)$ is bounded in terms of $n$ and
  $\deg(G)$. 
  Consider the algebraic group $\overline{G}=\CN_G(H)/H$,
  let $\overline{\alpha}\subseteq\overline{G}$ denote the image
  of $\alpha$.
  By Proposition~\ref{quotient} we have
  $|\overline{\alpha}^3|\le\CK^2|\overline{\alpha}|$.
  By Proposition~\ref{Chevalley-embedding} and
  Fact~\ref{closed-set-constructions}.\eqref{item:15}
  there is an embedding
  $\overline{G}\le GL(n',\Fclosed)$ where $n'$ and
  $\deg(\overline{G})$ are bounded in terms of $n$, $\deg(G)$ and
  $\deg(H)$.
  Clearly $\dim(\overline{G})<\dim(G)$,
  so by the induction hypothesis
  we obtain a virtually soluble normal subgroup
  $\overline{\Delta}\triangleleft\Span{\overline{\alpha}}$
  such that
  $\overline\alpha$ is covered by $\CK^{2m(n',\deg(\overline{G}))}$
  cosets of $\overline{\Delta}$.
  We define $\Delta$ to be
  the preimage of $\overline{\Delta}$ in $\Span\alpha$.
  Then $\Delta$ is virtually soluble
  since the class of virtually soluble groups is closed under extensions
  (see e.g. \cite{KKMM}). The induction step is complete.
\end{proof}

The following consequence of well-known results is of independent interest.

\begin{lem}
  \label{Platonov}
Let $\Delta$ be a virtually soluble subgroup of
$GL(n,\Fclosed)$ and let S be the soluble radical of $\Delta$.
Then $\Delta$ has a characteristic subgroup
$\Delta_0 \ge S$ such that $\Delta_0/S$ is a direct product of simple
groups of Lie type of the same characteristic as \Fclosed and
$|\Delta/\Delta_0| \le f(n)$ 
(where $f(n)$ is as in Proposition~\ref{Weisfeiler}).
Moreover the Lie rank of the simple factors appearing in $\Delta_0/S$ is
bounded by $n$ and the number of simple factors is also at most $n$.  
\end{lem}

\begin{proof}
  If ${\rm char}(\Fclosed)=0$  this is a theorem of Platonov (see
 \cite{Wehr}).
 Assume char$(\Fclosed)=p >0$. Let $D$ be
 the Zariski closure of $\Delta$.
 Then $D^0$ is soluble
 (see \cite[Theorem~5.11]{Wehr}) and $(D^0)\Delta=D$ hence
 $\tilde\Delta=\Delta/\big(\Delta\cap D^0\big)\cong D/D^0$.
 By a result of Platonov (see \cite[Lemma~10.10]{Wehr})
 we have $D=(D^0)G$ where $G$ is some finite subgroup of $D$,
 hence $G/\big(G \cap D^0\big) \cong D/D^0$.
 Now $\tilde\Delta$ is isomorphic to a quotient of the finite group
 $G\le GL(n,\Fclosed)$  by a soluble normal subgroup.
 Therefore Proposition~\ref{Weisfeiler} implies that $\tilde\Delta$ has a
 characteristic subgroup $H$ of index at most $f(n)$ such that
 $H/Sol(\tilde\Delta)$ is in $Lie^*(p)$
 (we can take $H/Sol(\tilde\Delta)$ to be the $Lie^*(p)$ part
 of the socle of $\tilde\Delta/Sol(\tilde\Delta)$).
 Using \cite[Theorem~3.4B]{Dix} it follows that
 $H/Sol(\tilde\Delta)$ is isomorphic to a quotient
 of a finite subgroup of $GL(n,\Fpclosed)$.
 As in the proof of Proposition~\ref{onto}
 we see that the number of simple factors in $H/Sol(\tilde\Delta)$
 and their Lie ranks are bounded by $n$.
 Let $\Delta_0$ be the subgroup of $\Delta$
 which corresponds to $H$. This is a characteristic subgroup since the kernel
 of the homomorphism $\Delta\to\big(\tilde\Delta/Sol(\tilde\Delta)\big)$
 is $Sol(\Delta)$, which is characteristic in $\Delta$.
 We obtain our statement.
\end{proof}

Combining Corollary~\ref{virtually-soluble-improved} and
Lemma~\ref{Platonov}
we see that a non-growing subset $\alpha\subset GL(n,\Fclosed)$
is covered by a few cosets of a soluble by $Lie^*(p)$ normal subgroup
of \Span\alpha.
To obtain another such subgroup $\Gamma$ for which $\alpha^6Sol(\Gamma)$
contains $\Gamma$ we need a bit more work.
The following two lemmas taken together describe the structure of
a (possibly infinite) soluble by $Lie^*(p)$ linear group.

\begin{lem}
  \label{radical}
  Let $S\le GL(n,\Fclosed)$ be a soluble subgroup
  normalised by a subset $\alpha\subseteq GL(n,\Fclosed)$.
  Then there is a closed subgroup $D\le GL(n,\Fclosed)$
  containing $\alpha$ and $S$,
  and a homomorphism $\phi:D\to GL(n',\Fclosed)$
  such that $\ker(\phi)$ is soluble, contains $S$,
  and $n'$ depends only on $n$.
\end{lem}

\begin{proof}
  If $S$ is abelian then we consider the centralisers
  $A=\CC_{GL(n,\Fclosed)}(S)$ and $B=\CC_{GL(n,\Fclosed)}(A)$.
  By \cite[Theorem~6.2]{Wehr} we have homomorphisms
  $$
  \phi_1:\CN_{GL(n,\Fclosed)}(A)\to GL(n^2,\Fclosed)
  \;,\quad
  \phi_2:\CN_{GL(n,\Fclosed)}(B)\to GL(n^2,\Fclosed)
  $$
  whose kernels are precisely $A$ and $B$.
  Note that $A\cap B=\CZ(A)$ contains $S$.
  Since $\alpha$ normalises $S$, it also normalises $A$ and $B$.
  The lemma holds in this case with the following settings:
  $$
  D = \CN_{GL(n,\Fclosed)}(A)\cap\CN_{GL(n,\Fclosed)}(B) \;,
  $$
  $$
  \phi=(\phi_1,\phi_2) :
  D \longrightarrow
  GL(n^2,\Fclosed)\times GL(n^2,\Fclosed) \;\le\;
  GL(2n^2,\Fclosed) \;.
  $$
  In the general case we do induction on the derived length of $S$,
  which is bounded in terms of $n$ \cite[Theorem~3.7]{Wehr}.
  The commutator subgroup $S^*$ is normalised by the subset
  $\alpha^*=\alpha\cup S$,
  we apply to them the induction hypothesis.
  We obtain a closed subgroup $D^*\le GL(n,\Fclosed)$ containing
  $\alpha\cup S$
  and a homomorphism $\phi^*:D^*\to GL(m^*,\Fclosed)$
  such that $\ker(\phi^*)$ is soluble, contains $S^*$, and
  $m^*$ depends only on $n$.
  The image $\phi^*(S)$ is abelian
  and  it is normalised by $\phi^*(\alpha)$.
  By the above settled case there is a closed subgroup 
  $D^{**}\le GL(m^*,\Fclosed)$ containing $\phi^*(\alpha)$
  and a homomorphism $\phi^{**}:D^{**}\to GL(m^{**},\Fclosed)$
  such that $\ker(\phi^{**})$ is soluble, contains $\phi^*(S)$,
  and $m^{**}$ depends only on $m^*$, hence only on $n$.
  We set
  $$
  D={\phi^*}^{-1}(D^{**}) \;,\quad
  \phi=\phi^{**}\circ\phi^* \;,\quad
  m=m^{**} \;,
  $$
  the induction step is complete.
\end{proof}

\begin{lem}
  \label{structure}
  Let $\Lambda$ be a subgroup of $GL(n,\Fclosed)$,
  ${\rm char}(\Fclosed) = p$ and $L$ a finite normal subgroup of
  $\Lambda$ such that $L$ is in $Lie^*(p)$. Then
  $\Lambda/L\CC_\Lambda(L)\le f(n^2)$ where $f()$ is
  as in Proposition~\ref{Weisfeiler}.
\end{lem}

\begin{proof}
  By \cite[Theorem~6.2]{Wehr} $\Lambda/\CC_\Lambda(L)$ is a subgroup of
  $GL(n^2,\Fclosed)$ hence by Proposition~\ref{Weisfeiler}
  it has a soluble by $Lie^*(p)$
  normal subgroup $N$ of index at most $f(n^2)$.
  On the other hand $\Lambda/\CC_\Lambda(L)$ is isomorphic to
  a subgroup $A$ of $Aut(L)$ containing $Inn(L)\cong L$.
  It is easy to see that the socle of $A$ is $Inn(L)$.
  Therefore all soluble by $Lie^*(p)$
  normal subgroups of $A$ are actually $Lie^*(p)$ subgroups of $Inn(L)$.
  Our statement follows.
\end{proof}

We need two more auxiliary results on $Lie^*(p)$ groups.

\begin{lem}
  \label{dense-subset}
  Let $H$ be a normal subgroup of a group $G$ and
  assume that $H$ is a direct product of at most $m$
  finite simple groups of Lie type of rank at most $m$.
  Let $\alpha$ be a symmetric subset of $G$ covered by
  $x$ cosets of $H$.
  If $|\alpha|\ge|H|/y$ then $H$
  has a (possibly trivial) characteristic subgroup $N$ such that
  $N$ is contained in $\alpha^6$
  and $|H/N|\le(xy)^{Cm^2}$ for some constant $C$.
\end{lem}

\begin{proof}
  If $L$ is a simple direct factor of $H$ and
  $k=k(L)$ is the minimal degree of a non-trivial
  complex representation of $L$ then by
  Proposition~\ref{minimal-complex-representation-simple} we have
  $|L| < k^{\frac{C}3m}$ for some absolute constant $C$.
  Let $k_0 < k_1 <....$ be the different numbers $k(L)$.
  Define $H_i$ as the product of the
  direct factors $L$ for which $k(L)\ge k_i$.
  The $H_i$ are characteristic subgroups of $H$.
  By our assumptions for all indices $i$ we have
  $|\alpha^2 \cap H_i|\ge|\alpha|/x|H/H_i|\ge|H_i|/xy$.
  By Proposition~\ref{govers-trick}
  if $|\alpha^2\cap H_i| > |H_i|/(k_i)^{1/3}$
  then we have $H_i\subseteq\alpha^6$.
  Let $j$ be the smallest index for which this holds.
  By the above for all $i < j$  we have
  $k_i\le(xy)^3$ hence if $L$ is a simple constituent of $|H/H_j|$
  then  $|L| < (xy)^{Cm}$.
  Setting $N=H_j$ we obtain that $|H/N|\le (xy)^{Cm^2}$,
  as required. 
\end{proof}

\begin{lem}
  \label{direct-product}
  Let $L=L_1\times\dots\times L_m$ be a  direct product of simple
  groups of Lie type of rank at most $r$.
  Let $\alpha$ be a symmetric generating set of $L$
  which projects onto all simple quotients of $L$.
  Then $\alpha^{c(m,r)}=L$ where $c(m,r)$ depends only on $m$ and $r$.
\end{lem}

\begin{proof}
  We need the following
  \begin{claim}
    Let $x=(x_1,\dots x_t)$ be an element of a product 
    $L_1\times\dots\times L_t$ of simple groups of
    Lie type of rank at most $r$ such that all $x_i$ are non-trivial.
    Then each element of $L_1\times\dots\times L_t$
    is a product of at most $Cr$ conjugates of $x$
    for an absolute constant $C$.
  \end{claim}
  For $t=1$ this is proved in \cite{LL}
  and the general case is an obvious consequence.

  We prove the lemma by induction on $m$.
  It is clear that $\alpha^2$ has two elements
  whose first projections are
  the same, hence $\alpha^3$ contains
  a non-trivial element $a=(a_1,\dots,a_m)$ such that $a_1=1$.
  Assume that $a_{i+1},\dots,a_m$ are the projections of $a$
  different from $1$.
  By the induction hypothesis we know that $\beta=\alpha^{c(m-1,r)}$
  projects onto the quotient $L/L_1$. By the claim each element
  of $L_{i+1}\times\dots\times L_m$ is a product of at most $Cr$
  conjugates of $a$ by elements of $\beta$,
  hence this subgroup is contained in $(\alpha^3 \beta^2)^{Cr}$.
  Using again the induction hypothesis we see that $\beta$ projects
  onto $L_1\times\dots\times L_{m-1}$ hence
  $L\le\beta L_m \le (\alpha^3 \beta^3)^{Cr}$. 
  We obtain that $L\le\alpha^{3Cr(c(m-1,r)+1)}$ which completes the induction
  step.
\end{proof}

Finally we are ready to prove Theorem~\ref{pre-tyukszem}.

\begin{thm}
  \label{tyukszem}
  Let $\alpha\subseteq GL(n,\Fclosed)$ be a finite symmetric subset
  such that
  $\big|\alpha^3\big|\le \CK|\alpha|$ for some $\CK\ge\Frac32$.
  Then there are normal subgroups $S\le\Gamma$ of \Span\alpha
  and a bound $m$ depending only on $n$
  such that $\Gamma\subseteq\alpha^6S$,
  the subset $\alpha$
  can be covered by $\CK^m$ cosets of $\Gamma$,
  $S$ is soluble,
  and the quotient group $\Gamma/S$ is the product of finite simple groups of
  Lie type of the same characteristic as $\Fclosed$.
  (In particular, in characteristic $0$ we have $\Gamma=S$.) 
  Moreover, the the Lie rank of the simple factors appearing in $\Gamma/S$
  is bounded by $n$,
  and the number of factors is also at most $n$.
\end{thm}

\begin{proof}
  If ${\rm char}(\Fclosed)=0$ then our statement follows from
  Corollary~\ref{virtually-soluble-improved}
  and Lemma~\ref{Platonov}.
  Assume that ${\rm char}(\Fclosed)=p > 0$.
  Corollary~\ref{virtually-soluble-improved} and
  Lemma~\ref{Platonov}
  imply that $\Lambda=\Span\alpha$ has a normal subgroup $\Delta$
  such that $\Delta/Sol(\Delta)$ is in $Lie^*(p)$ and $\alpha$ is covered by
  $K^{a(n)}$ cosets of $\Delta$ where $a(n)$ depends on $n$.
  Moreover $\Delta/Sol(\Delta)$ is the direct
  product $L_1\times\dots\times L_t$ of at most $n$ simple groups of Lie type
  of rank at most $n$. We set $S=Sol(\Delta)$.
  The proof of our theorem reduces to the following.
  
  \begin{claim}
    The group $\Lambda$ has a normal subgroup $\Gamma$ such that $\Delta
    \ge\Gamma\ge S$, $S\alpha^6 \ge\Gamma$ and $\alpha$ is covered by $K^{m}$
    cosets of $\Gamma$.
  \end{claim}
  To prove the claim, by Lemma~\ref{radical} and Proposition~\ref{quotient}
  we might as well assume  (at the cost of enlarging $n$ and $K$)
  that $S=\{1\}$, i.e. $\Delta=L_1\times\dots\times L_t$.
  In this case Proposition~\ref{structure}
  implies that $\Lambda$ has a normal subgroup $H$
  of index at most $f(n^2)$ such that
  $H$ is the direct product of $\Delta$ and
  $C=\CC_\Lambda(\Delta)$.  Set $\gamma=\alpha^{2f(n^2)}\cap H$.
  Slightly adjusting
  the proof of Proposition~\ref{normal}.\eqref{item:34}
  we see that $\gamma$ generates $H$ and
  $|\gamma^3| \le K_0|\gamma|$ where $K_0=K^{7f(n^2)}$.

  Denote by $N_j$ the (unique) direct complement of $L_j$ in $H$.
  Using Theorem~\ref{simple-final}
  (as in the proof of Proposition~\ref{onto})
  we see that for the quotients $H/N_j\cong L_j$ we
  have two possibilities; either $\gamma^3$ projects onto $H/N_j$
  (in which case $|\gamma N_j/N_j|\ge|L_j|/K_0^2$
  by Proposition~\ref{quotient})
  or $|\gamma N_j/N_j| \le K_0 ^{b(n)}$
  where $b(n)$ depends only on $n$. Let
  $H/N_1,\dots,H/N_i$ be the quotients for which the first possibility holds
  and which also satisfy $|L_j| > K_0^{b(n)+4}$.
  
  Since $H/C$ is a direct
  product of nonabelian simple groups it follows that conjugation by $\alpha$
  permutes the simple factors, therefore it permutes the subgroups $N_j$.
  By an argument as in the proof of
  Proposition~\ref{normal}.\eqref{item:35}
  we see that the set $\big\{N_1,\dots,N_i\big\}$
  is invariant under conjugation by $\alpha$.
  Therefore $N=N_1\cap\dots\cap N_i$ and $I=N_{i+1}\cap\dots\cap N_t$ are
  normal subgroups of $\Lambda$.
  By our assumptions
  $\gamma^3$ projects onto all simple quotients of $H/N$ and $(\gamma N)/N$
  generates this group. By Lemma~\ref{direct-product}
  we see that $\gamma^{c(n)}$ projects
  onto $H/N$ where $c(n)$ depends on $n$.
  This implies $|\alpha|K^{d(n)} \ge|H/N|$
  where $d(n)=2f(n^2)c(n)$.

  The subgroup $D=I \cap \Delta= L_1\times\dots\times L_i$ is also normal in
  $\Lambda$ and we have $H/N\cong D$, hence $|\alpha|K^{d(n)} \ge|D|$.
  By our assumptions $\gamma$ projects onto at most $K_0^{n(b(n)+4)}=K^{e(n)}$
  elements of $H/I$.
  Since $\alpha^2\cap H\subseteq\gamma$, the natural isomorphism
  between $H/I$ and $\Delta/D$ implies that $\alpha^2\cap\Delta$ projects onto
  at most $K^{e(n)}$ elements of $\Delta/D$.
  Using Proposition~\ref{coset} we see that
  $\alpha$ is covered by $K^{a(n)+e(n)}$ cosets of $D$.
  Since $|\alpha|\ge|D|/K^{d(n)}$,
  Lemma~\ref{dense-subset} implies that
  $D$ has a characteristic subgroup $\Gamma$ contained
  in $\alpha^6$ such that $|D/\Gamma|\le K^{(a(n)+d(n)+e(n))Cn^2}$.
  The subgroup $\Gamma$ is normal in $\Lambda$ and $\alpha$ is covered by 
  $|D/\Gamma|K^{a(n)+e(n)}$ cosets of $\Gamma$.
  Our statement follows.
\end{proof}

Theorem~\ref{pre-tyukszem} does not hold for all $\CK\ge1$.
For example $\alpha$ could be a subgroup of $GL(n,\Fclosed)$ isomorphic to
$Alt(n)$. However the structure of subsets $\alpha$ with
$|\alpha^3|<\Frac32|\alpha|$ is completely described in
Proposition~\ref{Freiman-2over3}.

\section{Examples}
\label{sec:examples}

In this section we give some examples which show that
the constant $\e(r)$ for which
Theorem~\ref{simple-final} holds must be less than $\Frac{C}{r}$.
It will be convenient to rely on \cite[Section~3]{BKL}
in describing our examples.

\begin{ex}
  Consider the group $SL(n,q)$ where $n\ge3$ (which has Lie rank $r=n-1$).
  Let $H$ be the subgroup of all diagonal matrices, this has order
  $(q-1)^{n-1}$.
  If $N$ denotes the subgroup of all monomial matrices then $N/H\simeq S_n$
  Choose an element $s$ of $N$ projecting onto an $n$-cycle of $N/H$.
  If $e_1,\dots,e_n$ is the standard basis of $\Fq^n$,
  consider the subgroup $L_{1,2}\simeq SL(2,q)$ which fixes $e_3,\dots,e_n$.
  In \cite[Theorem~3.1]{BKL} a 3-element generating set $\{a,b,c\}$ of
  $L_{1,2}$ is chosen.
  As shown in \cite{BKL} $s,a,b$ and $c$ generate $SL(n,q)$ (moreover, the
  diameter of the corresponding Cayley graph is logarithmic).

  Now $s$ normalises the diagonal subgroup $H$
  and it is clear that $a$, $b$ and $c$ normalise a subgroup $H_0$
  of index $(q-1)^2$ in $H$ (the group of diagonal matrices fixing $e_1$ and
  $e_2$).
  Our generating $A$ set will consist of $H$, $a$, $b$, $c$ and $s$.
  We claim that 
  $$
  \big|A^3\big|\le |H|\big(3(q-1)^2+58\big)+64 \;.
  $$
  It is straightforward to see that 
  $$
  \big|A^3\big| \le
  \Big|H\{a,b,c,s\}H\Big| + 57|H| +64 \;.
  $$
  Since $s$ normalises $H$ we have $\big|HsH\big|=|H|$.
  Since $a$ (resp. $b$ and $c$) normalises $H_0$
  we have  $\big|HaH\big|\le|H|(q-1)^2$ 
  (and analogous inequalities hold for $b$ and $c$)
  which implies the claim.

  Setting $q=3$ we obtain the generating set with
  $\big|A^3\big|\le100|A|$
  mentioned in the introduction.

  Clearly, there are many ways in which the above construction can be
  extended. For example the full diagonal subgroup $H$ can be replaced by
  its characteristic subgroups isomorphic to $C_t^{n-1}$
  where $t$ divides $q-1$.
  This way e.g. we can construct large families of generating sets  of
  constant growth whenever $q$ is odd.

  It would be most interesting to find some essentially different families of
  examples of large generating sets of $SL(n,q)$ with constant growth.

  The above generating sets of ``moderate growth'' are ``dense'' subsets of the
  union of a few cosets of some subgroup. This can be avoided.
  Assume that $q=2^p$ where $p\ge n$ is an odd prime. It is well-known that
  all divisors of $q-1$ are greater that $2p+1$.
  Replace $H$ in the above construction by a subset $P\subseteq H$
  of the form 
  $\Prod^{n-1}\{g,g^2,\dots,g^n\}\subseteq\Prod^{n-1}C_{q-1}\simeq H$
  which is invariant under conjugation by the cyclic element $s$.
  Now $A=P\cup\{a,b,c,s\}$ is a generating set of size roughly $n^{n-1}$
  with $A^3$ of size roughly $n^n$.
  It is easy to see that $P$ is far from being a subgroup of $SL(n,q)$. 
\end{ex}

\section{Appendix}

In this appendix we prove rigorously the algebraic geometry facts used
in the paper.
For reference we use
\cite[Sections~I.1,~I.2,~I.7~and~II.3]{Ha},
and also \cite[Section~I.3]{Ko}. Besides that, we need
Proposition~\ref{Bezout-variant}, which is a
version of B\'ezout's theorem, stated and proved in \cite{Fu}.

Let $\Fclosed^m$ denote the $m$-dimensional affine space over the
algebraically closed field $\Fclosed$, and $\BP^m$ denote its
projective closure.
For a locally closed subset $X\subseteq\Fclosed^m$,
in this appendix $\cl{X}$ denotes (as before)
the closure of $X$ in $\Fclosed^m$,
and $\cl{X}^{\BP^m}$ denotes the closure of $X$ in $\BP^m$.
Similarily, $\deg(X)$ and $\deg(\cl{X})$ denotes the degrees in the
sense of Definition~\ref{degree}, and
$\deg_{\BP^m}(\cl{X}^{\BP^m})$ denotes the degree of the projective
variety $\cl{X}^{\BP^m}\subseteq\BP^m$ in the sense of
\cite[Section~I.7]{Ha}. Note, that both notions of degree depend not
only on the isomorphism type of $X$, but also on
the particular embedding of $X$ into the affine (or projective) space.

\begin{prop}[Fulton, see \cite{Fu}]
  \label{Bezout-variant}
  Let $P,Q$ be irreducible closed subsets of the projective space
  $\BP^m$, and let $Z_1,\dots, Z_k$ be the irreducible components of
  $P\cap Q$. Then
  $$
  \deg_{\BP^m}(P)\cdot\deg_{\BP^m}(Q)\ge
  \sum_{i=1}^k\deg_{\BP^m}(Z_i) \;.
  $$
\end{prop}

Definitions~\ref{closed-open}, \ref{irreducible},
\ref{irreducible-decomposition} and \ref{dimension}
are standard, we do not comment on them.
On the other hand, the degree is usually defined for projective
varieties, and in Definition~\ref{degree} we deal with locally closed
subsets of $\Fclosed^m$. The connection with the usual notions is
explained by the following:

\begin{prop} \label{degree-explanation}
  For a locally closed subset $X\subseteq\Fclosed^m$ we have
  $$\dim(X)=\dim(\cl{X})=\dim(\cl{X}^{\BP^m}) \;,$$
  $$\deg(X)=\deg(\cl{X})=\deg_{\BP^m}(\cl{X}^{\BP^m}) \;.$$
  Moreover, $X$ is irreducible iff $\cl{X}^{\BP^m}$ is irreducible.
\end{prop}

\begin{proof}
  The last statement follows from \cite[Ex.I.1.6]{Ha}.
  Then it is enough to prove the two equalities for irreducible $X$.
  So we assume that $X$ is irreducible.
  The equality of dimensions is \cite[Ex.I.2.7]{Ha}.
  Let $\CL$ denote the collection of affine subspaces
  $L\subseteq\Fclosed^m$ of dimension $m-\dim(X)$.
  For all members $L\in\CL$,
  the intersection $\cl{L}^{\BP^m}\cap\cl{X}^{\BP^m}$ is either infinite,
  or it has at most $\deg_{\BP^m}(\cl{X}^{\BP^m})$ points.
  Moreover, for almost all $L$ the intersection
  $\cl{L}^{\BP^m}\cap\cl{X^{\BP^m}}$ 
  have exactly $\deg_{\BP^m}(\cl{X}^{\BP^m})$ points and
  $\cl{L}^{\BP^m}$ avoids the 
  smaller dimensional boundary $\cl{X}^{\BP^m}\setminus X$.
  This proves that $\deg(X)=\deg_{\BP^m}(\cl{X^{\BP^m}})$.
  The same argument applied to $\cl{X}$ completes the proof.
\end{proof}

Remark~\ref{too-many-points} follows immediately from our definition of
$\deg(X)$, as a single point has degree $1$. 
\xxx{The now corrected}
Definition~\ref{morphisms} and Remark~\ref{category-of-closed-sets}
are standard, we do not comment on them.

\begin{proof}[Proof of Fact~\ref{dimension-degree}]
  \eqref{item:3}, \eqref{item:4} and \eqref{item:5} follows from
  Proposition~\ref{degree-explanation} and the analogous statements
  for projective varieties.
  \eqref{item:7} follows from \cite[Ex.I.1.6]{Ha}
  and from the definition of the dimension.

  Combining Proposition~\ref{degree-explanation} with
  $\cl{X\cup Y}^{\BP^m}=\cl{X}^{\BP^m}\cup\cl{Y}^{\BP^m}$,
  $\cl{X\cap Y}^{\BP^m}\subseteq\cl{X}^{\BP^m}\cap\cl{Y}^{\BP^m}$,
  $X\setminus\cl{Y}\subseteq\cl{X}^{\BP^m}\setminus\cl{Y}^{\BP^m}$
  we obtain most of  \eqref{item:6},
  with the exception of its last equality.
  Next we consider the intersection
  $$
  \big(X\times\Fclosed^m\big)\cap\big(\Fclosed^m\times Y\big) =
  X\times Y\subseteq\Fclosed^{2m} \;.
  $$
  Taking closures in $\BP^{2m}$ and applying \cite[Theorem~I.7.7]{Ha}
  we obtain the last equality of  \eqref{item:6}.

  If $X$ and $Y$ are irreducible then
  $\cl{X}\times\cl{Y}=\cl{X\times Y}$
  is irreducible by \cite[Ex.I.3.15(d)]{Ha},
  hence \eqref{item:8} follows from \cite[Ex.I.1.6]{Ha}.

  Next we introduce two invariants of closed subsets.
  If $Z\subseteq\Fclosed^m$ is a closed set with irreducible
  decomposition $Z=\bigcup_iZ_i$ then we define
  $$
  N(Z)=\sum_i(d+1)^{\dim(Z_i)}\deg(Z_i)
  \quad\text{and}\quad
  D(Z)=\sum_id^{\dim(Z_i)}\deg(Z_i) \;.
  $$
  Let $F$ be the zero set of a polynomial of degree $d$
  which does not vanish identically on $Z$.
  By Proposition~\ref{Bezout-variant} we have
  $N(Z_i\cap F)<N(Z_i)$ and $D(Z_i\cap F)\le D(Z_i)$
  whenever $Z_i\subsetneq F$,
  therefore
  $N(Z\cap F)<N(Z)$ and $D(Z\cap F)\le D(Z)$.
  To obtain $X$ we start from $\Fclosed^m$,
  and add the equations of $X$ of degree $d$ one by one,
  until their common zero locus becomes $X$.
  We obtain that $\deg(X)\le D(X)\le D(\Fclosed^m)=d^m$,
  and the invariant $N$ decreases in each step,
  i.e. we need at most $N(\Fclosed^m)=(d+1)^m$ equations.
  One direction of
  \xxx{the now corrected}
  \eqref{item:9} is proved.
  The other direction of \eqref{item:9} follows from \cite[Section~I.3]{Ko}
  (the construction of the Chow variety).  
\end{proof}

\begin{proof}  [Proof of Fact~\ref{closed-set-constructions}]
  Let $X\subseteq\Fclosed^n$ and $Y\subseteq\Fclosed^m$ denote the
  ambient spaces (see the note after Definition~\ref{closed-open}),
  and let $\pi:\Fclosed^n\times\Fclosed^m\to\Fclosed^m$ denote the
  projection to the second factor.
  Note that $\Gamma_f$ is isomorphic to $X$ (via the first
  projection), and $f(X)=\pi(\Gamma_f)$.

  We already proved
  \eqref{item:10} with the exception of the degree estimates 
  which we postpone for a while.
  
  In the proof of \eqref{item:11} we may (and do) assume that $X$ is 
  irreducible. If $\cl{f(X)}=A\cup B$ were a proper decomposition into
  closed subsets then $X=f^{-1}(A)\cup f^{-1}(B)$ would also be a proper
  decomposition, a contradiction. Hence $\cl{f(X)}$ is also irreducible. 
  By \cite[Ex.II.3.19(b)]{Ha} the subset $f(X)$ contains a dense open
  subset $U\subseteq\cl{f(X)}$. It remains to estimate $\deg(f)$.
  Let $L\subseteq\Fclosed^m$ be an affine subspace of dimension
  $m-\dim\big(\cl{f(X)}\big)$
  which intersects $U$ in exactly $\deg(U)=\deg\big(\cl{f(X)}\big)$
  points
  (see Definition~\ref{degree} and
  Fact~\ref{dimension-degree}.\eqref{item:3}).
  Then $\pi^{-1}(L)$ is an affine subspace, hence
  $
  \deg(f)=\deg(\Gamma_f)\ge\deg\big(\Gamma_f\cap\pi^{-1}(L)\big)
  $.
  But $\Gamma_f\cap\pi^{-1}(L)$ is isomorphic to
  $f^{-1}\big(L\big) = f^{-1}\big(U\cap L)$,
  hence it has at least $\deg\big(\cl{f(X)}\big)$ connected components.
  This implies that $\deg(f)\ge\deg\big(\cl{f(X)}\big)$,
  \eqref{item:11} is proved.

  Next we prove \eqref{item:12}.
  We know that $f^{-1}(T)$ is isomorphis to $\Gamma_f\cap\pi^{-1}(T)$,
  and $\pi^{-1}(T) = \Fclosed^n\times T$ have degree $\deg(T)$
  by Fact~\ref{dimension-degree}.\eqref{item:6}.
  Then $\deg\big(f^{-1}(T)\big)\le\deg(T)\deg(f)$
  by Fact~\ref{dimension-degree}.\eqref{item:6}.
  In the spacial case $T=\{y\}$ we obtain
  $\deg\big(f^{-1}(y)\big)\le\deg(f)$,
  which completes the proof of \eqref{item:12}.

  The closed complement considered in \eqref{item:13}
  is the union of a number of the locally closed subsets of
  \eqref{item:10},
  hence its degree bound follows immediately from \eqref{item:10}.
  So  \eqref{item:13} is proved.

  \cite[Ex.II.3.22(b)]{Ha} contains
  the inequality of \eqref{item:14} as well as the openness and
  denseness of $X_{\min}$.
  The difference $X\setminus X_{\min}$
  is the inverse image of the union of a number of the locally closed
  subsets of \eqref{item:10}, 
  hence its degree bound follows from \eqref{item:10} and
  \eqref{item:12}.
  This proves \eqref{item:14}.

  In \eqref{item:15},
  the graph of the restricted morphism $f\big|_S$ is
  $\Gamma_f\cap(S\times\Fclosed^m)$.
  By Fact~\ref{dimension-degree}.\eqref{item:6}
  it has degree at most
  $\deg(\Gamma_f)\deg\big(S\times\Fclosed^m\big)=\deg(f)\deg(S)$.
  Moreover, if $S$ is an irreducible component of $X$
  then the graph of $f\big|_S$ is the corresponding component of
  $\Gamma_f$.
  This proves \eqref{item:15}.
\end{proof}

\begin{proof}
  [Proof of Fact~\ref{closed-set-constructions}.\eqref{item:10},
  counting the sheep]
  First we bound the number of the parts in the partitions of $Z$.
  In the proof we partition $Z'_j$ in at most $d+1$ steps.
  In the very first step we subdivide $Z'_j$ into $(d+1)^2+2$ parts,
  and the algorithm stops in two of them.
  Suppose that $C$ is a partition class constructed before the $(l-1)$-th
  polynomial division and the algorithm did not stop in $C$.
  Before the $l$-th division we subdivide $C$ into $d+2$ parts,
  in one of them the algorithm stops, in the other $d+1$ it continues.
  Altogether we cut $Z'_j$ into at most
  $2+\sum_{l=1}^d(d+1)^{l+1}\le(d+2)^{d+1}$ pieces,
  and we repeat this cutting less than $(d+1)^{k+1}$
  \xxx{Changed!}
  times.
  Hence we obtain altogether at most $(d+2)^{(d+1)((d+1)^{k+1}-1)}$
  \xxx{Changed!}
  parts $Z_i$. Finally we cut each $Z_i$ again into at most $d+2$ parts.
  
  Let $p(t,\ux)$ and $q(t,\ux)$ be polynomials of $t$-degree at most $d$
  and $\ux$-degree at most $e$.
  We divide by the leading $t$-coefficients, then all $t$-coefficients
  are rational functions of degree at most
  (with nonstandard notation) $e/e$.
  We do polynomial division: both the quotient and the remainder have
  coefficients of degree at most $e^2/e^2$.
  We run Euclid's algorithm for $p$ and $q$.
  We do at most $d$ divisions.
  In each quotient and in each remainder the $t$-coefficients have
  degrees at most $e^{2^d}/e^{2^d}$.
  Then we multiply through with the denominators.
  
  In the proof of 
  Claim~\ref{Euclids-algorithm}.\eqref{item:16}
  we run Euclid's algorithm at most $(d+1)^{k+1}-1$
  \xxx{Changed!}
  times.
  So each polynomial we encounter (including the $P_i$)
  has $t$-degree at most $d$
  and $\ux$-degree at most $d^{((d+1)^{k+1}-1)2^d}$,
  \xxx{Changed!}
  hence their total degree is at most $d^{(d+1)^{k+1}2^d}$.
  \xxx{Changed!}
  In the proof of 
  Claim~\ref{Euclids-algorithm}.\eqref{item:17}
  each $Z_i$ is subdivided into at most $d+2$
  locally closed subsets defined via the vanishing or non-vanishing
  of several $k$-variate polynomials of degree at most
  $d^{(d+1)^{k+1}2^d}$.
  \xxx{Changed!}

  In the proof of Fact~\ref{closed-set-constructions}.\eqref{item:10} 
  we start from $\cl{f(X)}$
  (which has degree at most $\deg(f)$),
  and apply Claim~\ref{Euclids-algorithm}
  at most $\dim(X)+\deg(f)-1$ times.
  Each time we subdivide each locally closed subset
  into at most
  $\Phi\big(\Phi\big(\dots\Phi\big(\deg(f)\big)\big)\dots\big)$
  pieces and each piece is defined with equations of
  degree at most $\Phi\big(\Phi\big(\dots\Phi\big(\deg(f)\big)\big)\dots\big)$.
  At the end we obtain altogether at most $D$ locally closed parts
  and their degrees are at most $D$
  (see Fact~\ref{dimension-degree}.\eqref{item:9}).

  Finally, in Fact~\ref{closed-set-constructions}.\eqref{item:13}
  the subset in question is the union of a number of the locally closed
  subsets of \eqref{item:10},
  and the subset in Fact~\ref{closed-set-constructions}.\eqref{item:14}
  is the inverse image of such a union.
  Hence their degrees are at most $D^2$ and $D^2\deg(f)$ respectively.
\end{proof}

\vskip 20pt
\noindent
L\'aszl\'o Pyber\\
A. R\'enyi Institute of Mathematics\\
Hungarian Academy of Sciences\\
P.O. Box 127\\
H-1364 Budapest\\
pyber@renyi.hu\\
\vskip 5pt
\noindent
Endre Szab\'o\\
A. R\'enyi Institute of Mathematics\\
Hungarian Academy of Sciences\\
P.O. Box 127\\
H-1364 Budapest\\
endre@renyi.hu\\

\end{document}